\documentclass[11pt]{article}
\usepackage[margin = 1in]{geometry}

\usepackage[affil-it]{authblk}
\usepackage{amsfonts,bbm,enumerate}
\usepackage{graphicx}
\graphicspath{ {images/} }
\usepackage{epstopdf}
\usepackage{mathtools}
\usepackage{amsfonts}
\usepackage{amsthm,amsmath,amstext,amssymb}
\usepackage{subcaption}

\usepackage{comment}
\usepackage[font=footnotesize,labelfont=bf]{caption}
\usepackage{makecell}

\usepackage{url}
\usepackage{longtable}
\usepackage{mathtools}
\usepackage{multirow}
\usepackage{array}
\usepackage{tabularx}
\usepackage[table]{xcolor}
\usepackage{booktabs}

\usepackage{pgfplots}
\usepackage{pgfplotstable}
\pgfplotsset{compat=1.18}
\usetikzlibrary{arrows.meta,calc,decorations.pathreplacing,positioning}

\usepackage{hyperref}
\usepackage{natbib}
\usepackage{booktabs}
\usepackage{algorithm}
\usepackage{algpseudocode}
\usepackage{xcolor}
\usepackage[utf8]{inputenc} 
\usepackage[T1]{fontenc}

\makeatletter
\newcommand{\ostar}{\mathbin{\mathpalette\make@circled\star}}
\newcommand{\make@circled}[2]{%
  \ooalign{$\m@th#1\smallbigcirc{#1}$\cr\hidewidth$\m@th#1#2$\hidewidth\cr}%
}
\newcommand{\smallbigcirc}[1]{%
  \vcenter{\hbox{\scalebox{0.77778}{$\m@th#1\bigcirc$}}}%
}
\makeatother

\ifpdf
  \DeclareGraphicsExtensions{.eps,.pdf,.png,.jpg}
\else
  \DeclareGraphicsExtensions{.eps}
\fi
\renewcommand{\d}{\mathrm{d}}

\newcommand{\x}{{\mathbf{x}}}

\newcommand{\U}{{\mathbf{U}}}

\renewcommand{\exp}{{\rm{exp}}}

\DeclareMathOperator*{\argmin}{argmin}

\DeclareMathOperator*{\sargmin}{sargmin}

\newcommand{\indi}{{\mathbbm{1}}}

\let\narrowhat\hat
\let\hat\widehat

\def\mnar{\mathsf{MNAR}}

\newcommand{\indep}{\perp \!\!\! \perp}

\newtheorem{Theorem}{Theorem}

\newtheorem{Lemma}[Theorem]{Lemma}

\newtheorem{Corollary}[Theorem]{Corollary}

\newtheorem{Proposition}[Theorem]{Proposition}

\DeclareMathAlphabet\mathbfcal{OMS}{cmsy}{b}{n}

\hypersetup{
    colorlinks=true,
    linkcolor=red!90!black,
    filecolor=blue!95!black,      
    urlcolor=cyan,
    citecolor=blue!95!black
}

\makeatletter
\newcommand*{\rom}[1]{\expandafter\@slowromancap\romannumeral #1@}
\makeatother

\allowdisplaybreaks

\begin{document}
	\title{Adaptive confidence intervals with missing data}
	
	\author{Tianyi Ma$^*$, Kabir A. Verchand$^\circ$, Chao Gao$^\dagger$ and Richard J. Samworth$^*$}
	\affil{$^*$ Statistical Laboratory, University of Cambridge\\
    $^\circ$ Department of Data Sciences and Operations, University of Southern California\\
    $^\dagger$ Department of Statistics, University of Chicago}
	\date{}
	\maketitle

\begin{abstract}
We consider the construction of confidence intervals for population means when observations are subject to missingness. To accommodate more general missingness mechanisms than missing completely at random (MCAR), we adopt a reparametrised version of the realisable contamination model of \citet{ma2024estimation}, which is a mixture of an MCAR version and a missing not at random version of the same base distribution~$P$. We characterise the minimax length of confidence intervals that can adapt to potentially unknown parameters of the model, including the contamination fraction, for Gaussian base distributions and for nonparametric classes satisfying certain tail or symmetry assumptions. In all of these settings, we provide explicit constructions of simple, practical and finite-sample valid adaptive confidence intervals that attain the corresponding minimax rates. Finally, we provide implications of our results for causal inference. 
\end{abstract}


\begin{sloppypar}
\section{Introduction} \label{sec: introduction}

Missing values represent one of the most common challenges faced by statistical practitioners.  The classical theory of missing data \citep{rubin1976inference,little2019statistical} is mainly built on the assumptions of missing completely at random (MCAR) and missing at random (MAR).  MCAR refers to the assumption that the missingness and data generating mechanisms are independent, while MAR postulates that, roughly speaking\footnote{See \citet{seaman2013what} for a rigorous definition.}, the missingness depends only on the observed entries. The MCAR assumption is widely adopted in nonparametric or high-dimensional statistical problems, e.g. \citet{loh2012high}, \citet{belloni2017linear}, \citet{loh2018high}, \citet{follain2022high}, \citet{zhu2022high}, both to motivate methodological development and establish theoretical guarantees. On the other hand, the MAR assumption has been studied primarily in the context of parametric maximum likelihood estimation, where standard regularity conditions typically yield parametric rates of convergence \citep{takai2013asymptotic}; however, we highlight the recent work of \citet{cherief2026asymptotics} and \citet{zou2026asymptotics}, who studied nonparametric Bayesian estimation and nonparametric maximum likelihood estimation respectively under MAR. 

It has long been recognised that the MCAR assumption is often overly idealised and can be unrealistic in many practical applications \citep{mccaffrey2011missing,little2012prevention,kennedy2018evaluation,ma2024estimation}. Although the MAR assumption is weaker than MCAR, it may still be difficult to justify in practice \citep{robins1997non,malinsky2022semiparametric,ren2023multiple,cherief2025parametric}, and tractable theory and methodology beyond maximum likelihood estimation remain limited.  Another fundamental issue is that the binary nature of these assumptions means that they are brittle: theoretical guarantees under MCAR or MAR assumptions are no longer assured as soon as there is even a small violation.  Ideally, one would like to be able to quantify the extent of departures from these assumptions and their consequential effects on the performance of inferential procedures.

To address these issues, \citet{ma2024estimation} propose two contamination models to study estimation problems beyond the idealised MCAR setting. The first, called the \emph{arbitrary contamination model}, is where an MCAR version of a \emph{base distribution}~$P$ is contaminated by a missing not at random (MNAR) version of an \emph{arbitrary distribution}~$Q$, where MNAR permits arbitrary dependence between the missingness and data generating mechanisms; this can be viewed as the missing data analogue of the Huber contamination model \citep{huber1964robust}. The second, called the \emph{realisable contamination model}, occurs when an MCAR version of~$P$ is contaminated by an MNAR version of the \emph{same} distribution $P$. The realisable contamination model is flexible enough to capture any missingness mechanism, while \citet{ma2024estimation} showed that it permits faster rates of convergence for mean estimation compared to the arbitrary contamination model. Moreover, minimax rates in the realisable contamination model can converge to zero with the sample size even if the contamination fraction converges to one, whereas these rates become infinite in the arbitrary contamination model as soon as half of the data are contaminated.  

In order to describe our contributions in this work, we now introduce our notation for missing data and the statistical model we consider.
We use the symbol~$\star$ to denote missing entries and let\footnote{See Appendix~\ref{sec:properties-of-R-star} for the properties of the space~$\mathbb{R}_{\star}$, including its topology and Borel $\sigma$-algebra.} $\mathbb{R}_{\star} \coloneqq \mathbb{R} \cup \{\star\}$. Define the operator $\ostar : \mathbb{R} \times \{0,1\} \to \mathbb{R}_{\star}$ by $x\ostar\omega \coloneqq x$ if $\omega=1$ and $x\ostar\omega \coloneqq \star$ if $\omega=0$, so $x\ostar\omega$ represents what is seen by the statistician, along with $\omega$, which indicates whether~$x$ is observed or not. For a distribution~$P$ on $\mathbb{R}$, let
\begin{align}
    \mnar_P \coloneqq \Bigl \{ \mathsf{Law}\bigl(X \ostar \Omega): \, X \sim P \text{ and } \Omega \in \{0, 1\}\Bigr\}, \label{eq:mnar-def}
\end{align}
where in the definition of $\mnar_P$, we allow arbitrary dependence between $X$ and $\Omega$. Thus, $\mnar_P$ is the set of all MNAR versions of the base distribution~$P$ and is a subset of distributions on~$\mathbb{R}_{\star}$. If $X$ is independent of $\Omega$, then we say $X\ostar\Omega$ is MCAR. For $\epsilon_1,\epsilon_2\in[0,1]$ with $\epsilon_1 + \epsilon_2 \leq 1$, define
\begin{equation}
\mathcal{M}(P,\epsilon_1,\epsilon_2) \coloneqq \Bigl\{
(1-\epsilon_1-\epsilon_2)P+\epsilon_1 \delta_{\{\star\}}+\epsilon_2 Q : Q\in\mathsf{MNAR}_{P} \Bigr\},\label{eq:model-mcar-mnar}
\end{equation}
where $\delta_{\{\star\}}$ denotes a Dirac point mass on $\star$. Here, $\epsilon_1$ and $\epsilon_2$ represent the proportions of MCAR and MNAR missingness, respectively. This is a reparametrisation of the realisable contamination model: in particular, $\mathcal{M}(P,\epsilon_1,\epsilon_2)$ coincides with the one-dimensional realisable contamination model $\mathcal{R}\bigl(P,\epsilon_2,\frac{1-\epsilon_1-\epsilon_2}{1-\epsilon_2}\bigr)$ of~\citet{ma2024estimation}. Equivalently, writing $Z \coloneqq X\ostar\Omega$, we have $\mathsf{Law}(Z)\in\mathcal{M}(P,\epsilon_1,\epsilon_2)$ if and only if
\begin{align}
    1-\epsilon_1-\epsilon_2 \leq \mathbb{P}(Z\neq\star \,|\, X) \leq 1-\epsilon_1 \quad\text{almost surely;} \label{eq:realisability-ineq}
\end{align}
see \citet[Eq.~(12)]{ma2024estimation}.
Finally, we remark that since MCAR is a special case of MNAR, we have $\mathcal{M}(P,\epsilon_1,\epsilon_2) \subseteq \mathcal{M}(P,\epsilon_1',\epsilon_2')$ if $\epsilon_2'\geq \epsilon_2$ and $\epsilon_1'+\epsilon_2' \geq \epsilon_1+\epsilon_2$. 

When all model parameters are known, non-adaptive confidence intervals can be constructed by adding and subtracting high-probability error bounds from the minimax optimal point estimators of \citet{ma2024estimation}. 
In this paper, we seek to make inference about the mean of the base distribution~$P$ on~$\mathbb{R}$ under model~\eqref{eq:model-mcar-mnar} via confidence intervals that can adapt to unknown parameters. Initially, in Section~\ref{sec:gaussian}, we consider the case where the base distribution~$P$ is Gaussian and determine the minimax length of adaptive confidence intervals in four regimes according to whether the standard deviation~$\sigma$ and the contamination parameters~$(\epsilon_1,\epsilon_2)$ are known or unknown.  These ideas are further developed in the nonparametric base distribution settings of Section~\ref{sec:nonparametric}, where we characterise the minimax length of confidence intervals that can adapt to the unknown contamination parameters~$(\epsilon_1,\epsilon_2)$ over classes of base distributions satisfying one of three tail conditions or a symmetry constraint. In all of these settings, we construct simple, practical and finite-sample valid adaptive confidence intervals that attain the corresponding minimax rates; in fact, in some cases, we can show that, asymptotically, the confidence interval lengths are sharp even at the level of constants. All proofs are deferred to the appendix.

\subsection{Related work}

As mentioned above, \citet{ma2024estimation} proposed two contamination models that interpolate between the idealised MCAR (or MAR) model and general MNAR models. In particular, they derived minimax rates of mean estimation under both arbitrary and realisable contamination models, and extended these ideas to regression settings with missing responses.  See also \citet{horowitz1995identification} for related models.  \citet{cherief2025parametric} considered parametric maximum mean discrepancy estimation under missing data and showed its robustness to missingness and model misspecification; part of their analysis adopted the arbitrary contamination framework of \citet{ma2024estimation}. \citet{diakonikolas2026high} and \citet{verchand2026high} also studied the statistical and computational tradeoff in the realisable contamination model. Other works that have considered problems without MCAR or MAR assumptions include, for example, \citet{imbens2004confidence} for partial identification intervals, \citet{sell2024nonparametric} for nonparametric classification and \citet{ma2025deep} for nonparametric regression.

Another line of related work concerns adaptive confidence intervals for contaminated, but fully observed, data. \citet{luo2024adaptive} initiated this direction by considering inference for a location parameter under the Huber contamination model. 
This was subsequently extended to adaptive confidence intervals for binomial proportions \citep{cho2025robust} and linear regression coefficients \citep{xie2025confidence} under Huber contamination, while \citet{wang2026adaptive} considered inference problems under Efron's two-group model.\footnote{Here, the data are drawn from the mixture distribution $(1-\epsilon)N(\theta,\sigma^2) + \epsilon \int_{\mathbb{R}} N(\eta,\sigma^2) \,\mathrm{d}Q(\eta)$ for some unknown $\epsilon \in [0,1/2)$, $\theta \in \mathbb{R}$ and distribution $Q$ on $\mathbb{R}$; in general $\sigma$ may be known or unknown but we take $\sigma = 1$ in this discussion.} To facilitate comparison, let us focus on the case where the base distribution is $P=N(\theta,1)$ for some $\theta\in\mathbb{R}$, and view the miscoverage level as constant. Under Huber contamination with contamination proportion~$\epsilon$, \citet{luo2024adaptive} showed that the minimax length for adaptive confidence intervals is proportional to
\begin{align}
    \frac{1}{\sqrt{\log{n}}} + \frac{1}{\sqrt{\log(1/\epsilon)}} \qquad\text{for } \epsilon\leq 1/20. \label{eq:intro-huber-rate}
\end{align}
This reveals an exponential adaptation cost when there is no contamination: the minimax length of any adaptive confidence interval is at least of order $\frac{1}{\sqrt{\log n}}$ when $\epsilon=0$, whereas the standard (non-adaptive) confidence interval for a Gaussian mean has length of order $\frac{1}{\sqrt{n}}$ when $\epsilon=0$. On the other hand, \citet{wang2026adaptive} proved that the minimax length under Efron's two-group model scales as
\begin{align}
    \frac{1}{n^{1/4}} + \frac{\sqrt{\epsilon}}{\sqrt{1 \vee \log(n\epsilon^2)}} \qquad\text{for } \epsilon\leq\epsilon_{\max} < \frac{1}{2}. \label{eq:intro-efron-rate}
\end{align}
In this case, we see a polynomial adaptation cost compared with what is achievable when $\epsilon$ is known to be zero. For our realisable contamination model~\eqref{eq:model-mcar-mnar}, writing $\epsilon_+ \coloneqq \epsilon_1+\epsilon_2$, we show in Theorem~\ref{thm:gaussian-minimax-rates}\emph{(c)} that the minimax length for adaptive confidence intervals has rate
\begin{align}
    \frac{1}{\sqrt{n}} + \frac{\log(\frac{1}{1 - \epsilon_+})}{\sqrt{1 \vee \log(n\epsilon_+^2)}} \qquad\text{for } \epsilon_+ \leq 1-\frac{C}{n}, \label{eq:intro-gaussian-rate}
\end{align}
where $C>0$ depends only on the miscoverage level. Thus, if $\epsilon_+=0$, then we obtain the parametric rate $\frac{1}{\sqrt{n}}$, which is intuitive since $\epsilon_+=0$ means that there is no missingness and we are able to infer this from our data.  At first glance,~\eqref{eq:intro-gaussian-rate} may suggest that there is no adaptation cost (when the data is uncontaminated) in our realisable model.  However, in the context of missing data, `no contamination' should be interpreted as MCAR, rather than $\epsilon_+ = 0$ (which corresponds to no missingness). Thus, if we consider MCAR data with constant $\epsilon_1>0$ and $\epsilon_2=0$, then the rate in~\eqref{eq:intro-gaussian-rate} becomes $\frac{1}{\sqrt{\log n}}$, so in fact the adaptation cost under MCAR data is exponential.  On the positive side, the fact that the rate in~\eqref{eq:intro-gaussian-rate} only depends on $(\epsilon_1,\epsilon_2)$ through $\epsilon_+$ means that even when $\epsilon_1=0$ and $\epsilon_2>0$ is constant (i.e.~we have a constant proportion of MNAR contamination), the problem is no more difficult, in the sense that the rate in~\eqref{eq:intro-gaussian-rate} is still $\frac{1}{\sqrt{\log n}}$.  This is the minimax optimal non-adaptive rate for the length of a confidence interval for the mean with a constant proportion of MNAR data, so there is no adaptation cost in this case; see Table~\ref{table:gaussian-rates}.  Finally, another distinction from~\eqref{eq:intro-huber-rate} and~\eqref{eq:intro-efron-rate} is that the minimax length remains finite even if most data are contaminated in our model, as long as $\epsilon_+\leq 1-\frac{C}{n}$.

\subsection{Notation}
For an arbitrary index set $I$ and functions $f,g:I\to\mathbb{R}$, we write $f\lesssim g$ if there exists a universal constant $C>0$ such that $f(i) \leq Cg(i)$ for all $i\in I$. We write $f\gtrsim g$ if $g\lesssim f$ and write $f\asymp g$ if both $f\lesssim g$ and $f\gtrsim g$. Similarly, the notation $\lesssim_{\alpha}$, $\gtrsim_{\alpha}$ and $\asymp_{\alpha}$ indicates that the implicit constant $C$ may depend on~$\alpha$. For a measurable space $\mathcal{X}$, we write $\mathcal{P}(\mathcal{X})$ for the set of all probability measures on~$\mathcal{X}$. For a measure $\mu$ that is absolutely continuous with respect to another measure~$\pi$, we write $\frac{\mathrm{d}\mu}{\mathrm{d}\pi}$ for the Radon--Nikodym derivative. For $\theta\in\mathbb{R}$ and $\sigma>0$, let $\phi_{(\theta,\sigma)}$ denote the probability density function of $N(\theta,\sigma^2)$ and let $\Phi_{(\theta,\sigma)}$ denote the distribution function of $N(\theta,\sigma^2)$; we write $\phi\coloneqq \phi_{(0,1)}$ and $\Phi\coloneqq\Phi_{(0,1)}$. We also adopt the conventions that $\star \cdot 0 \coloneqq 0 \eqqcolon 0 \cdot \star$, that $x \cdot \star \coloneqq \star \eqqcolon \star \cdot x$ for $x\in\mathbb{R} \setminus \{0\}$ and that $\star + x \coloneqq \star$ for $x \in \mathbb{R}_\star$.

\section{The Gaussian setting} \label{sec:gaussian}

In this section, we will consider observations from the realisable Gaussian model $\mathcal{M}\bigl(N(\theta,\sigma^2), \epsilon_1, \epsilon_2\bigr)$; see~\eqref{eq:model-mcar-mnar}.  We will construct confidence intervals for $\theta$ from observations $Z_1, \ldots, Z_n \overset{\mathrm{iid}}{\sim} M \in \mathcal{M}\bigl(N(\theta,\sigma^2), \epsilon_1, \epsilon_2\bigr)$, and study the minimax length of confidence intervals that adapt to various parameters.  

Let $n\in\mathbb{N}$ and $\alpha\in(0,1)$.
For $\mathcal{S} \subseteq (0,\infty)$ and $\mathcal{E} \subseteq \{(\epsilon_1,\epsilon_2) \in [0,1]^2 : \epsilon_1 + \epsilon_2 \leq 1\}$, we define $\widehat{\mathcal{CI}}(\mathcal{S}, \mathcal{E}) \equiv \widehat{\mathcal{CI}}_{n,\alpha}(\mathcal{S}, \mathcal{E})$ to be the set of all confidence intervals for $\theta$ based on $Z_1, \ldots Z_n$, with uniform coverage at least $1-\alpha$ and that adapt to $\sigma\in\mathcal{S}$ and $(\epsilon_1,\epsilon_2) \in \mathcal{E}$, i.e.~the set of all intervals $\widehat{\mathrm{CI}}$ such that
\[
 \inf_{\substack{\theta\in\mathbb{R},\, \sigma \in \mathcal{S} \\ (\epsilon_1, \epsilon_2) \in \mathcal{E}}}\; \inf_{M\in\mathcal{M}(N(\theta,\sigma^2),\epsilon_1,\epsilon_2)}\; \mathbb{P}_M\bigl(\theta \in \widehat{\mathrm{CI}}(Z_1,\ldots,Z_n)\bigr) \geq 1 - \alpha,
\]
where the subscript in $\mathbb{P}_M$ indicates that $Z_1,\ldots,Z_n \overset{\mathrm{iid}}{\sim} M$.
When $\mathcal{S}=\{\sigma\}$ is a  singleton set, we write $\widehat{\mathcal{CI}}(\sigma,\mathcal{E}) \equiv \widehat{\mathcal{CI}}\bigl(\{\sigma\},\mathcal{E}\bigr)$, and the confidence intervals in this set are allowed to depend on $\sigma,\mathcal{E}$. Similarly, when $\mathcal{E} = \{(\epsilon_1,\epsilon_2)\}$, the confidence intervals in $\widehat{\mathcal{CI}}(\mathcal{S},\epsilon_1,\epsilon_2) \equiv \widehat{\mathcal{CI}}\bigl(\mathcal{S},\{(\epsilon_1,\epsilon_2)\}\bigr)$ are allowed to depend on $\mathcal{S},  (\epsilon_1,\epsilon_2)$. For $\sigma\in\mathcal{S}$, $(\epsilon_1,\epsilon_2) \in \mathcal{E}$ and a class of confidence intervals $\widehat{\mathcal{CI}}$, we define
\begin{align}
L_{n, \sigma, \epsilon_1, \epsilon_2}(\widehat{\mathcal{CI}}) \coloneqq \inf\Bigl\{r>0 : \inf_{\hat{\mathrm{CI}} \in \hat{\mathcal{CI}}}\; \sup_{\theta\in\mathbb{R}}\; \sup_{M \in \mathcal{M}(N(\theta,\sigma^2), \epsilon_1, \epsilon_2)} \mathbb{P}_M \bigl(|\hat{\mathrm{CI}}(Z_1,\ldots,Z_n)| \geq r\bigr) \leq \alpha \Bigr\} \label{eq:minimax-quantile-gaussian}
\end{align}
to be the \emph{minimax length}\footnote{Strictly speaking, this is the \emph{lower minimax quantile} of the length of confidence intervals; see \citet{ma2024high}. However, this agrees with the minimax quantile at all but at most countably many values of $\alpha$, and moreover, we view $\alpha$ as constant and do not aim to derive sharp dependence on this value.} of confidence intervals in $\widehat{\mathcal{CI}}$ based on $Z_1,\ldots,Z_n \overset{\mathrm{iid}}{\sim} M \in \mathcal{M}\bigl(N(\theta,\sigma^2), \epsilon_1, \epsilon_2\bigr)$. Thus, when both $\mathcal{S}$ and $\mathcal{E}$ are singleton sets, $L_{n, \sigma, \epsilon_1, \epsilon_2}\bigl(\widehat{\mathcal{CI}}(\mathcal{S},\mathcal{E})\bigr)$ has the same order as the minimax point estimation rate; when at least one of $\mathcal{S}$ and $\mathcal{E}$ is not a singleton, it represents the minimax length of the corresponding adaptive confidence intervals.

\subsection{Minimax rates} \label{sec:gaussian-minimax-rates}
The following theorem provides upper and lower bounds on the minimax length of confidence intervals for $\theta$ when different parameters are known or unknown. 

\begin{Theorem} \label{thm:gaussian-minimax-rates}
    Let $n\in\mathbb{N}$ and $\alpha\in(0,1)$. Define $\mathcal{E}_{n,\alpha} \coloneqq \{(\epsilon_1,\epsilon_2) \in [0,1]^2 : \epsilon_1+\epsilon_2 \leq 1-\frac{4\log_2(24/\alpha)}{n}\}$ and $\mathcal{E}_n' \coloneqq \{(\epsilon_1,\epsilon_2) \in [0,1]^2 : \epsilon_1+\epsilon_2 \leq 1-\frac{1}{2n}\}$. There exists a universal constant $C>0$ such that the following statements hold.
        
    \begin{itemize}
       \item[(a)] \emph{(Known $\sigma,\epsilon_1,\epsilon_2$)} If $\alpha\in[\frac{C}{n(1-\epsilon_1)}, \frac{1}{4}]$, $\sigma>0$ and $(\epsilon_1,\epsilon_2) \in \mathcal{E}_{n,\alpha}$, then
        \begin{align*}
            L_{n,\sigma,\epsilon_1,\epsilon_2} \bigl(\widehat{\mathcal{CI}}(\sigma,\epsilon_1,\epsilon_2)\bigr) \asymp_{\alpha} \sigma \biggl\{\frac{1}{\sqrt{n(1-\epsilon_1)}} + \frac{\log\bigl(\frac{1-\epsilon_1}{1-\epsilon_1-\epsilon_2}\bigr)}{\sqrt{1\vee \log\{n\epsilon_2^2/(1-\epsilon_1)\}}}\biggr\}.
        \end{align*}
        \item[(b)] \emph{(Unknown $\sigma$, known $\epsilon_1,\epsilon_2$)} If $\alpha\in[\frac{C}{n(1-\epsilon_1)}, \frac{1}{4}]$, $\sigma\in (0,\infty) \eqqcolon \mathcal{S}$ and $(\epsilon_1,\epsilon_2) \in \mathcal{E}_{n,\alpha}$, then
        \begin{align*}
            L_{n,\sigma,\epsilon_1,\epsilon_2} \bigl(\widehat{\mathcal{CI}}(\mathcal{S},\epsilon_1,\epsilon_2)\bigr) \asymp_{\alpha} \sigma \biggl\{\frac{1}{\sqrt{n(1-\epsilon_1)}} + \frac{\log\bigl(\frac{1-\epsilon_1}{1-\epsilon_1-\epsilon_2}\bigr)}{\sqrt{1\vee \log\{n\epsilon_2^2/(1-\epsilon_1)\}}}\biggr\}.
        \end{align*}
        \item[(c)] \emph{(Known $\sigma$, unknown $\epsilon_1,\epsilon_2$)} If $\alpha\in(\frac{C}{n}, \frac{1}{9})$, $\sigma>0$ and $(\epsilon_1,\epsilon_2) \in \mathcal{E}_{n,\alpha}$, then
        \begin{align*}
            L_{n,\sigma,\epsilon_1,\epsilon_2}\bigl( \widehat{\mathcal{CI}}(\sigma,\mathcal{E}_{n,\alpha}) \bigr) \lesssim_{\alpha} \sigma\biggl\{\frac{1}{\sqrt{n}} + \frac{\log\bigl(\frac{1}{1-\epsilon_1-\epsilon_2}\bigr)}{\sqrt{1\vee \log\{n(\epsilon_1+\epsilon_2)^2\}}}\biggr\} \lesssim L_{n,\sigma,\epsilon_1,\epsilon_2}\bigl( \widehat{\mathcal{CI}}(\sigma,\mathcal{E}_{n}') \bigr).
        \end{align*}
        \item[(d)] \emph{(Unknown $\sigma,\epsilon_1,\epsilon_2$)} Let $\sigma_{\max}>0$ and $\mathcal{S}\coloneqq (0,\sigma_{\max}]$. If $\alpha\in(\frac{C}{n}, \frac{1}{9})$, $\sigma\in\mathcal{S}$ and $(\epsilon_1,\epsilon_2) \in \mathcal{E}_{n,\alpha}$, then
        \begin{align*}
            &L_{n,\sigma,\epsilon_1,\epsilon_2}\bigl( \widehat{\mathcal{CI}}(\mathcal{S},\mathcal{E}_{n,\alpha}) \bigr) \lesssim_{\alpha,\sigma_{\max}} \biggl(\frac{1}{\sqrt{n}} + \epsilon_1 + \epsilon_2 \biggr) \vee \biggl(\sqrt{\log\Bigl(\frac{1}{1-\epsilon_1 - \epsilon_2}\Bigr)} - 1 \biggr)\\
            &\hspace{10cm} \lesssim_{\sigma} L_{n,\sigma,\epsilon_1,\epsilon_2}\bigl( \widehat{\mathcal{CI}}(\mathcal{S},\mathcal{E}_{n}') \bigr).
        \end{align*}
    \end{itemize}
\end{Theorem}

We provide explicit constructions of confidence intervals that achieve the length guarantees in each of the four cases (see Section~\ref{sec:gaussian-CI-construction}). In case~\emph{(c)}, there is a slight discrepancy between the parameter spaces appearing in the upper and lower bounds: the upper bound on the minimax length is established for confidence intervals that adapt to $(\epsilon_1,\epsilon_2)\in\mathcal{E}_{n,\alpha} \subseteq \mathcal{E}_n'$, whereas the lower bound holds for confidence intervals that adapt to $(\epsilon_1,\epsilon_2)\in\mathcal{E}_n'$.  The small difference in the two sets concerns the extent to which  $\epsilon_1 + \epsilon_2$ is bounded away from $1$, but both restrictions are of the form $\epsilon_1 + \epsilon_2 \leq 1 - C/n$, where $C > 0$ depends only on $\alpha$ for the upper bound and is a universal constant for the lower bound.  A similar remark applies to case~\emph{(d)}, except that the confidence intervals are additionally required to adapt to the unknown $\sigma\in(0,\sigma_{\max}]$.  The fact that the bounds in Theorem~\ref{thm:gaussian-minimax-rates} hold for contamination fractions close to $1$ is in sharp contrast with the Huber contamination model, where the minimax rate becomes infinite as soon as half of the data are contaminated \citep[e.g.][]{ma2024estimation}.

\begin{table}[ht]
\centering

\renewcommand{\arraystretch}{1.35}
\setlength{\tabcolsep}{8pt}

\newcommand{\thickrule}{\arrayrulecolor{black}\specialrule{1.4pt}{0pt}{0pt}}
\newcommand{\headrule}{\arrayrulecolor{black}\specialrule{0.7pt}{0pt}{0pt}}
\newcommand{\lightrowrule}{\arrayrulecolor{black!20}\specialrule{0.35pt}{0pt}{0pt}}

{
\begin{tabularx}{0.75\textwidth}{@{}
  >{\hsize=1.5\hsize\raggedright\arraybackslash}X
  >{\hsize=0.75\hsize\centering\arraybackslash}X
  >{\hsize=0.75\hsize\centering\arraybackslash}X
@{}}

\thickrule

\rule[-1.2em]{0pt}{3em} &
\makecell{MCAR \\ ($\epsilon_1>0$, $\epsilon_2=0$)} &
\makecell{MNAR \\ ($\epsilon_1=0$, $\epsilon_2>0$)} \\

\headrule

\rule[-1.2em]{0pt}{3em} \emph{(a)} (Known $\sigma,\epsilon_1,\epsilon_2$) & $\dfrac{1}{\sqrt{n}}$ & $\dfrac{1}{\sqrt{\log n}}$ \\
\lightrowrule

\rule[-1.2em]{0pt}{3em} \emph{(b)} (Unknown $\sigma$, known $\epsilon_1,\epsilon_2$) & $\dfrac{1}{\sqrt{n}}$ & $\dfrac{1}{\sqrt{\log n}}$ \\
\lightrowrule

\rule[-1.2em]{0pt}{3em} \emph{(c)} (Known $\sigma$, unknown $\epsilon_1,\epsilon_2$) & $\dfrac{1}{\sqrt{\log n}}$ & $\dfrac{1}{\sqrt{\log n}}$ \\
\lightrowrule

\rule[-1.2em]{0pt}{3em} \emph{(d)} (Unknown $\sigma,\epsilon_1,\epsilon_2$) & $1$ & $1$ \\

\thickrule

\end{tabularx}
}

\arrayrulecolor{black}
\caption{Comparison of the rates in the four different cases of Theorem~\ref{thm:gaussian-minimax-rates} under MCAR and MNAR data; all problem parameters except $n$ are viewed as constants in this table.}
\label{table:gaussian-rates}
\end{table}

We now discuss the implications of the rates in the four different cases of Theorem~\ref{thm:gaussian-minimax-rates}. First suppose that the true data generating mechanism is MCAR with $\epsilon_1>0$ and $\epsilon_2=0$; for simplicity, we view all problem parameters except $n$ as constants. Then the rates when $\epsilon_1,\epsilon_2$ are known (cases~\emph{(a)} and~\emph{(b)}) are parametric, whereas the rate for unknown $\epsilon_1,\epsilon_2$ is $\frac{1}{\sqrt{\log n}}$ when $\sigma$ is known (case~\emph{(c)}) and constant when $\sigma$ is also unknown (case~\emph{(d)}).  Informally then, for MCAR data, the price of adaptation to unknown $\epsilon_1,\epsilon_2$ is heavy, whereas the cost of adapting to unknown $\sigma$ is relatively mild.

Next, consider the case where the true data generating mechanism is MNAR with $\epsilon_1=0$ and $\epsilon_2 > 0$ (again regarded as a constant). In this case, the rates for cases~\emph{(a)}, \emph{(b)} and~\emph{(c)} are $\frac{1}{\sqrt{\log n}}$, while the rate for case~\emph{(d)} is still constant. Thus, for MNAR data, adaptation only incurs a cost when all of the parameters $\sigma,\epsilon_1,\epsilon_2$ are unknown. We refer the readers to Table~\ref{table:gaussian-rates} for a comparison of the different rates under MCAR and MNAR settings.

\subsection{Constructions of the confidence intervals} \label{sec:gaussian-CI-construction}

The goal of this subsection is to describe the construction of the confidence intervals that attain the upper bounds in the four cases of Theorem~\ref{thm:gaussian-minimax-rates}.  We first define the minimum and maximum observed entries as 
\[
Z_{\min} \coloneqq \min\{Z_i : i\in[n],\, Z_i \neq \star\} \quad \text{ and } \quad Z_{\max} \coloneqq \max\{Z_i : i\in[n],\, Z_i \neq \star\}.
\]
If $Z_i=\star$ for all $i\in[n]$, then we define $Z_{\min}\coloneqq -\infty$ and $Z_{\max} \coloneqq \infty$.  Next, we define the empirical fraction of observed entries
\begin{align}
\narrowhat{p}_n \coloneqq \frac{1}{n} \sum_{i=1}^{n} \mathbbm{1}_{\{Z_i \neq \star\}}. \label{eq:hat-pn}
\end{align}
Under the convention that $\indi_{\{\star\leq x\}} = \indi_{\{\star> x\}} = 0$ for all $x\in\mathbb{R}$, for $\alpha\in(0,1)$, we define the \emph{lower empirical quantile function} $Q_n^-: (0,1) \to (-\infty,\infty]$ by
\begin{subequations}
\begin{align}
    Q_n^-(p) \coloneqq 
    \begin{cases}
        \inf\bigl\{x \in \mathbb{R} : \frac{1}{n} \sum_{i=1}^n \indi_{\{Z_i \leq x\}} \geq p\bigr\} \quad&\text{if } p \leq \narrowhat{p}_n \\
        Z_{\max} &\text{otherwise},
    \end{cases}
\end{align}
and the \emph{upper empirical quantile function} $Q_n^+: (0,1) \to [-\infty,\infty)$ by
\begin{align}
    Q_n^+(p) \coloneqq 
    \begin{cases}
        \inf\bigl\{x \in \mathbb{R} : \frac{1}{n} \sum_{i=1}^n \indi_{\{Z_i > x\}} \leq p\bigr\} \quad&\text{if } p < \narrowhat{p}_n \\
        Z_{\min} &\text{otherwise}.
    \end{cases}
\end{align}
\end{subequations}

\subsubsection{Warm-up: A simplified contamination model} \label{sec:warm-up}
Consider the simplified Gaussian contamination model
\begin{equation} 
R \coloneqq (1-\epsilon)P+\epsilon Q, \quad \text{where } Q\in\mathsf{MNAR}_P,\,P=N(\theta,\sigma^2). \label{eq:model-mnar}
\end{equation}
This corresponds to our model $\mathcal{M}\bigl(N(\theta,\sigma^2),0,\epsilon\bigr)$.  The key relation is that the distribution $R \in  \mathcal{P}(\mathbb{R}_{\star})$ admits the representation in~\eqref{eq:model-mnar} if and only if
\begin{align} \label{ineq:sandwich-q-1}
1 - \epsilon \leq \frac{\mathrm{d}R}{\mathrm{d}P_{\star}}(x) \leq 1, 
\end{align}
for $P$-almost all $x\in\mathbb{R}$, where $P_\star$ is the extension of $P$ to $\mathbb{R}_\star$; see Appendix~\ref{sec:properties-of-R-star} for further details.

\begin{figure}[ht]
    \centering
    \begin{tikzpicture}[
    x=0.9cm,y=0.9cm,
    >=Latex,
    font=\small,
    every node/.style={inner sep=2pt}
]

\def\xmin{0.4}
\def\xmax{12.0}
\def\yz{3.0}
\def\ytheta{0}
\def\qminus{3.3}   
\def\qplus{8.5}    
\def\L{5.0}        
\def\U{6.8}        
\def\thetatrue{5.9}

\fill[blue!10] (\xmin,\yz-0.2) rectangle (\qminus,\yz+0.2);
\fill[orange!12] (\qplus,\yz-0.2) rectangle (\xmax,\yz+0.2);

\draw[->,semithick] (\xmin,\yz) -- (\xmax,\yz);
\node[above=7pt] at (6.2,\yz) {Observations $Z_1,\dots,Z_n$};

\foreach \x in {0.8,1.2,1.8,2.3,2.8,3.3,4.0,4.6,5.1,5.8,6.4,7.0,7.5,8.0,8.5,9.1,9.8,10.7,11.3}
  \fill[black] (\x,\yz) circle (1.8pt);

\draw[thick,blue!70!black] (\qminus,\yz-0.28) -- (\qminus,\yz+0.28);
\draw[thick,orange!80!black] (\qplus,\yz-0.28) -- (\qplus,\yz+0.28);

\node[below=8pt,anchor=north east,align=right,text=blue!70!black,font=\scriptsize]
  at (3.14,\yz)
  {$Q_n^{-}(q_t)$\\empirical CDF};
\node[below=8pt,anchor=north west,align=left,text=orange!80!black,font=\scriptsize]
  at (8.66,\yz)
  {$Q_n^{+}(q_t)$\\empirical survival function};

\draw[densely dotted,blue!70!black] (\qminus,\yz-0.28) -- (\qminus,1.60);
\draw[densely dotted,orange!80!black] (\qplus,\yz-0.28) -- (\qplus,1.08);

\draw[->,thick,orange!80!black] (\qplus,1.08) -- (\L,1.08)
  node[midway,above] {$-\sigma t$};
\draw[->,thick,blue!70!black] (\qminus,1.60) -- (\U,1.60)
  node[pos=0.70,above] {$+\sigma t$};

\draw[->,semithick] (\xmin,\ytheta) -- (\xmax,\ytheta);
\draw[very thick] (\L,\ytheta) -- (\U,\ytheta);

\draw[thick] (\L,\ytheta-0.25) -- (\L,\ytheta+0.25);
\draw[thick] (\U,\ytheta-0.25) -- (\U,\ytheta+0.25);
\draw[densely dotted,orange!80!black] (\L,1.08) -- (\L,\ytheta+0.24);
\draw[densely dotted,blue!70!black] (\U,1.60) -- (\U,\ytheta+0.24);

\node[below=8pt,anchor=north east,align=right,font=\scriptsize,xshift=-5pt] at (\L,\ytheta)
{$Q_n^+(q_t)-\sigma t$\\(lower endpoint)};
\node[below=8pt,anchor=north west,align=left,font=\scriptsize,xshift=5pt] at (\U,\ytheta)
{$Q_n^{-}(q_t)+\sigma t$\\(upper endpoint)};

\draw[thick,dashed] (\thetatrue+0.3,\ytheta-0.35) -- (\thetatrue+0.3,\ytheta+0.4);
\node[above=4pt] at (\thetatrue+0.3,\ytheta + 0.3) {$\theta$};

\draw[decorate,decoration={brace,mirror,amplitude=5pt},thick]
  (\L,\ytheta-0.9) -- (\U,\ytheta-0.9)
  node[midway,below=8pt] {$\widehat{\mathrm{CI}}_{n,\sigma,\alpha}^{(1)}(t)$};

\end{tikzpicture}
    \caption{An illustration of $\widehat{\mathrm{CI}}_{n,\sigma,\alpha}^{(1)}(t)$ in~\eqref{eq:GST-warmup}, where $q_t \coloneqq 1 - \Phi(t) + \sqrt{\log(4/\alpha)/(2n)}$.}
    \label{fig:gaussian-TS}
\end{figure}
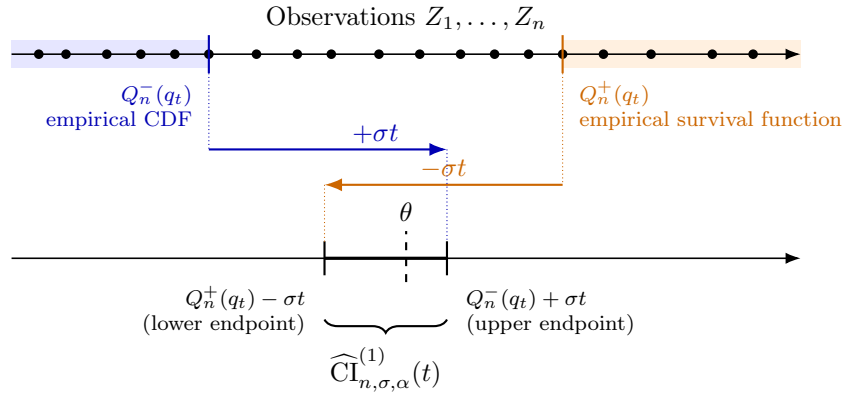

First, assume that $\sigma$ and $\epsilon$ are both known. We define a family of confidence intervals $\{\widehat{\mathrm{CI}}_{n,\sigma,\alpha}^{(1)}(t):t \geq 0\}$ for $\theta$, where 
\begin{align}
\widehat{\mathrm{CI}}_{n,\sigma,\alpha}^{(1)}(t) \coloneqq \Biggl[Q_n^+\biggl(1-\Phi(t) + \sqrt{\frac{\log(4/\alpha)}{2n}}\biggr)-\sigma t,\, Q_n^-\biggl(1-\Phi(t)+\sqrt{\frac{\log(4/\alpha)}{2n}}\biggr) + \sigma t \Biggr]; \label{eq:GST-warmup}
\end{align}
see Figure~\ref{fig:gaussian-TS} for an illustration.
In Appendix~\ref{sec:all-known} we prove that, for $\epsilon \in \bigl[0,1-\frac{4\log_2(8/\alpha)}{n}\bigr]$, the confidence intervals $\{\widehat{\mathrm{CI}}_{n,\sigma,\alpha}^{(1)}(t):t \geq 0\}$ have simultaneous coverage for $\theta$; i.e.
\begin{align}
    \mathbb{P}_R\biggl( \theta\in \bigcap_{t\geq 0} \widehat{\mathrm{CI}}^{(1)}_{n,\sigma,\alpha}(t) \biggr) \geq 1-\alpha. \label{eq:warmup-coverage-all-known}
\end{align}
To see the intuition behind this result, we focus on the event that $E_0 \coloneqq \bigcup_{i=1}^n \{Z_i \neq \star\}$, which has high probability.  Since
\begin{align*}
    \{Q_n^-(p) \leq a\} \cap E_0 = \biggl\{\frac{1}{n} \sum_{i=1}^n \indi_{\{Z_i \leq a\}} \geq p\biggr\} \cap E_0,
\end{align*}
we have
\begin{align}
    &\bigcap_{t \geq 0} \biggl\{\theta \geq Q_n^-\biggl(1-\Phi(t)+\sqrt{\frac{\log(4/\alpha)}{2n}}\biggr) + \sigma t\biggr\} \cap E_0\nonumber\\
    &\hspace{4cm} = \bigcap_{t \geq 0} \biggl\{\frac{1}{n} \sum_{i=1}^n \indi_{\{Z_i -\theta \leq \sigma t\}} \geq 1-\Phi(t)+\sqrt{\frac{\log(4/\alpha)}{2n}}\biggr\} \cap E_0. \label{eq:warmup-DKW}
\end{align}
Moreover, the event~\eqref{eq:warmup-DKW} has low probability by the (one-sided) Dvoretzky--Kiefer--Wolfowitz--Masssart--Reeve (DKWMR) inequality \citep{reeve2024short}. A similar argument can be applied to the left endpoint of the confidence interval. Thus, we may show that $\theta \in \bigcap_{t\geq 0} \widehat{\mathrm{CI}}^{(1)}_{n,\sigma,\alpha}(t)$ with high probability.  In the special case $t=0$ this is an interval centred at the median, which was proposed by \citet{scheffe1945non} for the setting of no missingness $(\epsilon = 0)$.  To obtain a confidence interval whose length achieves the minimax rate in our missing data context, we take $t=t_{n,\epsilon,\alpha} \coloneqq \sqrt{0\vee\{\log(n\epsilon^2)-\log\log(4/\alpha)\}/2}$.  We show in Proposition~\ref{prop:all-param-known}\emph{(b)} that with this choice of~$t$, the length of our confidence interval satisfies with probability at least $1-\alpha/2$ that
\begin{align}
    \bigl| \hat{\mathrm{CI}}^{(1)}_{n,\sigma,\alpha}(t_{n,\epsilon,\alpha}) \bigr| \lesssim_{\alpha} \sigma\biggl\{\frac{1}{\sqrt{n}} + \frac{\log\bigl(1/(1-\epsilon)\bigr)}{\sqrt{1\vee \log(n\epsilon^2)}}\biggr\}, \label{eq:warmup-all-known-rate}
\end{align}
under the simplified model~\eqref{eq:model-mnar}.

Next, suppose that both $\sigma\in(0,\sigma_{\max}]$ and $\epsilon\in[0,1]$ are unknown, where $\sigma_{\max}>0$ is a known upper bound on $\sigma$. Replacing $\sigma$ by $\sigma_{\max}$ in~\eqref{eq:GST-warmup} yields a family of confidence intervals that no longer depends on the unknown $\sigma$. Moreover, by a similar argument as above (using the DKWMR inequality), these confidence intervals continue to enjoy simultaneous coverage over all $t\geq 0$. In the previous paragraph, we chose $t$ explicitly as a function of $\epsilon$. When $\epsilon$ is also unknown, however, the simultaneous coverage over all $t\geq 0$ allows us to choose $t$ adaptively, by minimising the confidence interval length over $t\geq 0$, or by taking the intersection of all the confidence intervals over $t\geq 0$. This then yields an adaptive confidence interval that does not depend on either $\sigma$ or $\epsilon$. We refer the readers to Appendix~\ref{sec:proofs-for-warmup} for further details, where we consider four cases of the simplified contamination model~\eqref{eq:model-mnar}, according to whether $\sigma$ and $\epsilon$ are known or unknown, and construct confidence intervals with coverage and length guarantees in each case.

\subsubsection{The general setting} \label{sec:gaussian-CI-construction-general}
We return to the model $\mathcal{M}\bigl(N(\theta,\sigma),\epsilon_1,\epsilon_2\bigr)$ considered in Theorem~\ref{thm:gaussian-minimax-rates}. For $t \geq 0$ and $\lambda_1,\lambda_2,\beta,\bar{\sigma} > 0$, define generalised Gaussian Scheff{\'e}--Tukey intervals as
\begin{align}
    &\widehat{\mathrm{CI}}_{n}^{(\mathrm{GST})}(t; \lambda_1,\lambda_2,\beta,\bar{\sigma})\nonumber\\
    &\coloneqq \Biggl[Q_n^+\biggl(\lambda_1\bigl\{1-\Phi(t)\bigr\} + \sqrt{\frac{\lambda_2\log(4/\beta)}{2n}}\biggr)-\bar{\sigma} t,\, Q_n^-\biggl(\lambda_1\bigl\{1-\Phi(t)\bigr\} + \sqrt{\frac{\lambda_2\log(4/\beta)}{2n}}\biggr) + \bar{\sigma} t \Biggr]. \label{eq:GST}
\end{align}
When $\lambda_1=\lambda_2=1$, $\beta=\alpha$ and $\bar{\sigma}=\sigma$, we recover the confidence interval~\eqref{eq:GST-warmup}. In what follows, we will make specific choices of $\lambda_1,\lambda_2,\beta$ and $\bar{\sigma}$ in the four different cases of Theorem~\ref{thm:gaussian-minimax-rates}, so that the confidence intervals have coverage at least~$1 - \alpha$ and the desired length guarantees.

\begin{itemize}
    \item[\emph{(a)}]  Known $\sigma,\epsilon_1,\epsilon_2$: We take $\lambda_1=1-\epsilon_1$, $\lambda_2=18(1-\epsilon_1)$, $\beta=\alpha/4$ and $\bar{\sigma}=\sigma$. Using a missing data version of the DKWMR inequality (Lemma~\ref{lemma:dkw-missing}), we show that $\widehat{\mathrm{CI}}_{n}^{(\mathrm{GST})}(t; \lambda_1,\lambda_2,\beta,\bar{\sigma})$ has simultaneous coverage for all $t\geq 0$.  The desired length guarantee is obtained with the choice $t = t_{n,\epsilon_1,\epsilon_2,\alpha} \coloneqq \sqrt{0\vee\bigl\{\log\bigl(n\epsilon_2^2/(1-\epsilon_1)\bigr)-\log\log(16/\alpha)\bigr\}/2}$.
    \item[\emph{(b)}] Unknown $\sigma$, known $\epsilon_1,\epsilon_2$: Let $r_{n,\epsilon_1,\epsilon_2,\alpha}$ be defined as~\eqref{eq:r-n-eps12-alpha}, which depends only on $n,\epsilon_1,\epsilon_2$ and~$\alpha$, and represents a bound on the relative error of an estimate of $\sigma$. When $r_{n,\epsilon_1,\epsilon_2,\alpha} \leq 0.9$, we take $\lambda_1=1-\epsilon_1$, $\lambda_2 = 18(1-\epsilon_1)$ and $\beta=\alpha/12$, similarly to~\emph{(a)}. However, as we do not have access to the true $\sigma$, we take $\bar{\sigma}$ to be an estimate~$\hat{\sigma}_{\mathrm{R}}$ (defined in~\eqref{eq:hat-sigma-R}) that majorises $\sigma$ with high probability, and choose $t=t_{n,\epsilon_1,\epsilon_2,\alpha}' \coloneqq \sqrt{1\vee\bigl\{\log\bigl(n\epsilon_2^2/(1-\epsilon_1)\bigr)-\log\log(48/\alpha)\bigr\}/2}$. Surprisingly, although the true~$\sigma$ is replaced by a data-dependent high-probability upper bound $\hat{\sigma}_{\mathrm{R}}$, the estimation error in $\hat{\sigma}_{\mathrm{R}}$ does not affect the resulting rate of the confidence interval length. Finally, when $r_{n,\epsilon_1,\epsilon_2,\alpha}>0.9$, we take our confidence interval to be $[Z_{\min},Z_{\max}]$; see the definition of $\hat{\mathrm{CI}}^{(4)}_{n,\epsilon_1,\epsilon_2,\alpha}$ in~\eqref{eq:gaussian-general-CI-b}.
    \item[\emph{(c)}] Known $\sigma$, unknown $\epsilon_1,\epsilon_2$: We take $\lambda_1=\lambda_2=1$, $\beta=\alpha$ and $\bar{\sigma}=\sigma$.  Since $\mathcal{M}(P,\epsilon_1,\epsilon_2) \subseteq \mathcal{M}(P,0,\epsilon_1+\epsilon_2)$, the argument outlined in Section~\ref{sec:warm-up} can be employed to show that $\widehat{\mathrm{CI}}_{n}^{(\mathrm{GST})}(t; \lambda_1,\lambda_2,\beta,\bar{\sigma})$ has simultaneous coverage for all $t\geq 0$. Thus, we may choose $t$ to minimise the length of $\widehat{\mathrm{CI}}_{n}^{(\mathrm{GST})}(t; \lambda_1,\lambda_2,\beta,\bar{\sigma})$ over $t\geq 0$. 
    \item[\emph{(d)}] Unknown $\sigma,\epsilon_1,\epsilon_2$: In this case, we assume that $\sigma\in(0,\sigma_{\max}]$ is unknown, but $\sigma_{\max}$ is known. We thus take $\lambda_1=\lambda_2=1$, $\beta=\alpha$ and $\bar{\sigma}=\sigma_{\max}$. Again, $\widehat{\mathrm{CI}}_{n}^{(\mathrm{GST})}(t; \lambda_1,\lambda_2,\beta,\bar{\sigma})$ has simultaneous coverage for all $t\geq 0$, and we choose $t$ to minimise the length of the confidence interval. 
\end{itemize}

Finally, we note that an immediate extension of the methodology developed here holds for location--scale families with common distribution function $F$, at least in cases~\emph{(a)},~\emph{(c)} and~\emph{(d)}.  Indeed, replacing the Gaussian distribution function $\Phi$ with the distribution function $F$ in~\eqref{eq:GST} yields valid (adaptive) confidence intervals in these cases, with identical proofs.  On the other hand, in case~\emph{(b)}, a new estimate $\hat{\sigma}_{\mathrm{R}}$ would need to be constructed on a case-by-case basis.

\section{The nonparametric setting} \label{sec:nonparametric}

Our aim in this section is to study adaptive confidence intervals under nonparametric assumptions on the base distribution $P$.  Specifically, in Section~\ref{sec:adaptation-tail}, we impose tail conditions on $P$, while in Section~\ref{sec:adaptation-shape}, we consider a symmetry assumption.  

Let $\mathcal{P}$ be a class of distributions on $\mathbb{R}$ and let $\mathcal{E} \subseteq \{(\epsilon_1,\epsilon_2) \in [0,1]^2 : \epsilon_1 + \epsilon_2 \leq 1\}$.  We define $\widehat{\mathcal{CI}}(\mathcal{P},\mathcal{E})$ as the set of all confidence intervals for the mean of $P \in \mathcal{P}$ based on $Z_1,\ldots,Z_n\overset{\mathrm{iid}}{\sim} M \in \mathcal{M}(P,\epsilon_1,\epsilon_2)$, with uniform coverage at least $1-\alpha$ and that adapt to $(\epsilon_1,\epsilon_2) \in \mathcal{E}$, i.e.~the set of all intervals $\widehat{\mathrm{CI}}$ such that
\begin{align*}
    \inf_{P\in\mathcal{P}}\; \inf_{(\epsilon_1,\epsilon_2) \in \mathcal{E}}\; \inf_{M\in\mathcal{M}(P,\epsilon_1,\epsilon_2)}\; \mathbb{P}_M\bigl(\theta\in\hat{\mathrm{CI}}(Z_1,\ldots,Z_n)\bigr) \geq 1-\alpha.
\end{align*}
For $(\epsilon_1,\epsilon_2) \in \mathcal{E}$, we define
\begin{align*}
    L_{n,\epsilon_1,\epsilon_2}\bigl(\hat{\mathcal{CI}}(\mathcal{P},\mathcal{E})\bigr) \coloneqq \inf\Bigl\{r>0 : \inf_{\hat{\mathrm{CI}} \in \hat{\mathcal{CI}}(\mathcal{P},\mathcal{E})}\; \sup_{P\in\mathcal{P}}\; \sup_{M \in \mathcal{M}(P, \epsilon_1, \epsilon_2)} \mathbb{P}_M \bigl(|\hat{\mathrm{CI}}(Z_1,\ldots,Z_n)| \geq r\bigr) \leq \alpha \Bigr\}
\end{align*}
to be the minimax length of confidence intervals in $\hat{\mathcal{CI}}(\mathcal{P},\mathcal{E})$ based on $Z_1,\ldots,Z_n\overset{\mathrm{iid}}{\sim} M \in \mathcal{M}(P,\epsilon_1,\epsilon_2)$.

\subsection{Confidence intervals under tail conditions} \label{sec:adaptation-tail}

\subsubsection{Minimax rates}
For $a<b$, let $\mathcal{P}_{\mathrm{BD}}(\theta,[a,b])$ denote the set of distributions supported on $[a,b]$ with mean $\theta$. For $\sigma>0$, we define $\mathcal{P}_{\mathrm{SG}}(\theta,\sigma)$ to be the set of all distributions~$P$ on $\mathbb{R}$ such that if $X\sim P$, then $\mathbb{E}(X)=\theta$ and $\mathbb{E}\, \exp\bigl((X-\theta)^2/\sigma^2\bigr) \leq 2$. Finally, let $\mathcal{P}_{\mathrm{FV}}(\theta,\sigma)$ be the set of all distributions~$P$ on~$\mathbb{R}$ such that $\mathbb{E}_P(X)=\theta$ and $\mathrm{Var}_P(X) \leq \sigma^2$.

\begin{Theorem}
\label{thm:nonparametric-minimax-rates}
    Let $n\in\mathbb{N}$ and $\alpha\in(0,1)$. Define $\mathcal{E} \coloneqq \{(\epsilon_1,\epsilon_2)\in[0,1]^2 : \epsilon_1+\epsilon_2 \leq 1\}$, and let $(\epsilon_1,\epsilon_2)\in\mathcal{E}$ be such that $\epsilon_+ \coloneqq\epsilon_1 + \epsilon_2 \leq \frac{27\log(12/\alpha)}{n}$. The following statements hold.
    \begin{itemize}
        \item[(a)] (Bounded random variables) Let $a<b$ and $\mathcal{P}_{\mathrm{BD}}\coloneqq \bigcup_{\theta\in[a,b]} \mathcal{P}_{\mathrm{BD}}(\theta,[a,b])$. If $\alpha\in(0,\frac{1}{4}]$, then
        \begin{align*}
            L_{n,\epsilon_1,\epsilon_2} \bigl(\hat{\mathcal{CI}}(\mathcal{P}_{\mathrm{BD}},\mathcal{E})\bigr) \asymp_{\alpha} (b-a) \Bigl(\frac{1}{\sqrt{n}} + \epsilon_+\Bigr).
        \end{align*}
        \item[(b)] (Sub-Gaussian random variables) Let $\sigma_{\max}>0$ and $\mathcal{P}_{\mathrm{SG}} \coloneqq \bigcup_{\theta\in\mathbb{R}} \mathcal{P}_{\mathrm{SG}}(\theta,\sigma_{\max})$. There exists a universal constant $c>0$ such that if $\alpha\in(6e^{-cn},\frac{1}{4}]$, then
        \begin{align*}
            L_{n,\epsilon_1,\epsilon_2} \bigl(\hat{\mathcal{CI}}(\mathcal{P}_{\mathrm{SG}},\mathcal{E})\bigr) \asymp_{\alpha} \sigma_{\max} \Biggl\{ \frac{1}{\sqrt{n}} + \frac{\epsilon_+}{1-\epsilon_+} \sqrt{\log\biggl(\frac{1}{\epsilon_+}\biggr)} \wedge \sqrt{\log\biggl(\frac{1}{1-\epsilon_+}\biggr)} \Biggr\}.
        \end{align*}
        \item[(c)] (Finite-variance random variables) Let $\sigma_{\max}>0$ and $\mathcal{P}_{\mathrm{FV}} \coloneqq \bigcup_{\theta\in\mathbb{R}} \mathcal{P}_{\mathrm{FV}}(\theta,\sigma_{\max})$. If $\alpha\in(0,\frac{1}{4}]$, then
        \begin{align*}
            L_{n,\epsilon_1,\epsilon_2}\bigl(\hat{\mathcal{CI}}(\mathcal{P}_{\mathrm{FV}},\mathcal{E})\bigr) \asymp_{\alpha} \sigma_{\max} \biggl\{ \frac{1}{\sqrt{n}} + \sqrt{\frac{\epsilon_+}{1-\epsilon_+}}\biggr\}.
        \end{align*}
    \end{itemize}
\end{Theorem}
We provide explicit constructions of confidence intervals that achieve the length guarantees in each of the three cases in Section~\ref{sec:construct-CI-nonparam}; the corresponding coverage guarantees hold for all $(\epsilon_1,\epsilon_2) \in\mathcal{E}$.  One important feature of the rates in Theorem~\ref{thm:nonparametric-minimax-rates} is that they only depend on $(\epsilon_1,\epsilon_2)$ through $\epsilon_+ = \epsilon_1+\epsilon_2$. In other words, from a minimax length rate perspective, there is no difference between an MCAR model and an MNAR model with the same value of~$\epsilon_+$ when $\epsilon_1,\epsilon_2$ are unknown.  Indeed, from Theorem~\ref{thm:nonparametric-minimax-rates} we see that when the true data generating mechanism is MCAR with constant $\epsilon_1>0$, any confidence interval that adapts to $(\epsilon_1,\epsilon_2) \in \mathcal{E}$ must have a minimax length of constant order.  This is in stark contrast with the case of non-adaptive confidence intervals, where applications of standard sub-Gaussian tail bounds or Markov's inequality (in the finite variance case) yield valid intervals of parametric length centred at the sample mean under MCAR data.

Each of the minimax length rates in Theorem~\ref{thm:nonparametric-minimax-rates} decomposes as a sum of a parametric term and a contamination term, where the latter depends only on $\epsilon_+$ and reflects the different tail conditions imposed.  For small $\epsilon_+$, the contamination term is proportional to $\epsilon_+$, $\epsilon_+\sqrt{\log(1/\epsilon_+)}$ and $\epsilon_+^{1/2}$ in the bounded, sub-Gaussian and finite variance cases respectively, meaning that it dominates the parametric term when $\epsilon_+ \gg n^{-1/2}$, $\epsilon_+ \gg (n\log n)^{-1/2}$ and $\epsilon_+ \gg 1/n$ in the three cases.  The transition point between the two terms in the minimum in the sub-Gaussian setting occurs when $\epsilon_+$ is of constant order, so as $\epsilon_+ \to 1$, the minimax length grows proportionally to $\sqrt{\log\bigl(1/(1-\epsilon_+)\bigr)}$, while for the class of finite-variance distributions, it grows proportionally to $\sqrt{1/(1-\epsilon_+)}$.

As a final remark on Theorem~\ref{thm:nonparametric-minimax-rates}, as shown in Appendix~\ref{sec:proof-of-nonparametric-rates}, for any fixed $\epsilon_+ < 1$, our confidence interval $\hat{\mathrm{CI}}_{n,a,b,\alpha}^{(\mathrm{BD})}$ satisfies 
\begin{align*}
 \sup_{P\in\mathcal{P}_{\mathrm{BD}}} \sup_{M\in\mathcal{M}(P,\epsilon_1,\epsilon_2)} \mathbb{P}_M\bigl(|\hat{\mathrm{CI}}_{n,a,b,\alpha}^{(\mathrm{BD})}| > (b-a)\epsilon_+ \bigr) \rightarrow 0
\end{align*}
as $n \rightarrow \infty$, and our lower bound satisfies $\liminf_{n\to\infty} L_{n,\epsilon_1,\epsilon_2}\bigl(\hat{\mathcal{CI}}(\mathcal{P}_{\mathrm{BD}},\mathcal{E})\bigr) \geq (b-a)\epsilon_+$.  In other words, asymptotically, the interval $\hat{\mathrm{CI}}_{n,a,b,\alpha}^{(\mathrm{BD})}$ attains the minimax length even at the level of constants.  Similarly, in case~\emph{(c)}, for any fixed $\epsilon_+ < 1$, our construction $\hat{\mathrm{CI}}_{n,\sigma_{\max},\alpha}^{(\mathrm{FV})}$ satisfies 
\begin{align*}
    \sup_{P\in\mathcal{P}_{\mathrm{FV}}} \sup_{M\in\mathcal{M}(P,\epsilon_1,\epsilon_2)} \mathbb{P}_M\biggl(|\hat{\mathrm{CI}}_{n,\sigma_{\max},\alpha}^{(\mathrm{FV})}| > \sigma_{\max}\sqrt{\frac{\epsilon_+}{1-\epsilon_+}} \biggr) \rightarrow 0
\end{align*}
as $n \rightarrow \infty$, and our lower bound satisfies $\liminf_{n\to\infty} L_{n,\epsilon_1,\epsilon_2}\bigl(\hat{\mathcal{CI}}(\mathcal{P}_{\mathrm{FV}},\mathcal{E})\bigr) \geq \sigma_{\max}\sqrt{\epsilon_+/(1-\epsilon_+)}$.

\subsubsection{Constructions of the confidence intervals} \label{sec:construct-CI-nonparam}
For $Z_1,\ldots,Z_n \overset{\mathrm{iid}}{\sim} M \in \mathcal{M}(P,\epsilon_1,\epsilon_2)$, let\footnote{with the convention that $0/0\coloneqq 0$.} 
\[
\bar{Z}_n \coloneqq \frac{1}{n\narrowhat{p}_n}\sum_{i=1}^n Z_i \indi_{\{Z_i \neq \star\}},
\]
where we recall the definition of $\narrowhat{p}_n$ from~\eqref{eq:hat-pn} as the empirical fraction of observed entries.  We will also make use of 
\begin{align}
    \narrowhat{p}_{n,\alpha}^- \coloneqq \biggl(\narrowhat{p}_n - \sqrt{\frac{2\narrowhat{p}_n(1-\narrowhat{p}_n)\log(6/\alpha)}{n}} - \frac{3\log(6/\alpha)}{n}\biggr)\vee 0, \label{eq:p-lb}
\end{align}
which satisfies $\mathbb{P}(\narrowhat{p}_{n,\alpha}^- \leq p) \geq 1 - \alpha/2$, where $p\coloneqq \mathbb{P}(Z_1\neq\star)$,  by the empirical Bernstein inequality \citep[Theorem~1]{audibert2009exploration}.  We consider the family of confidence intervals
\[
\widehat{\mathrm{CI}}(s, \Delta_{\mathrm{L}}, \Delta_{\mathrm{U}}) \coloneqq \bigl[s \bar{Z}_n + \Delta_{\mathrm{L}},\, s \bar{Z}_n + \Delta_{\mathrm{U}}\bigr],
\]
and specify the values of $s$, $\Delta_{\mathrm{L}}$ and $\Delta_{\mathrm{U}}$ in each case.

\begin{itemize}
    \item[\emph{(a)}] (Bounded random variables)  We take $s = \narrowhat{p}_{n, \alpha}^-$, $\Delta_{\mathrm{L}} = \Delta_{\mathrm{BD},\mathrm{L}}$ and $\Delta_{\mathrm{U}} = \Delta_{\mathrm{BD},\mathrm{U}}$, where
    \begin{align*}
     \Delta_{\mathrm{BD},\mathrm{L}} \coloneqq (1-\narrowhat{p}_{n,\alpha}^-)a - (b-a)\sqrt{\frac{\log(4/\alpha)}{2n\narrowhat{p}_n}}  \quad \text{and} \quad \Delta_{\mathrm{BD},\mathrm{U}} \coloneqq (1-\narrowhat{p}_{n,\alpha}^-)b + (b-a)\sqrt{\frac{\log(4/\alpha)}{2n\narrowhat{p}_n}},
    \end{align*}
    and construct the confidence interval
    \begin{align}
        \hat{\mathrm{CI}}^{(\mathrm{BD})}_{n,a,b,\alpha} \coloneqq \widehat{\mathrm{CI}}(s, \Delta_{\mathrm{BD}, \mathrm{L}}, \Delta_{\mathrm{BD}, \mathrm{U}}). \label{eq:CI-BD}
    \end{align}
    
    \item[\emph{(b)}] (sub-Gaussian random variables) For $p\in[0,1]$, we let
    \begin{align*}
        r^{(\mathrm{SG})}(p) \coloneqq \sqrt{\log4 \cdot \log(1/p)} \wedge  \frac{11(1-p)}{5p} \log^{1/2}\Bigl(\frac{2}{1-p}\Bigr),
    \end{align*}
    and set
    \[
    \Delta_{\mathrm{SG}} \coloneqq \sigma_{\max} \Biggl(r^{(\mathrm{SG})}(\narrowhat{p}_{n,\alpha}^-) + 4\sqrt{\frac{\log(4)\log(2/\narrowhat{p}_{n,\alpha}^-)\log(4/\alpha)}{n \narrowhat{p}_n}} \Biggr).
    \]
    We define our sub-Gaussian confidence interval as
    \begin{align}
    \widehat{\mathrm{CI}}_{n, \sigma_{\max}, \alpha}^{(\mathrm{SG})} \coloneqq \widehat{\mathrm{CI}}(1, -\Delta_{\mathrm{SG}}, \Delta_{\mathrm{SG}}). \label{eq:CI-SG}
    \end{align}

    \item[\emph{(c)}] (Finite-variance random variables) For $p\in[0,1]$, let
    \begin{align*}
        r^{(\mathrm{FV})}(p) \coloneqq \sqrt{\frac{1-p}{p}},
    \end{align*}
    and set
    \[
    \Delta_{\mathrm{FV}} \coloneqq \sigma_{\max}r^{(\mathrm{FV})}(\narrowhat{p}_{n,\alpha}^-) + \sigma_{\max} \sqrt{\frac{2}{n\narrowhat{p}_n\narrowhat{p}_{n,\alpha}^-\alpha}}.
    \]
    Our finite-variance confidence interval is
    \begin{align}
    \widehat{\mathrm{CI}}_{n, \sigma_{\max}, \alpha}^{(\mathrm{FV})} \coloneqq \widehat{\mathrm{CI}}(1, -\Delta_{\mathrm{FV}}, \Delta_{\mathrm{FV}}). \label{eq:CI-FV}
    \end{align}
\end{itemize}
The reason that, in the bounded random variables case, our confidence interval is not centred at $\bar{Z}_n$ is that, unlike the sub-Gaussian and finite-variance classes, the class $\mathcal{P}_{\mathrm{BD}}(\theta,[a,b])$ is not translation invariant ($a$ and $b$ are assumed known).  The intuition is that, if $P\in \mathcal{P}_{\mathrm{BD}}(\theta,[a,b])$, $Z \sim M \in \mathsf{MNAR}_P$ and $p\coloneqq \mathbb{P}(Z \neq \star)$, then by Lemma~\ref{lemma:bounded-rv-bias},
\begin{align*}
    p\mathbb{E}(Z\,|\,Z\neq\star) + (1-p)a \leq \theta \leq p\mathbb{E}(Z\,|\,Z\neq\star) + (1-p)b,
\end{align*}
so our interval is obtained by replacing $p$ with its high-probability lower bound $\narrowhat{p}_{n,\alpha}^-$, and replacing $\mathbb{E}(Z\,|\,Z\neq\star)$ with high-probability upper and lower bounds.  On the other hand, in the translation invariant sub-Gaussian and finite-variance cases, our constructions begin with the bounds $|\mathbb{E}(Z\,|\,Z\neq\star) - \theta| \leq r^{(\mathrm{SG})}(p)$ and $|\mathbb{E}(Z\,|\,Z\neq\star) - \theta| \leq r^{(\mathrm{FV})}(p)$ respectively; see Lemmas~\ref{lemma:sub-gaussian-bias} and~\ref{lemma:finite-variance-bias}.  Again, we may then replace $p$ with $\narrowhat{p}_{n,\alpha}^-$ and control the sub-Gaussian norm (respectively variance) inflation in replacing $\mathbb{E}(Z\,|\,Z\neq\star)$ with $\bar{Z}_n$.

\subsubsection{Impossibility of adaptation to tail condition}

The optimal confidence intervals in Section~\ref{sec:construct-CI-nonparam} require knowledge of the tail condition of the underlying distribution. A natural question to ask is whether it is possible to further adapt to the tail conditions. Proposition~\ref{thm:finite-variance-lb} below provides a negative answer to this question.   For $\sigma\in(0,\sigma_{\max}]$, define $\mathcal{P}_{\mathrm{SG},\sigma} \coloneqq \bigcup_{\theta\in\mathbb{R}}\mathcal{P}_{\mathrm{SG}}(\theta,\sigma) \subseteq \mathcal{P}_{\mathrm{FV}}$. 

\begin{Proposition}\label{thm:finite-variance-lb}
    Let $\hat{\mathcal{CI}} = \hat{\mathcal{CI}}(\mathcal{P}_{\mathrm{FV}},\mathcal{E})$. For $\sigma\in(0,\sigma_{\max}]$, let 
    \begin{align*}
        L(n,\sigma,\epsilon_1,\epsilon_2) \coloneqq \inf \biggl\{r>0 :  \inf_{\hat{\mathrm{CI}} \in \hat{\mathcal{CI}}}\; \sup_{P\in\mathcal{P}_{\mathrm{SG},\sigma}} \sup_{M\in\mathcal{M}(P,\epsilon_1,\epsilon_2)} \mathbb{P}_M\bigl(|\hat{\mathrm{CI}}(Z_1,\ldots,Z_n)| \geq r\bigr) \leq \alpha \biggr\}.
    \end{align*}
    If $\alpha\in(0,\frac{1}{4}]$, $\sigma\in(0,\sigma_{\max}]$ and $\epsilon_+ \leq 1$, then 
    \begin{align*}
        L(n,\sigma,\epsilon_1,\epsilon_2) \geq \frac{c\sigma}{\sqrt{n}} \vee \sigma_{\max} \sqrt{\frac{\epsilon_+}{1-\epsilon_+}},
    \end{align*}
    where $c>0$ is a universal constant.
\end{Proposition}
We remark that $L_{n,\epsilon_1,\epsilon_2}\bigl(\hat{\mathcal{CI}}(\mathcal{P}_{\mathrm{FV}},\mathcal{E})\bigr) \geq L(n,\sigma,\epsilon_1,\epsilon_2)$ by definition.
Proposition~\ref{thm:finite-variance-lb} reveals that if $\sigma\in(0,\sigma_{\max}]$ and $\hat{\mathrm{CI}} \in \hat{\mathcal{CI}}(\mathcal{P}_{\mathrm{FV}},\mathcal{E})$, then the contamination term in the minimax length of $\hat{\mathrm{CI}}$ over the sub-Gaussian class $\mathcal{P}_{\mathrm{SG},\sigma}$ with sub-Gaussian norm (hence standard deviation) at most~$\sigma$ has the same rate as the minimax length of $\hat{\mathrm{CI}}$ over the set of distributions in $\mathcal{P}_{\mathrm{FV}}$ with variance at most~$\sigma_{\max}^2$. In other words, any adaptive confidence interval that has coverage over the finite-variance class has minimax length in the smaller sub-Gaussian class of the same order as the minimax length in the finite-variance class.

\subsection{Confidence intervals under symmetry constraints}\label{sec:adaptation-shape}

As our next nonparametric class, we consider symmetric distributions, without any tail conditions.  Taking $t=0$ in~\eqref{eq:GST-warmup} yields the following confidence interval
\begin{align*}
    \hat{\mathrm{CI}}^{(\mathrm{Sym})}_{n,\alpha} \coloneqq \Biggl[Q_n^+\biggl(\frac{1}{2} + \sqrt{\frac{\log(4/\alpha)}{2n}}\biggr),\, Q_n^-\biggl(\frac{1}{2}+\sqrt{\frac{\log(4/\alpha)}{2n}}\biggr) \Biggr].
\end{align*}
When the data are fully observed, this coincides with the interval proposed by \citet{scheffe1945non}. 
In this case, the interval does not depend on $\sigma$, $\epsilon_1$ or $\epsilon_2$. The theorem below provides coverage and length guarantees for this interval.

\begin{Theorem}\label{thm:symmetric}
Let $n\in\mathbb{N}$, $\alpha\in(0,1]$, $\epsilon_1,\epsilon_2 \in [0,1]$ be such that $\epsilon_+ \coloneqq \epsilon_1+\epsilon_2 \leq 1 - \frac{4\log_2(8/\alpha)}{n}$. Let~$P$ be a distribution on $\mathbb{R}$ that admits a Lebesgue density that is symmetric about~$\theta\in\mathbb{R}$. Suppose further that $f$ is positive and continuous at~$\theta$ and let $Z_1,\ldots,Z_n \overset{\mathrm{iid}}{\sim} M\in \mathcal{M}(P,\epsilon_1,\epsilon_2)$. Then 
\begin{align*}
    \mathbb{P}\Bigl(\theta \in \hat{\mathrm{CI}}^{(\mathrm{Sym})}_{n,\alpha}(Z_1,\ldots,Z_n)\Bigr) \geq 1-\alpha.
\end{align*}
Moreover, there exist $c\in(0,1/2)$ and $C>0$, depending only on~$f$ and $\alpha$, such that if $\epsilon_+ \leq c$ and $n \geq C$, then
\begin{align*}
    \bigl|\hat{\mathrm{CI}}^{(\mathrm{Sym})}_{n,\alpha}\bigr| \lesssim_{\alpha} \frac{1}{f(\theta)}\Bigl(\frac{1}{\sqrt{n}}+\epsilon_+\Bigr)
\end{align*}
with probability at least $1-\alpha/2$.
\end{Theorem}
Comparing Theorem~\ref{thm:symmetric} with Theorem~\ref{thm:gaussian-minimax-rates}\emph{(d)}, we see that when the base distribution $P$ is symmetric, $\epsilon_+\leq c$ and we regard $f(\theta)$ as constant, the length guarantee for $\hat{\mathrm{CI}}^{(\mathrm{Sym})}_{n,\alpha}$ is minimax optimal and is of the same order as if we knew in advance that $P$ were Gaussian. If we further assume that~$P$ has variance bounded by~$\sigma_{\max}^2$, then $\hat{\mathrm{CI}}^{(\mathrm{Sym})}_{n,\alpha/2} \cap \hat{\mathrm{CI}}^{(\mathrm{FV})}_{n,\sigma_{\max}, \alpha/2}$ is a confidence interval with coverage at least $1-\alpha$. Moreover, 
\begin{align*}
    \bigl| \hat{\mathrm{CI}}^{(\mathrm{Sym})}_{n,\alpha/2} \cap \hat{\mathrm{CI}}^{(\mathrm{FV})}_{n,\sigma_{\max}, \alpha/2} \bigr| \lesssim_{\alpha} \begin{cases}
        \frac{1}{f(\theta)}\bigl( \frac{1}{\sqrt{n}} + \epsilon_+ \bigr) \quad&\text{if } \epsilon_+ \leq c\\
        \sigma_{\max}\bigl(\frac{1}{\sqrt{n}} + \sqrt{\frac{\epsilon_+}{1-\epsilon_+}}\bigr) \quad&\text{if } c < \epsilon_+ \leq 1-\frac{27\log(12/\alpha)}{n};
    \end{cases}
\end{align*}
thus, under this additional variance bound, we can extend the range of $\epsilon_+$ for which we have a length guarantee compared with Theorem~\ref{thm:symmetric}.

\section{Consequences for causal inference} \label{sec:causal}

In this section, we detail a natural and straightforward extension of our results to the causal inference setting.  To this end, given $P \in \mathcal{P}(\mathbb{R}^2)$, we extend the definition of $\mnar_P$ in~\eqref{eq:mnar-def} by writing
\[
\mnar_P \coloneqq \Bigl\{\mathsf{Law}(X \ostar \Omega):\; X \sim P \text{ and } \Omega \in \bigl\{(0,1), (1, 0)\bigr\}\Bigr\},
\]
so only one coordinate of $X$ is observed.  Letting $P_0$ and $P_1$ denote the marginal distributions of~$P$, and letting $q,\eta \in [0,1]$ denote a unconfoundedness probability and contamination parameter respectively, we then consider the \emph{realisable causal class} of distributions
\begin{align} \label{eq:causal-model}
\mathcal{M}_{2}(P, q, \eta)\coloneqq (1 - \eta) \cdot \underbrace{\Bigl\{(1 - q)\, P_0 \otimes \delta_{\{\star\}} + q \, \delta_{\{\star\}} \otimes P_1\Bigr\}}_{\text{MCAR/Unconfounded component}} + \, \eta \, \underbrace{\mnar_{P}}_{\text{Departure term}}.
\end{align}
The distributions in $\mathcal{M}_2(P, q,\eta)$ are $\bigl((1-\eta), \eta\bigr)$ mixtures of an unconfounded distribution and an arbitrarily confounded version of the same base distribution.  If $M \in \mathcal{M}_2(P, q,\eta)$ has marginals $M_0$ and $M_1$, then $M_0 \in \mathcal{M}\bigl(P_0,(1-\eta)q,\eta\bigr)$ and $M_1 \in \mathcal{M}\bigl(P_1,(1-\eta)(1-q),\eta\bigr)$.  The following lemma gives a form of converse statement.   
\begin{Lemma}\label{lem:bivariate-univariate-reduction}
        Fix $P \in \mathcal{P}(\mathbb{R}^2)$ with marginals $P_0$ and $P_1$.  For any $M_0' \in \mathcal{M}\bigl(P_0,(1-\eta)q,\eta\bigr)$, 
        there exists $M \in \mathcal{M}_2(P, q,\eta)$ having first marginal~$M_0'$.  Similarly, for any $M_1' \in \mathcal{M}\bigl(P_1,(1-\eta)(1-q),\eta\bigr)$, 
        there exists $M \in \mathcal{M}_2(P, q,\eta)$ having second marginal $M_1'$.
\end{Lemma}

Given a sample from $M \in \mathcal{M}_2(P,q,\eta)$, our target of inference is the \emph{average treatment effect}
\[
\mathrm{ATE}_P \coloneqq \mathbb{E}_{P_1}(X) - \mathbb{E}_{P_0}(X);
\]
we will construct a $(1 - \alpha)$-level confidence interval $\widehat{\mathrm{CI}}_{\mathrm{ATE}_P}$ for $\mathrm{ATE}_P$ by combining confidence intervals for $\theta_1 \coloneqq \mathbb{E}_{P_1}(X)$ and $\theta_0 \coloneqq \mathbb{E}_{P_0}(X)$.  Indeed, the Minkowski difference of any $(1 - \alpha/2)$-level confidence interval for $\theta_1$ and any $(1 -\alpha/2)$-level confidence interval for $\theta_0$ yields a $(1-\alpha)$-level confidence interval for $\mathrm{ATE}_P$ whose length is bounded above by the sum of the lengths of the two individual intervals; see Proposition~\ref{prop:ATE-upper-bound} for a precise statement.

In order to assess the optimality of our intervals, for $\mathcal{P} \subseteq \mathcal{P}(\mathbb{R}^{2})$ and $\mathcal{I} \subseteq [0,1]^2$, we consider the class $\widehat{\mathcal{CI}}^{\mathrm{ATE}}(\mathcal{P}, \mathcal{I})$ of confidence intervals $\hat{\mathrm{CI}}$ that satisfy the coverage guarantee
\begin{align*}
    \inf_{P\in\mathcal{P}}\; \inf_{(q, \eta) \in \mathcal{I}}\; \inf_{M\in\mathcal{M}_2(P,q,\eta)}\; \mathbb{P}_M\bigl(\mathrm{ATE}_P\in\hat{\mathrm{CI}}(Z_1,\ldots,Z_n)\bigr) \geq 1-\alpha.
\end{align*}
In analogous fashion to Section~\ref{sec:nonparametric}, we define the minimax length of these intervals as
\begin{align}
L_{n,q,\eta}^{\mathrm{ATE}}\bigl(\widehat{\mathcal{CI}}^{\mathrm{ATE}}(\mathcal{P},\mathcal{I})\bigr) \coloneqq \inf\Bigl\{r>0 : \inf_{\hat{\mathrm{CI}} \in \widehat{\mathcal{CI}}^{\mathrm{ATE}}(\mathcal{P},\mathcal{I})}\; \sup_{P\in\mathcal{P}}\; \sup_{M \in \mathcal{M}_2(P,q,\eta)} \mathbb{P}_M \bigl(|\hat{\mathrm{CI}}(Z_1,\ldots,Z_n)| \geq r\bigr) \leq \alpha \Bigr\}. \label{eq:minimax-quantile-causal}
\end{align}
The quantity above is bounded below by the minimax length given oracle knowledge of $\theta_0$ (or $\theta_1$), and hence inherits our one-dimensional lower bounds in Sections~\ref{sec:gaussian} and~\ref{sec:nonparametric}.  Proposition~\ref{prop:ATE-lower-bound} makes this precise.   
\begin{Proposition}
    \label{prop:ATE-lower-bound}
    Let $\mathcal{P} \subseteq \mathcal{P}(\mathbb{R}^2)$ be such that, writing $\mathcal{P}_0 \coloneqq \{P_0 : P\in\mathcal{P}\}$ and $\mathcal{P}_1 \coloneqq \{P_1 : P\in\mathcal{P}\}$, we have $\{P_0 \otimes P_1:P_0 \in \mathcal{P}_0,P_1 \in \mathcal{P}_1\} \subseteq \mathcal{P}$. Let $\mathcal{E}_n'\coloneqq \{(\epsilon_1,\epsilon_2) \in [0,1]^2 : \epsilon_1+\epsilon_2 \leq 1-\frac{1}{2n}\}$ and let 
    \begin{align*}
        \mathcal{I}_n'\coloneqq \Bigl\{(q,\eta)\in[0,1]^2 : \bigl((1-\eta)q + \eta\bigr) \wedge \bigl((1-\eta)(1-q) + \eta\bigr) \leq 1-\frac{1}{2n} \Bigr\}.
    \end{align*}
    Then, for any $n\in\mathbb{N}$ and $(q,\eta)\in\mathcal{I}$, we have
    \begin{align*}
        L_{n,q,\eta}^{\mathrm{ATE}} \bigl(\widehat{\mathcal{CI}}^{\mathrm{ATE}}(\mathcal{P},\mathcal{I}_n')\bigr) \geq L_{n,(1-\eta)q,\eta}\bigl(\hat{\mathcal{CI}}(\mathcal{P}_0, \mathcal{E}_n')\bigr) \vee L_{n,(1-\eta)(1-q),\eta}\bigl(\hat{\mathcal{CI}}(\mathcal{P}_1, \mathcal{E}_n')\bigr).
    \end{align*}
\end{Proposition}
The condition on $\mathcal{P}$ in Proposition~\ref{prop:ATE-lower-bound} is a closure property under marginal products: it asks that whenever $P \in \mathcal{P}$ has marginals $P_0$ and $P_1$, the product distribution $P_0 \otimes P_1$ belongs to $\mathcal{P}$.  By combining the lower bounds from Sections~\ref{sec:gaussian} and~\ref{sec:nonparametric} with Proposition~\ref{prop:ATE-lower-bound}, we now show that in several common examples, it is generally impossible to adapt to the proportion of unmeasured confounding\footnote{Here, we adopt the terminology of~\cite{bonvini2022sensitivity}; see Section~\ref{sec:relation-sensitivity-analysis} for a precise connection.}.  
\begin{itemize}
    \item \textit{Finite covariance:} Let $\mathcal{P}_{2,\mathrm{FV}} \coloneqq \{P\in\mathcal{P}(\mathbb{R}^2) : \| \mathrm{Cov}(P) \|_{\mathrm{max}} \leq \sigma_{\max}\}$.  Then, applying Proposition~\ref{prop:ATE-lower-bound} in conjunction with Proposition~\ref{thm:finite-variance-lb} yields
    \[
    \liminf_{n \rightarrow \infty} L_{n,q,\eta}^{\mathrm{ATE}} \bigl(\widehat{\mathcal{CI}}^{\mathrm{ATE}}(\mathcal{P}_{2,\mathrm{FV}},\mathcal{I}_n')\bigr) \geq \sigma_{\max} \max\Biggl\{\sqrt{\frac{\eta + (1 - \eta) q}{1 - \eta - (1 \!-\! \eta) q}},   \sqrt{\frac{\eta + (1 - \eta) (1 - q)}{1 - \eta - (1 - \eta) (1 - q)}}\Biggr\}.
    \]
    Taking $\eta = 0$ in the above expression corresponds to an experiment with no unmeasured confounding, and yet 
    \[
       \liminf_{n \rightarrow \infty}\; L_{n,q,0}^{\mathrm{ATE}} \bigl(\widehat{\mathcal{CI}}^{\mathrm{ATE}}(\mathcal{P}_{2,\mathrm{FV}},\mathcal{I}_n')\bigr) \geq \sigma_{\max} \cdot \max\biggl\{ \sqrt{\frac{q}{1 - q}},   \sqrt{\frac{1 - q}{q}}\biggr\} \geq \sigma_{\max}.
    \]
    In other words, even if there is no unmeasured confounding, any adaptive confidence interval must be of at least constant length.
    \item \textit{Gaussian with unknown covariance:}
    Let $\mathcal{P}_{2,\mathrm{G}}\coloneqq \{N_2(\theta,\Sigma): \theta\in\mathbb{R}^2,\, \|\Sigma\|_{\max} \leq \sigma_{\max}\}$. Then, by Proposition~\ref{prop:ATE-lower-bound} and the lower bound proof of Theorem~\ref{thm:gaussian-minimax-rates}\emph{(d)} (see Proposition~\ref{prop:all-unknown-lb}), we find that
    \begin{align*}
    \liminf_{n \rightarrow \infty} \; L_{n,q,\eta}^{\mathrm{ATE}} \bigl(\widehat{\mathcal{CI}}^{\mathrm{ATE}}(\mathcal{P}_{2,\mathrm{G}},\mathcal{I}_n')\bigr) &\gtrsim \sigma_{\max} \cdot \max\biggl\{\eta + (1 - \eta) q,\, \eta + (1 - \eta) (1 - q), \\
    &\hspace{1cm}\sqrt{\log\Bigl(\frac{1}{(1 - \eta)(1 - q)}\Bigr)} - 1,\, \sqrt{\log\biggl(\frac{1}{(1 - \eta) q}\biggr)} - 1\biggr\}.
    \end{align*}
    Again, considering the case with no unmeasured confounding when $\eta = 0$, we find that any adaptive confidence interval must have length at least
    \[
    \liminf_{n \rightarrow \infty} \; L_{n,q,0}^{\mathrm{ATE}} \bigl(\widehat{\mathcal{CI}}^{\mathrm{ATE}}(\mathcal{P}_{2,\mathrm{G}},\mathcal{I}_n')\bigr) \gtrsim \sigma_{\max} \cdot \max\biggl\{q, 1 - q, \sqrt{\log\biggl(\frac{1}{1 - q}\biggr)} - 1, \sqrt{\log\biggl(\frac{1}{q}\biggr)} - 1\biggr\}.
    \]
    Thus, even with the strong assumption of Gaussianity of the base distribution, any adaptive confidence interval must be of at least constant length.  
    \item \textit{Gaussian with known, diagonal covariance: }  Let $\Sigma\in\mathbb{R}^{2\times 2}$ be diagonal and let $\mathcal{P}_{\Sigma,\mathrm{G}} \coloneqq \{N_2(\theta,\Sigma) : \theta\in\mathbb{R}^2\}$. Proposition~\ref{prop:ATE-lower-bound} and Theorem~\ref{thm:gaussian-minimax-rates}\emph{(c)} yield
    \begin{align*}
    L_{n,q,\eta}^{\mathrm{ATE}} \bigl(\widehat{\mathcal{CI}}^{\mathrm{ATE}}(\mathcal{P}_{\Sigma,\mathrm{G}},\mathcal{I}_n')\bigr) &\gtrsim \max\Biggl\{\frac{\Sigma_{11}}{\sqrt{n}} + \frac{\Sigma_{11}\log\bigl(\frac{1}{(1 - \eta)(1 - q)}\bigr)}{\sqrt{1 \vee \log\bigl\{n\bigl(\eta + (1 - \eta)q\bigr)^2\bigr\}}},\\
    &\hspace{3cm} \frac{\Sigma_{22}}{\sqrt{n}} + \frac{\Sigma_{22}\log\bigl(\frac{1}{(1 - \eta)q}\bigr)}{\sqrt{1 \vee \log\bigl\{n\bigl(\eta + (1 - \eta)(1 -q)\bigr)^2\bigr\}}}\Biggr\}.
    \end{align*}
    Once more simplifying to the case $\eta = 0$ yields
    \[    
        L_{n,q,0}^{\mathrm{ATE}} \bigl(\widehat{\mathcal{CI}}^{\mathrm{ATE}}(\mathcal{P}_{\Sigma,\mathrm{G}},\mathcal{I}_n')\bigr) \gtrsim \max\biggl\{\frac{\Sigma_{11}}{\sqrt{n}} + \frac{\Sigma_{11}\log\bigl(\frac{1}{1 - q}\bigr)}{\sqrt{1 \vee \log(nq^2)}}, \frac{\Sigma_{22}}{\sqrt{n}} + \frac{\Sigma_{22}\log\bigl(\frac{1}{q}\bigr)}{\sqrt{1 \vee \log\bigl(n(1 -q)^2\bigr)}}\biggr\}.
    \]
    As in the previous two cases, we see that it is impossible to adapt to $\eta = 0$ (where one would expect an interval of length $1/\sqrt{n}$).
\end{itemize}
The second example above reveals a distinction between identification and inference in this adaptive setting.  In particular, while necessary and sufficient conditions under which the treatment effect is identifiable have recently been characterised~\citep{cai2025makes}, we saw in this example that the minimax length of any valid, adaptive confidence interval in this setting cannot converge to zero.  On the other hand, it is possible in this example to estimate the average treatment effect consistently, so in particular it is identifiable.  Indeed,~\citet[][Theorem 7]{ma2024estimation} show that their average of extremes estimator $\widehat{\mathrm{ATE}}_P$ of $\mathrm{ATE}_P$, which does not require the knowledge of any of the parameters $\eta, q, \Sigma$, satisfies
\[
\bigl \lvert \widehat{\mathrm{ATE}}_P - \mathrm{ATE}_P \bigr \rvert \lesssim_{\alpha} \sigma_{\max}\cdot\Biggl\{ \frac{\log\bigl(1 + \frac{6\eta}{(1 - q)(1 - \eta)}\bigr)}{\sqrt{\log\bigl(n(1 - q)(1-\eta)\bigr)}} + \frac{\log\bigl(1 + \frac{6\eta}{q(1 - \eta)}\bigr)}{\sqrt{\log\bigl(nq(1-\eta)\bigr)}}\Biggr\},
\]
with probability at least $1-\alpha$ when $\alpha\in[4e^{-n(q\wedge(1-q))(1-\eta)/8}, 1]$.  

Finally, we emphasise that each of the lower bounds in the examples above is matched by applying our optimal adaptive univariate confidence interval for the corresponding class of distributions to obtain confidence intervals for each of the counterfactual means $\theta_0$ and $\theta_1$ and then taking the Minkowski difference to combine these intervals.  Thus, in the third example above, even though we are unable to adapt to $\eta=0$ at the parametric rate, we can still obtain consistent adaptive confidence intervals with a logarithmic rate.

\subsection{Relation to sensitivity analysis} \label{sec:relation-sensitivity-analysis}
In order to connect the model introduced here to related notions in the literature, we consider the more standard setting in which we observe independent copies of $(Y, A)$, where $Y = A Y(1) + (1 - A) Y(0)$ denotes the response variable, $A$ denotes a binary treatment indicator and $Y(0)$ and $Y(1)$ denote the potential outcomes.  

\paragraph{Marginal sensitivity condition.}
For $j,k \in \{0,1\}$, let $P\bigl(Y(j) \, | \, A=k\bigr)$ denote the conditional distribution of $Y(j)$ given that $A=k$.  Specialised to the covariate-free setting considered here, Tan's marginal sensitivity condition~\citep[Eq.~(9)]{tan2006distributional} posits that there exists $\Lambda \geq 1$ such that 
\begin{align}
\frac{1}{\Lambda} \leq \frac{\mathrm{d}P\bigl(Y(0) \,|\, A = 1\bigr)}{\mathrm{d}P\bigl(Y(0) \,|\, A = 0\bigr)} \leq \Lambda \quad \text{ and } \quad \frac{1}{\Lambda} \leq \frac{\mathrm{d}P\bigl(Y(1) \,|\, A = 1\bigr)}{\mathrm{d}P\bigl(Y(1) \,|\, A = 0\bigr)} \leq \Lambda. \label{eq:marginal-sensitivity}
\end{align} 

\begin{Lemma}
\label{Lemma:MaringalSensitivity}
    Any distribution in $\mathcal{M}_2(P, q, \eta)$ satisfies the marginal sensitivity condition~\eqref{eq:marginal-sensitivity} with  
    \[
    \Lambda = \biggl\{1 + \frac{\eta}{(1 - \eta)(1 - q)}\biggr\} \cdot \biggl\{1 + \frac{\eta}{(1 - \eta)q}\biggr\}.
    \]
\end{Lemma}
An immediate consequence of Lemma~\ref{Lemma:MaringalSensitivity} is that the lower bounds in the previous section can be applied to distributions satisfying the marginal sensitivity condition.

\paragraph{Rosenbaum's sensitivity condition.} The connection between Rosenbaum's sensitivity condition~\citep{rosenbaum87sensitivity,yadlowsky2022bounds} and the marginal sensitivity condition discussed in the previous paragraph is specified by~\citet[Proposition 7.1]{zhao2019sensitivity}.  
A similar relationship holds between our model and Rosenbaum's sensitivity condition.

\paragraph{Proportion of unmeasured confounding.} In the special case of the covariate-free causal inference setting above, the model~\eqref{eq:causal-model} is identical to the proportion of unmeasured confounding model proposed by~\citet{bonvini2022sensitivity}.  Indeed, \citet{bonvini2022sensitivity} consider a counterfactual distribution 
\[
\mathbb{Q} = (1 - \eta) \mathbb{Q}_0 + \eta \mathbb{Q}_1,
\]  
where $\mathbb{Q}_0$ is an unconfounded distribution for a triple $(Y_0, Y_1, A)$ in which $A \indep \{Y_0, Y_1\}$ and $\mathbb{Q}_1$ is an arbitrarily confounded distribution on $(Y_0, Y_1, A)$ that places no independence assumptions between~$A$ and $(Y_0, Y_1)$.
Their observed data are $(Y, A)$ where $Y \coloneqq (1 - A) Y_0 + AY_1$.  Now consider the bijective map $\psi: \mathbb{R} \times \{0, 1\} \rightarrow (\mathbb{R} \times \{\star\}) \cup (\{\star\} \times \mathbb{R})$ defined by
\[
\psi(Y, A) \coloneqq \bigl((1 - A) \ostar Y, A \ostar Y\bigr)^{\top} = \Omega(A) \ostar (Y_0, Y_1)^{\top},
\]
where $\Omega(0) \coloneqq (1, 0)^{\top}$ and $\Omega(1) \coloneqq (0, 1)^{\top}$.  By construction, with probability $1 - \eta$, we have $\Omega(A) \indep (Y_0, Y_1)$ and with probability $\eta$, $\Omega(A)$ can depend arbitrarily on the pair $(Y_0, Y_1)$.  Letting $q \coloneqq \mathbb{P}_{\mathbb{Q}_0}(A=1)$ in the unconfounded portion of the mixture, it thus follows that $\mathsf{Law}\bigl(\psi(Y, A)\bigr) \in \mathcal{M}_2(P,q,\eta)$, where $P$ denotes the distribution of $Y$.  The converse statement holds by considering the inverse map $\psi^{-1}$. 

\section{Numerical examples}
We study the performance of our confidence intervals on simulated and real data in this section. The code for our experiments is available at \url{https://github.com/kabirverchand/AdaptiveCI}.

\subsection{Synthetic Gaussian data}
We take the base distribution $P = N(0, 1)$ and observe data from
\begin{align*}
    (1-\epsilon_1-\epsilon_2)P + \epsilon_1 \delta_{\{\star\}} + \epsilon_2 Q,
\end{align*}
where $Q \coloneqq \mathrm{Law}(X \ostar \indi_{\{X\geq 1/2\}})$ when $X\sim P$.   Thus, when $\epsilon_2 > 0$, this mixture distribution is MNAR.  We apply the adaptive confidence intervals in Section~\ref{sec:gaussian-CI-construction-general}\emph{(c)} and~\emph{(d)}; in the latter case where the standard deviation $\sigma=1$ is unknown, we take $\sigma_{\max}=2$. Our experiments consider the following three settings:

\begin{enumerate}
    \item We fix $\epsilon_1=0$, $\epsilon_2=0.5$, vary the sample size from $n=10^2$ to $n=10^6$ and consider the settings of unknown $(\epsilon_1,\epsilon_2)$ with the standard deviation~$\sigma$ being either known or unknown.  In Figure~\ref{fig:gaaussian-CI-vary-n}, we plot our adaptive confidence intervals (based on one realisation).  These intervals cover the true population mean, which is identifiable, although they are not centred at this quantity.  When $\sigma$ is known, our intervals shrink as $n$ increases, albeit at logarithmic rate, which is optimal (Theorem~\ref{thm:gaussian-minimax-rates}\emph{(c)}). When $\sigma$ is unknown, the lengths of our intervals do not converge to zero, which is unavoidable by Theorem~\ref{thm:gaussian-minimax-rates}\emph{(d)}.  On the other hand, the naive Gaussian intervals centred at the sample mean, which  ignore the missing data aspect of the model, do not cover the true population mean in the large majority of realisations.

    \item We fix $n=1000$, vary $(\epsilon_1,\epsilon_2) \in \{(\epsilon_1',\epsilon_2')\in[0,1]^2 : \epsilon_1'+\epsilon_2' \leq 0.9\}$ and again consider unknown $(\epsilon_1,\epsilon_2)$ with $\sigma$ known or unknown. For each value of $(\epsilon_1,\epsilon_2)$, we compute the average length of the confidence intervals over 100 repetitions, and plot corresponding heat maps in Figure~\ref{fig:heat-map}.  The average length is  roughly constant for fixed $\epsilon_1+\epsilon_2$, in line with our minimax rates in Theorem~\ref{thm:gaussian-minimax-rates}\emph{(c)} and~\emph{(d)}.

    \item We fix $n=1000$, take $\epsilon_1\in\{0.05,0.5\}$ and vary $\epsilon_2$. In Figure~\ref{fig:fixed-eps-1}, we plot realisations of our adaptive confidence intervals with knowledge of $\sigma$ but no knowledge of $\epsilon_1,\epsilon_2$. We see that for fixed $\epsilon_1$, the lengths of the intervals decreases as $\epsilon_2$ decreases, again in line with Theorem~\ref{thm:gaussian-minimax-rates}\emph{(c)}. 
\end{enumerate}

\begin{figure}
    \centering
    \begin{subfigure}{0.85\linewidth}
        \includegraphics[width=\linewidth]{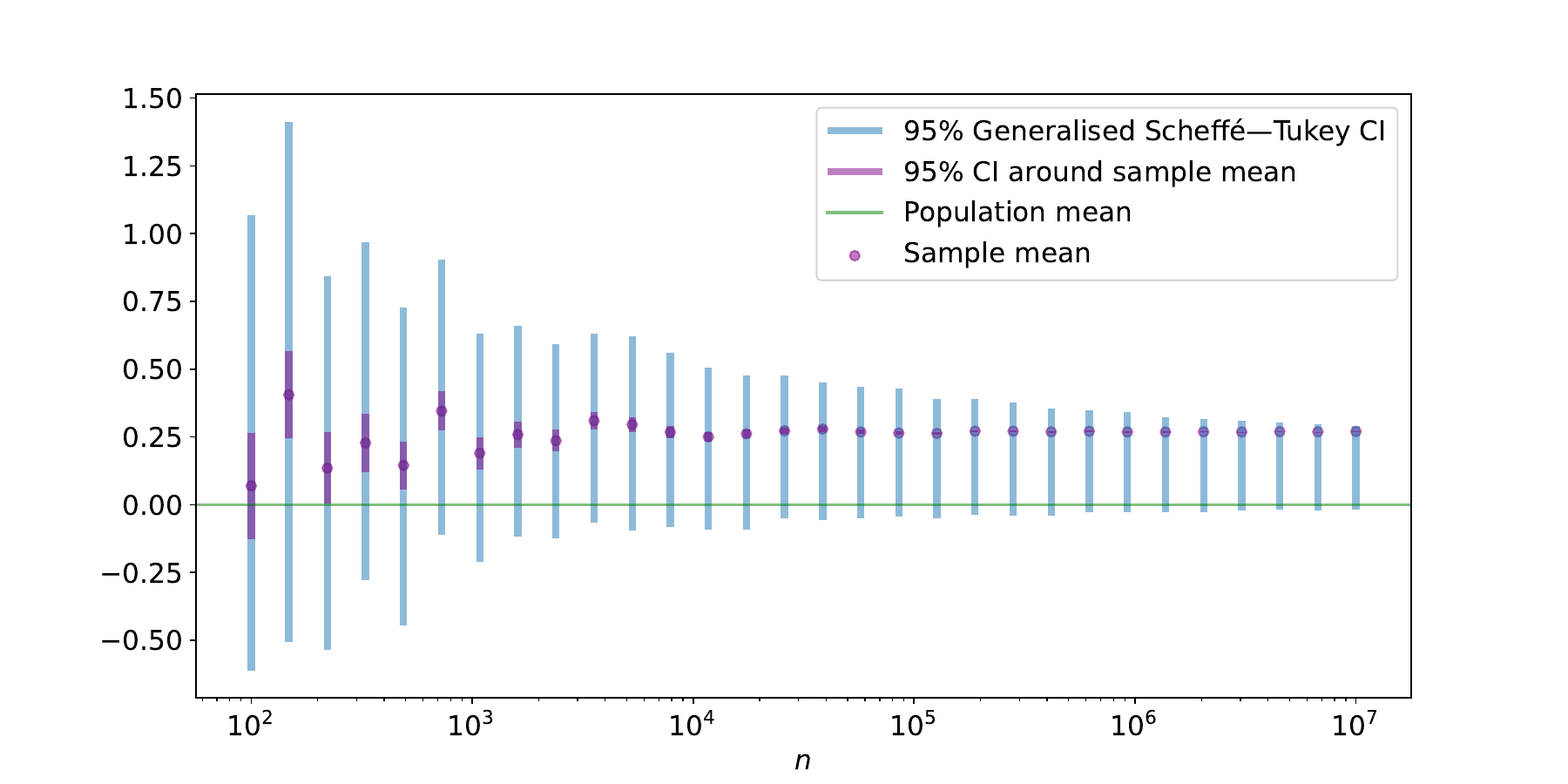}  
        \caption{Known $\sigma$, unknown $\epsilon_1,\epsilon_2$.}
    \end{subfigure}

    \begin{subfigure}{0.85\linewidth}
        \includegraphics[width=\linewidth]{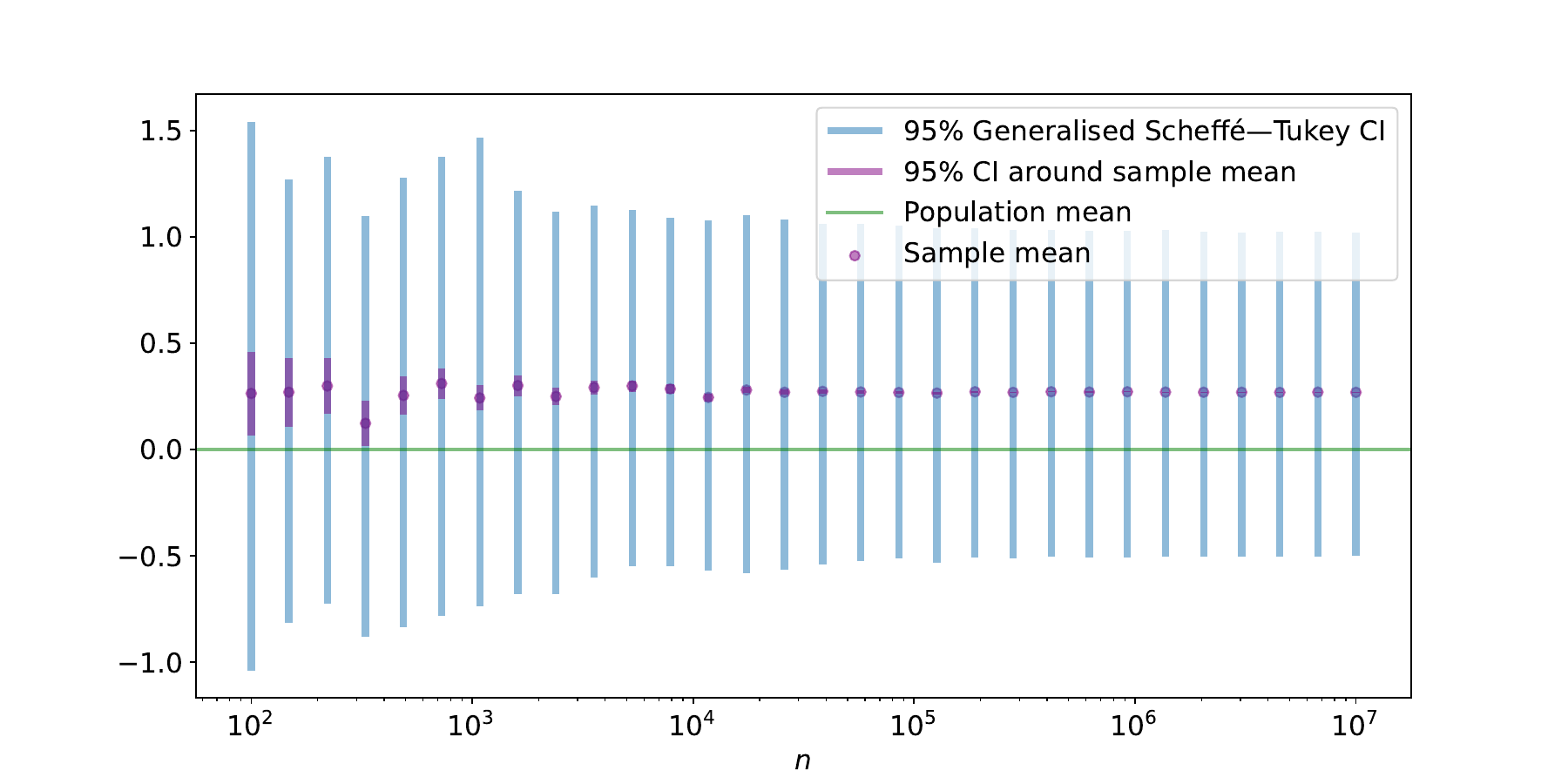}    
        \caption{Unknown $\sigma,\epsilon_1,\epsilon_2$.}
    \end{subfigure}

    \caption{Realisations of our adaptive confidence intervals and naive Gaussian intervals centred at the sample mean at different sample sizes $n$.}
    \label{fig:gaaussian-CI-vary-n}
\end{figure}

\begin{figure}
    \centering
    \begin{subfigure}{0.72\linewidth}
        \includegraphics[width=\linewidth]{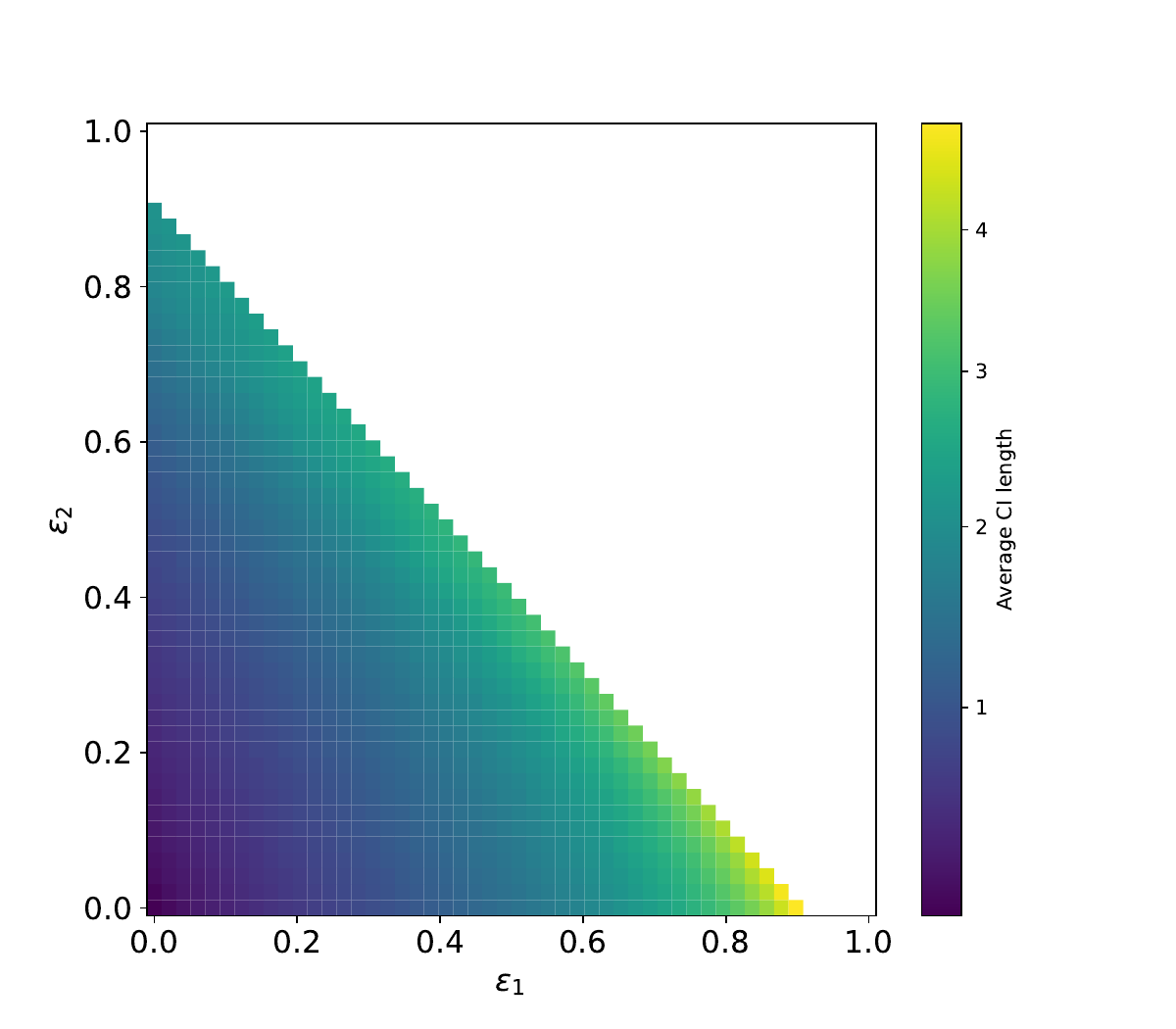}  
        \caption{Known $\sigma$, unknown $\epsilon_1,\epsilon_2$.}
    \end{subfigure}

    \begin{subfigure}{0.72\linewidth}
        \includegraphics[width=\linewidth]{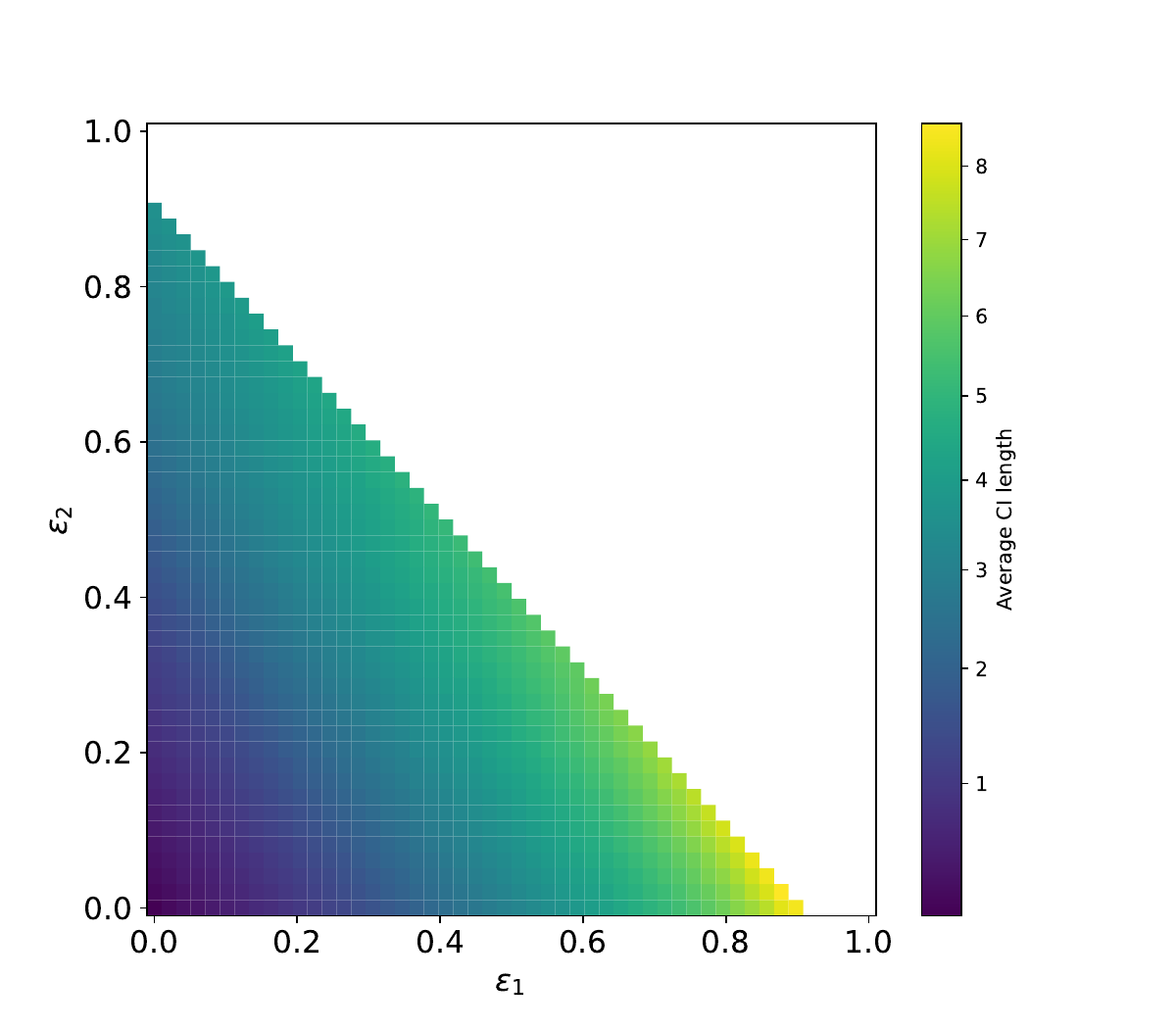}    
        \caption{Unknown $\sigma,\epsilon_1,\epsilon_2$.}
    \end{subfigure}

    \caption{Heat map of the average length of our adaptive confidence intervals.}
    \label{fig:heat-map}
\end{figure}

\begin{figure}
    \centering
    \begin{subfigure}{0.85\linewidth}
        \includegraphics[width=\linewidth]{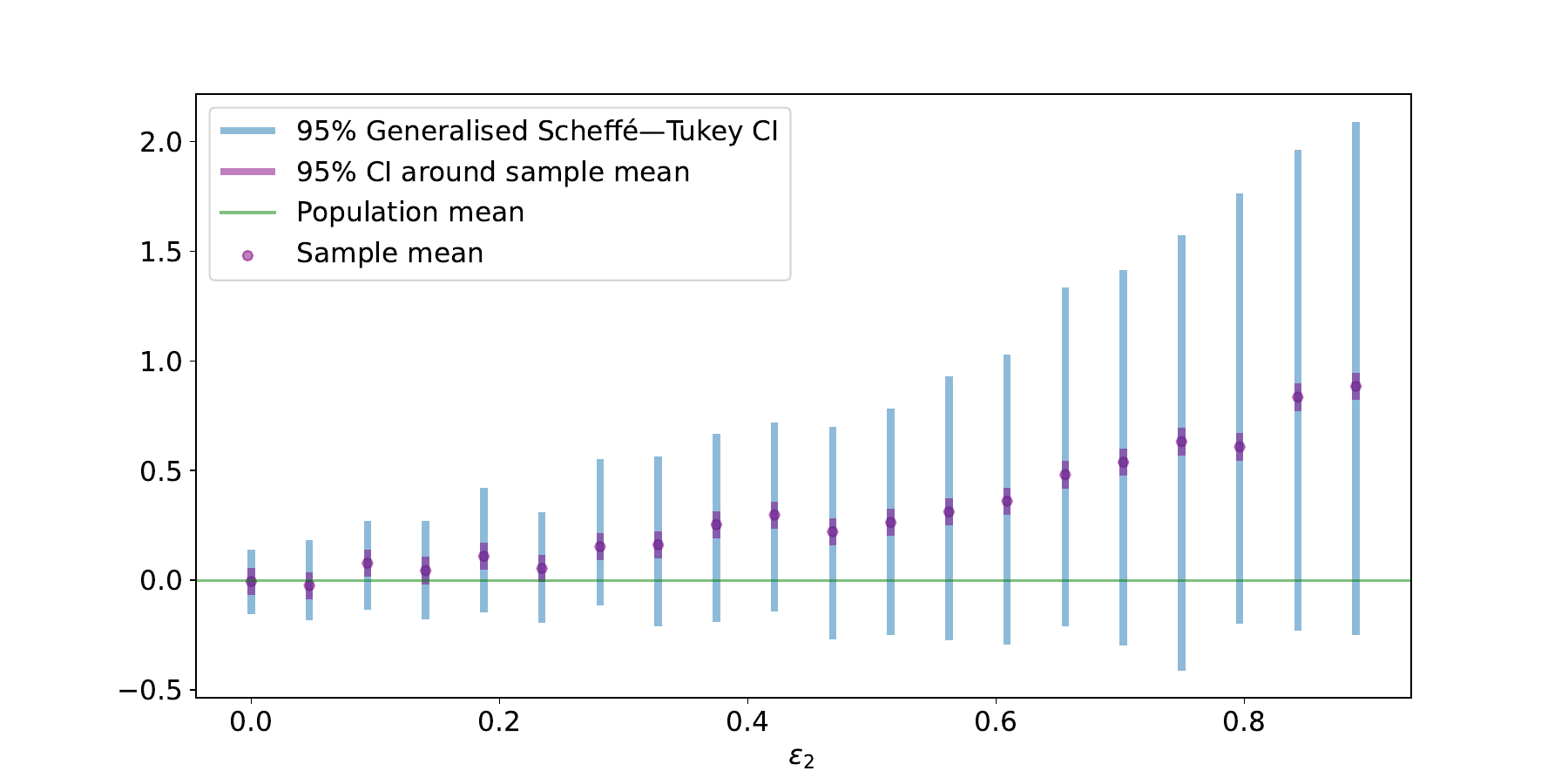}  
        \caption{$\epsilon_1=0.05$.}
    \end{subfigure}

    \begin{subfigure}{0.85\linewidth}
        \includegraphics[width=\linewidth]{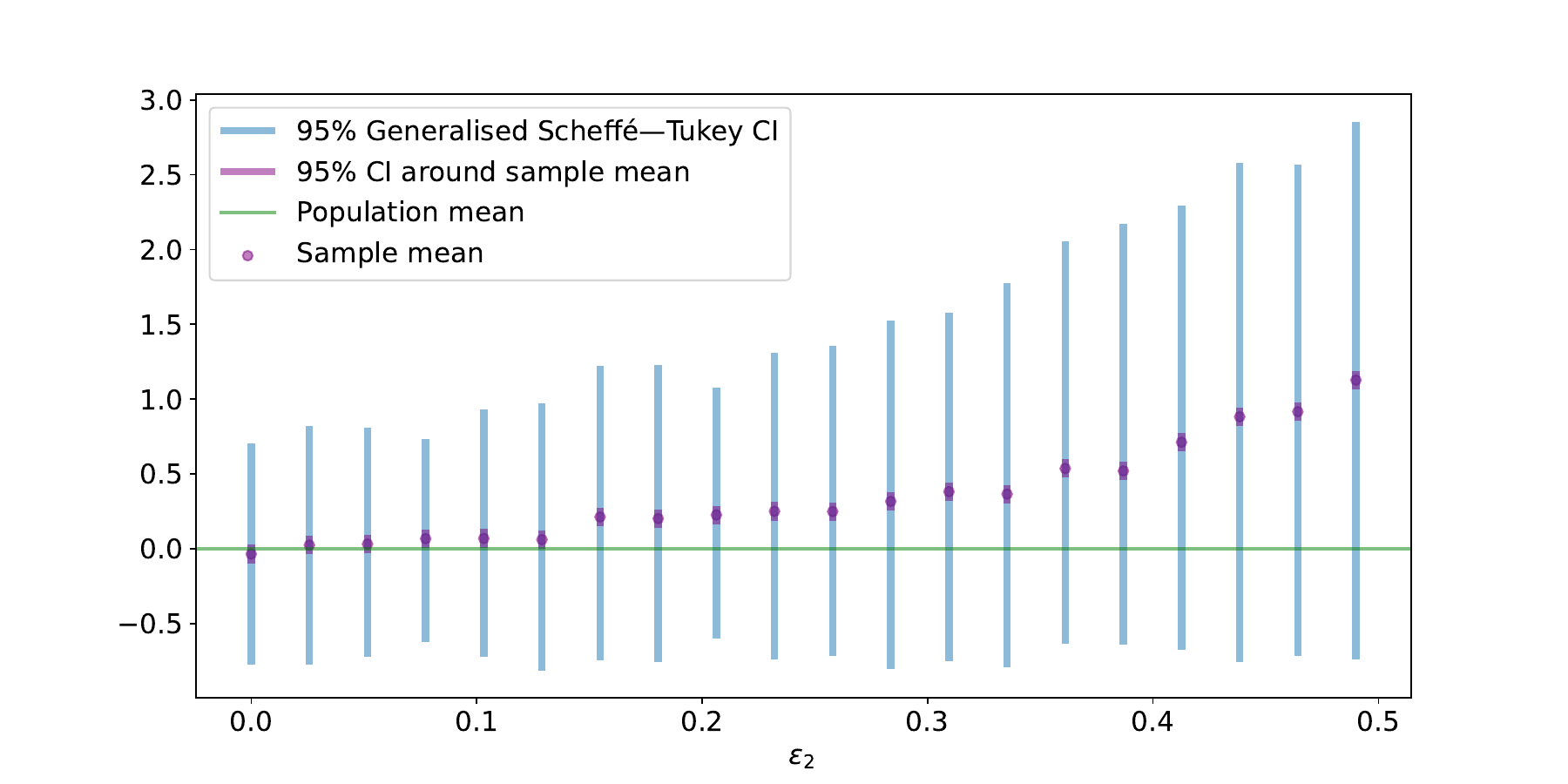}    
        \caption{$\epsilon_1=0.5$.}
    \end{subfigure}

    \caption{Realisations of our adaptive confidence intervals with known $\sigma$ but unknown $\epsilon_1,\epsilon_2$, as $\epsilon_2$ varies.}
    \label{fig:fixed-eps-1}
\end{figure}

\subsection{\texorpdfstring{Revisiting the test score data of~\citet{miratrix2018shape}}{Revisiting the test score data of Miratrix, Wager and Zubizarreta (2018)}}
In this section, we revisit the analysis of~\citet[Section 3]{miratrix2018shape} concerning a sample of $n=1136$ Aymara students in Chile, $847$ of whom took the Prueba de Selecci\'on Universitaria (PSU) mathematics test in 2008.  The minimum and maximum possible scores on the test were 150 and 850 respectively.  The goal is to make inference about the Aymara population mean score.  The analysis of \citet{miratrix2018shape} assumes that $\pi_{\max}/\pi_{\min} \leq 9$, where $\pi_{\min}$ and $\pi_{\max}$ are defined such that $\pi_{\min} \leq \mathbb{P}(\Omega=1 \,|\, X) \leq \pi_{\max}$ almost surely. While this can be ensured by taking $\epsilon_1 + \epsilon_2 \leq 8/9$ in our model~\eqref{eq:model-mcar-mnar} (see~\eqref{eq:realisability-ineq}), our boundedness confidence interval in Section~\ref{sec:adaptation-tail} adapts to all possible values of $\epsilon_1,\epsilon_2$ so does not require an assumption of this form.  Our confidence interval under symmetry constraints in Section~\ref{sec:adaptation-shape} does require $\epsilon_1+\epsilon_2 \leq 1-4\log_2(8/\alpha)/n$, but this still allows $\pi_{\max}/\pi_{\min}$ to be as large as $38.78$ for a $95\%$ confidence interval.  Our nonparametric $95\%$ confidence intervals are $[447,567]$ under symmetry and $[338,645]$ under boundedness (using the fact that the scores are bounded by $150$ and $850$).  Thus, the symmetry assumption does permit a considerably shorter interval.  For comparison, the corresponding $95\%$ confidence interval obtained by \citet{miratrix2018shape} with the bound on $\pi_{\max}/\pi_{\min}$ and under symmetry was\footnote{This interval was obtained by changing the value of $\alpha$ from its default to $0.05$ in the \texttt{R} code provided by the authors.} $[427,577]$.

\subsection{Adaptive treatment effect estimation}
Here, we consider the birthweight data of~\citet{almond2005costs}.  This study collected data from $n=4642$ mothers as well as weights of their children at birth.  We are interested in understanding whether or not the weight of a child at birth (in grams) is affected by whether or not the mother smokes.  Our population quantity of interest is therefore the difference in the average birthweights of children born to mothers who do not smoke (3778 of them) and those who do (the remaining 864).  To illustrate the difficulty of adapting to unobserved confounding, we assume that the underlying population of birthweights is Gaussian and the standard deviations is known (and to be the empirical standard deviation).  A point estimate of the average treatment effect (obtained as the midpoint of our confidence interval) is $\widehat{\mathrm{ATE}}_P = 284 \, \mathrm{grams}$, which suggests a fairly strong negative effect of smoking on birthweight and is consistent with findings in the literature~\citep{cattaneo2010efficient}.  This point estimate is robust to levels of confounding close to one.  On the other hand, our adaptive confidence interval is $(-641, 1209)$, which is rather wide (despite being minimax length-optimal).  Thus, even though our point estimate suggests a robust effect, an adaptive confidence interval tells a much more cautious story.  

\medskip
\noindent \textbf{Acknowledgements}: The first and last author were supported by the latter's European Research Council Advanced Grant 101019498. The third author was supported in part by NSF Grants ECCS-2216912 and DMS-2310769, and an Alfred Sloan fellowship.

\vskip 0.2in
\bibliographystyle{custom}
\bibliography{bibliography}
\end{sloppypar}

\appendix

\section{The extended space and realisable contamination model} \label{sec:properties-of-R-star}
In this section, we introduce some properties of the extended space $\mathbb{R}_{\star}$ and the realisable contamination model. We refer the readers to \citet[Section 2]{ma2024estimation} for more general settings.

Let $\mathcal{T}$ denote the standard topology on $\mathbb{R}$ and let $\mathcal{B}$ denote the Borel $\sigma$-algebra of $\mathbb{R}$. Then the space $\mathbb{R}_{\star} = \mathbb{R} \cup \{\star\}$ is equipped with the topology $\mathcal{T}_{\star} \coloneqq \mathcal{T} \bigcup \bigl\{A\cup\{\star\} : A\in\mathcal{T}\bigr\}$ and corresponding Borel $\sigma$-algebra $\mathcal{B}_{\star} \coloneqq \mathcal{B} \bigcup \bigl\{B\cup\{\star\} : B\in\mathcal{B}\bigr\}$. For a measure $\mu$ on $\mathbb{R}$, we define the extended measure $\mu_{\star}$ on $\mathbb{R}_{\star}$ by
\begin{align*}
    \mu_{\star}(B) \coloneqq \mu(B) \quad\text{and}\quad \mu_{\star}(B\cup\{\star\}) \coloneqq \mu(B) + 1
\end{align*}
for $B\in\mathcal{B}$. In the appendix, we will use $\lambda_{\star}$ for the extension of the Borel measure to $\mathbb{R}_{\star}$.

We have the following characterisation for the realisable contamination model.
\begin{Lemma} \label{lemma:realisability-characterisation}
    Let $P$ be a distribution on $\mathbb{R}$ that is absolutely continuous with respect to a measure $\mu$, and let $\epsilon_1,\epsilon_2\in[0,1]$ be such that $\epsilon_1+\epsilon_2\leq 1$. Then $M\in\mathcal{M}(P,\epsilon_1,\epsilon_2)$ if and only if $M$ is absolutely continuous with respect to $\mu_{\star}$ and
    \begin{align}
        (1-\epsilon_1-\epsilon_2)\frac{\mathrm{d}P}{\mathrm{d}\mu}(x)  \leq \frac{\mathrm{d}M}{\mathrm{d}\mu_{\star}}(x) \leq (1-\epsilon_1)\frac{\mathrm{d}P}{\mathrm{d}\mu}(x) \label{eq:realisability-characterisation}
    \end{align}
    for all $x\in\mathbb{R}$.
\end{Lemma}
If we take $\mu=P$, then~\eqref{eq:realisability-characterisation} yields
\begin{align*}
    1-\epsilon_1-\epsilon_2 \leq \frac{\mathrm{d}M}{\mathrm{d}P_{\star}}(x) \leq 1-\epsilon_1 
\end{align*}
for all $x\in\mathbb{R}$.
\begin{proof}
    This lemma follows from \citet[Proposition 4]{ma2024estimation}, with $\epsilon=\epsilon_2$ and $q=\frac{1-\epsilon_1-\epsilon_2}{1-\epsilon_2}$ therein.
\end{proof}

\section{The simplified contamination model in Section~\ref{sec:warm-up}} \label{sec:proofs-for-warmup}
In this section, we provide valid confidence intervals with length guarantees for the simplified contamination model~\eqref{eq:model-mnar}, under four different cases, depending on whether $\sigma$ and $\epsilon$ are known.

\subsection{Known \texorpdfstring{$\sigma$ and $\epsilon$}{sigma and epsilon}} \label{sec:all-known}
\begin{Proposition}
\label{prop:all-param-known}
    Let $n\in\mathbb{N}$, $\sigma>0$, $\alpha\in(0,1]$, $\epsilon \in \bigl[0,1-\frac{4\log_2(8/\alpha)}{n}\bigr]$, and $Z_1,\ldots,Z_n \overset{\mathrm{iid}}{\sim} (1-\epsilon)N(\theta,\sigma^2)+\epsilon Q$ where $Q\in\mathsf{MNAR}_{N(\theta,\sigma^2)}$.
    
    (a) The confidence intervals $\{\widehat{\mathrm{CI}}_{n,\sigma,\alpha}^{(1)}(t):t \geq 0\}$ defined in~\eqref{eq:GST-warmup} have simultaneous coverage for~$\theta$, i.e.
    \begin{align*}
        \mathbb{P}\biggl( \theta\in \bigcap_{t\geq 0} \widehat{\mathrm{CI}}^{(1)}_{n,\sigma,\alpha}(t) \biggr) \geq 1-\alpha.
    \end{align*}
    
    (b) Let $t_{n,\epsilon,\alpha} \coloneqq \sqrt{0\vee\{\log(n\epsilon^2)-\log\log(4/\alpha)\}/2}$. There exists a universal constant $C>0$ such that if $\alpha \in \bigl[ \frac{C}{n}, 1\bigr]$, then with probability at least $1-\alpha/2$, we have
    \begin{align*}
        \bigl| \hat{\mathrm{CI}}^{(1)}_{n,\sigma,\alpha}(t_{n,\epsilon,\alpha}) \bigr| \leq C_{\alpha}\sigma\biggl\{\frac{1}{\sqrt{n}} + \frac{\log\bigl(1/(1-\epsilon)\bigr)}{\sqrt{1\vee \log(n\epsilon^2)}}\biggr\},
    \end{align*}
    where $C_{\alpha} \geq 1$ depends only on $\alpha$.
\end{Proposition}
\begin{proof}
    \emph{(a)} For $i\in[n]$, write\footnote{Here we interpret $0\cdot\star\coloneqq 0$.} $Z_i \coloneqq L_i X_i + (1-L_i)(X_i \ostar \Omega_i)$ where $(X_1,L_1,\Omega_1),\ldots,(X_n,L_n,\Omega_n)$ are independent triples with $X_i \sim N(\theta,\sigma^2)$, $\Omega_i$ a binary random variable such that $X_i \ostar \Omega_i \sim Q$ and $L_i \sim \mathrm{Bern}(1-\epsilon)$ independent of $(X_i,\Omega_i)$. In other words, if $L_i=1$, then $Z_i = X_i \sim N(\theta,\sigma^2)$ and if $L_i=0$, then $Z_i \sim Q$. Let $\mathcal{I} \coloneqq \{i\in[n] : L_i = 1\}$ so that $|\mathcal{I}| \sim \mathrm{Bin}(n, 1-\epsilon)$. By the multiplicative Chernoff bound \citep[Theorem~2.3(c)]{McDiarmid1998}, and since $1-\epsilon \geq \frac{4\log_2(8/\alpha)}{n}$, we deduce that
    \begin{align*}
        \mathbb{P}\bigl(|\mathcal{I}| \leq \log_2(8/\alpha)\bigr) \leq \exp\bigl(-9\log_2(8/\alpha)/8\bigr) \leq \frac{\alpha}{8}.
    \end{align*}
    Therefore,
    \begin{align}
        \mathbb{P}(Z_{\max} \leq \theta) \leq \mathbb{P}\Bigl(\max_{i\in\mathcal{I}}X_i \leq \theta\Bigr) &\leq \mathbb{P}\Bigl(\max_{i\in\mathcal{I}}X_i \leq \theta \,\Big|\, |\mathcal{I}| > \log_2(8/\alpha)\Bigr) + \mathbb{P}\Bigl(|\mathcal{I}| \leq \log_2(8/\alpha)\Bigr) \nonumber\\
        &\leq 2^{-\log_2(8/\alpha)} + \frac{\alpha}{8} \leq \frac{\alpha}{4}. \label{eq:Zmax<theta}
    \end{align}
    Similarly, we have $\mathbb{P}(Z_{\min} \geq \theta) \leq \frac{\alpha}{4}$.
    Thus, by Lemma~\ref{lemma:empirical-quantile-property} and the one-sided Dvoretzky--Kiefer--Wolfowitz--Masssart--Reeve (DKWMR) inequality \citep[][Corollary~2]{reeve2024short} and Eq.~\eqref{ineq:sandwich-q-1}, we deduce that
    \begin{align*}
        &\mathbb{P}\biggl\{Q_n^-\biggl(1-\Phi(t)+\sqrt{\frac{\log(4/\alpha)}{2n}}\biggr) < \theta - \sigma t \text{ for some } t\geq 0\biggr\}\\
        &\hspace{1cm} \leq \mathbb{P}\biggl\{\frac{1}{n}\sum_{i=1}^n \indi_{\{Z_i \leq \theta - \sigma t\}} \geq 1-\Phi(t)+\sqrt{\frac{\log(4/\alpha)}{2n}} \text{ for some } t\geq 0\biggr\} + \mathbb{P}(Z_{\max} \leq \theta) \leq \frac{\alpha}{2}.
    \end{align*}
    Similarly, by Lemma~\ref{lemma:empirical-quantile-property}, Eq.~\eqref{ineq:sandwich-q-1} and the one-sided DKWMR inequality, 
    \begin{align*}
        &\mathbb{P}\biggl\{Q_n^+\biggl(1-\Phi(t)+\sqrt{\frac{\log(4/\alpha)}{2n}}\biggr) > \theta + \sigma t \text{ for some } t\geq 0\biggr\} \leq \frac{\alpha}{2}.
    \end{align*}
    Combining these two bounds yields that
    \begin{align*}
        \mathbb{P}\biggl\{Q_n^+\biggl(1-\Phi(t)+\sqrt{\frac{\log(4/\alpha)}{2n}}\biggr) - \sigma t \leq \theta \leq Q_n^-\biggl(1-\Phi(t)+\sqrt{\frac{\log(4/\alpha)}{2n}}\biggr) + \sigma t \;\;\forall\; t\geq 0\biggr\} \geq 1-\alpha,
    \end{align*}
    as required.

    \medskip
    \emph{(b)} We have $\mathbb{P}(Z_1 \leq \theta + \sigma x) \geq (1-\epsilon)\bigl\{1-\Phi(-x)\bigr\}$ and $\mathbb{P}(Z_1 > \theta + \sigma x) \geq (1-\epsilon)\bigl\{1-\Phi(x)\bigr\}$ for all $x\in\mathbb{R}$. Thus, by the DKWMR inequality, the event
    \begin{align*}
        \mathcal{E}_1 &\coloneqq \biggl\{\frac{1}{n} \sum_{i=1}^n \indi_{\{Z_i \leq \theta + \sigma(b-t)\}} \geq (1-\epsilon)\bigl\{ 1 - \Phi(t-b)\bigr\} - \sqrt{\frac{\log(4/\alpha)}{2n}} \;\;\forall\; b,t\in\mathbb{R}\biggr\}\\
        &\hspace{2cm} \bigcap\biggl\{\frac{1}{n} \sum_{i=1}^n \indi_{\{Z_i > \theta + \sigma(t-b)\}} > (1-\epsilon)\bigl\{ 1 - \Phi(t-b)\bigr\} - \sqrt{\frac{\log(4/\alpha)}{2n}} \;\;\forall\; b,t\in\mathbb{R}\biggr\}
    \end{align*}
    has probability at least $1-\alpha/2$. Now suppose that, for some $t_*,b\geq 0$, we have
    \begin{align}
        (1-\epsilon)\frac{1 - \Phi(t_*-b)}{1-\Phi(t_*)} \geq 1 + 6e^{t_*^2}\sqrt{\frac{\log(4/\alpha)}{2n}}. \label{eq:sufficient-cond-gaussian-known-eps}
    \end{align}
    Then by Lemma~\ref{lemma:mills-ratio}, $1-\Phi(t_*) \geq \frac{2}{\sqrt{2\pi}} \cdot \frac{1}{t_*+\sqrt{4+t_*^2}} e^{-t_*^2/2} \geq \frac{1}{3}e^{-t_*^2}$. Thus, 
    \begin{align*}
        (1-\epsilon)\bigl\{1 - \Phi(t_*-b)\bigr\} - \sqrt{\frac{\log(4/\alpha)}{2n}} \geq 1 - \Phi(t_*) + \sqrt{\frac{\log(4/\alpha)}{2n}}.
    \end{align*}
    Therefore, on the event $\mathcal{E}_1$, we have $\frac{1}{n} \sum_{i=1}^n \indi_{\{Z_i \leq \theta + \sigma(b-t_*)\}} \geq  1 - \Phi(t_*) + \sqrt{\frac{\log(4/\alpha)}{2n}}$ as well as $\frac{1}{n} \sum_{i=1}^n \indi_{\{Z_i > \theta + \sigma(t_*-b)\}} > 1 - \Phi(t_*) + \sqrt{\frac{\log(4/\alpha)}{2n}}$. Thus by Lemma~\ref{lemma:empirical-quantile-property}, with probability at least $1-\alpha/2$, we have $Q_n^-\bigl(1-\Phi(t_*) + \sqrt{\frac{\log(4/\alpha)}{2n}}\bigr) \leq \theta + \sigma(b-t_*)$ and $Q_n^+\bigl(1-\Phi(t_*) + \sqrt{\frac{\log(4/\alpha)}{2n}}\bigr) > \theta + \sigma(t_*-b)$, so $|\hat{\mathrm{CI}}^{(1)}_{n,\sigma,\alpha}(t_*)| \leq 2\sigma b$. Now we take $t_* \equiv t_{n,\epsilon,\alpha} \coloneqq \sqrt{0\vee \{\log(n\epsilon^2)-\log\log(4/\alpha)\}/2}$ and assume throughout the proof that $C$ is large enough that $n\geq 400\log(4/\alpha)$.

    \textbf{Case 1:} Suppose that $\epsilon \leq \sqrt{\frac{\log(4/\alpha)}{n}}$, and let $b \coloneqq 20\sqrt{\frac{\log(4/\alpha)}{n}}$. In this case, we have $t_*=0$ and $b\leq 1$, so by Lemma~\ref{lemma:gaussian-cdf-ratio}\emph{(a)},
    \begin{align*}
        (1-\epsilon)\frac{1 - \Phi(t_*-b)}{1-\Phi(t_*)} \geq \biggl(1 - \sqrt{\frac{\log(4/\alpha)}{n}}\biggr) \biggl(1 + \frac{10}{\sqrt{e}}\sqrt{\frac{\log(4/\alpha)}{n}}\biggr) \geq 1 + 6e^{t_*^2}\sqrt{\frac{\log(4/\alpha)}{2n}}.
    \end{align*}
    This verifies~\eqref{eq:sufficient-cond-gaussian-known-eps}. Thus, in this case, we have $|\hat{\mathrm{CI}}^{(1)}_{n,\sigma,\alpha}(t_*)| \leq 40\sigma \sqrt{\frac{\log(4/\alpha)}{n}}$ with probability at least $1-\alpha/2$.

    \textbf{Case 2:} Suppose that $\sqrt{\frac{\log(4/\alpha)}{n}} < \epsilon \leq \frac{1}{20}$, and let $b\coloneqq 20\epsilon/(t_*+1)$. In this case, $t_*= \sqrt{\{\log(n\epsilon^2)-\log\log(4/\alpha)\}/2}$ and $b\leq 1$. Then by Lemma~\ref{lemma:gaussian-cdf-ratio}\emph{(a)},
    \begin{align*}
        (1-\epsilon)\frac{1 - \Phi(t_*-b)}{1-\Phi(t_*)} \geq (1-\epsilon)\biggl(1+\frac{10\epsilon}{\sqrt{e}}\biggr) &\geq 1+\biggl(\frac{10}{\sqrt{e}} - 1 - \frac{1}{2\sqrt{e}}\biggr)\epsilon \geq 1 + 6e^{t_*^2}\sqrt{\frac{\log(4/\alpha)}{2n}}.
    \end{align*}
    This verifies~\eqref{eq:sufficient-cond-gaussian-known-eps}. Thus, in this case, we have
    \begin{align}
        |\hat{\mathrm{CI}}^{(1)}_{n,\sigma,\alpha}(t_*)| \leq \frac{40\sigma\epsilon}{1+\sqrt{\{\log(n\epsilon^2) - \log\log(4/\alpha)\}/2}} \label{eq:CI-case-2}
    \end{align}
    with probability at least $1-\alpha/2$. Moreover, when $\log(n\epsilon^2) \leq 2\log\log(4/\alpha)$, we have
    \begin{align}
        1+\sqrt{\frac{\log(n\epsilon^2) - \log\log(4/\alpha)}{2}} \geq 1 \geq \sqrt{\frac{\log(n\epsilon^2)}{2\log\log(4/\alpha)}}, \label{eq:denom-lb-1}
    \end{align}
    and when $\log(n\epsilon^2) > 2\log\log(4/\alpha)$, we have
    \begin{align}
        1+\sqrt{\frac{\log(n\epsilon^2) - \log\log(4/\alpha)}{2}} \geq 1+\sqrt{\frac{\log(n\epsilon^2)}{4}} \geq \sqrt{\frac{\log(n\epsilon^2)}{4}}. \label{eq:denom-lb-2}
    \end{align}
    Hence, by~\eqref{eq:CI-case-2},~\eqref{eq:denom-lb-1},~\eqref{eq:denom-lb-2}, and the fact that $x\leq \log\bigl(\frac{1}{1-x}\bigr)$, we have
    \begin{align*}
        |\hat{\mathrm{CI}}^{(1)}_{n,\sigma,\alpha}(t_*)| \leq 40\bigl\{2 \vee \sqrt{2\log\log(4/\alpha)}\bigr\} \cdot \frac{\sigma\epsilon}{\sqrt{\log(n\epsilon^2)}} \lesssim_{\alpha} \frac{\sigma\log\bigl(1/(1-\epsilon)\bigr)}{\sqrt{\log(n\epsilon^2)}},
    \end{align*}
    with probability at least $1-\alpha/2$.

    \textbf{Case 3:} Suppose that $\frac{1}{20} <\epsilon \leq \Bigl\{1- \bigl(\frac{\log(4/\alpha)}{n\epsilon^2}\bigr)^{1/60}\Bigr\} \vee \frac{1}{20}$, and let $b\coloneqq 30\log\bigl(\frac{1}{1-\epsilon}\bigr) / t_*$. In this case, $t_* = \sqrt{\{\log(n\epsilon^2)-\log\log(4/\alpha)\}/2}$ and $b\leq t_*$. Then by Lemma~\ref{lemma:gaussian-cdf-ratio}\emph{(b)},
    \begin{align*}
        (1-\epsilon)\frac{1 - \Phi(t_*-b)}{1-\Phi(t_*)} \geq \frac{1-\epsilon}{\sqrt{2}} \exp\biggl\{15\log\biggl(\frac{1}{1-\epsilon}\biggr)\biggr\} &= \frac{1}{\sqrt{2}(1-\epsilon)^{14}}\\
        &\geq 1+3\sqrt{2}\epsilon = 1 + 6e^{t_*^2}\sqrt{\frac{\log(4/\alpha)}{2n}}.
    \end{align*}
    This verifies~\eqref{eq:sufficient-cond-gaussian-known-eps}. Thus, by~\eqref{eq:denom-lb-1} and~\eqref{eq:denom-lb-2}, we have
    \begin{align*}
        |\hat{\mathrm{CI}}^{(1)}_{n,\sigma,\alpha}(t_*)| \leq 60\bigl\{2 \vee \sqrt{2\log\log(4/\alpha)}\bigr\} \cdot \frac{\sigma\log\bigl(1/(1-\epsilon)\bigr)}{\sqrt{\log(n\epsilon^2)}},
    \end{align*}
    with probability at least $1-\alpha/2$.

    \textbf{Case 4:} Suppose that $\Bigl\{1- \bigl(\frac{\log(4/\alpha)}{n\epsilon^2}\bigr)^{1/60}\Bigr\} \vee \frac{1}{20} < \epsilon \leq 1-\frac{4\log_2(8/\alpha)}{n}$. By Lemma~\ref{lemma:max-min}, with probability at least $1-\alpha/2$, 
    \begin{align*}
        |\hat{\mathrm{CI}}^{(1)}_{n,\sigma,\alpha}(t_*)| \leq Z_{\max} + \sigma t_* - \bigl(Z_{\min} - \sigma t_*\bigr) &\leq 2\sigma\sqrt{2\log(4n/\alpha)} + \sigma \sqrt{2\{\log n - \log\log(4/\alpha)\}}\\
        &\leq 5\sigma\sqrt{\log n - \log\log(4/\alpha)} \leq C_3 \frac{\sigma\log\bigl(1/(1-\epsilon)\bigr)}{\sqrt{\log(n\epsilon^2)}},
    \end{align*}
    where $C_3>0$ depends only on $\alpha$ and the final inequality uses the fact that $\epsilon > 1- \bigl(\frac{\log(4/\alpha)}{n\epsilon^2}\bigr)^{1/60}$.

    The final result follows by combining all the four cases.
\end{proof}

\subsection{\texorpdfstring{Known $\sigma$, unknown $\epsilon$}{Known sigma, unknown epsilon}} \label{sec:known-sigma-unknown-eps}

One issue about the confidence interval in~\eqref{eq:GST-warmup} is that it requires knowledge of both $\sigma$ and $\epsilon$ (as $\hat{\mathrm{CI}}^{(1)}_{n,\sigma,\alpha}$ depends on $\sigma$ and $t_{n,\epsilon,\alpha}$ depends on $\epsilon$). When $\sigma$ is known but $\epsilon$ is unknown, we may choose $t\geq 0$ that minimises the length of $\hat{\mathrm{CI}}^{(1)}_{n,\sigma,\alpha}(t)$, i.e.~we define
\begin{align}
    \hat{t}_{n,\sigma,\alpha} \in \sargmin_{t\geq 0} \bigl|\hat{\mathrm{CI}}^{(1)}_{n,\sigma,\alpha}(t)\bigr| \label{eq:hat-t}
\end{align}
where $\sargmin$ denotes the smallest element of the $\argmin$ set,
and consider the confidence interval $\hat{\mathrm{CI}}^{(1)}_{n,\sigma,\alpha}(\hat{t}_{n,\sigma,\alpha})$. This confidence interval has desired coverage as a consequence of Proposition~\ref{prop:all-param-known}\emph{(a)}. Moreover, its length is no larger than the length of $\hat{\mathrm{CI}}^{(1)}_{n,\sigma,\alpha}(t_{n,\epsilon,\alpha})$, so its length also satisfies the rate in Proposition~\ref{prop:all-param-known}\emph{(b)}.

\begin{Proposition} \label{prop:known-sigma-unknown-epsilon}
  Let $n\in\mathbb{N}$, $\sigma>0$, $\alpha\in(0,1]$, $\epsilon \in \bigl[0,1-\frac{4\log_2(8/\alpha)}{n}\bigr]$, and $Z_1,\ldots,Z_n \overset{\mathrm{iid}}{\sim} (1-\epsilon)N(\theta,\sigma^2)+\epsilon Q$, where $Q\in\mathsf{MNAR}_{N(\theta,\sigma^2)}$. Then
  \begin{align*}
      \mathbb{P}\bigl(\theta \in \hat{\mathrm{CI}}^{(1)}_{n,\sigma,\alpha}(\hat{t}_{n,\sigma,\alpha})\bigr) \geq 1-\alpha. 
  \end{align*}
  Moreover, there exists a universal constant $C>0$ such that if $\alpha\in\bigl[\frac{C}{n}, 1\bigr]$, then with probability at least $1-\alpha/2$,
  \begin{align*}
      \bigl|\hat{\mathrm{CI}}^{(1)}_{n,\sigma,\alpha}(\hat{t}_{n,\sigma,\alpha})\bigr| \leq  C_{\alpha}\sigma\biggl\{\frac{1}{\sqrt{n}} + \frac{\log\bigl(1/(1-\epsilon)\bigr)}{\sqrt{1\vee \log(n\epsilon^2)}}\biggr\},
  \end{align*}
  where $C_{\alpha}\geq 1$ is defined as in Proposition~\ref{prop:all-param-known}.
\end{Proposition}

\subsection{\texorpdfstring{Unknown $\sigma$, known $\epsilon$}{Unknown sigma, known epsilon}} \label{sec:warm-up-unknown-sigma}

Surprisingly, Proposition~\ref{prop:all-param-known} suggests a confidence interval for~$\sigma$ that can be combined with our confidence interval in Section~\ref{sec:all-known} to obtain a valid confidence interval when $\sigma$ is unknown but $\epsilon$ is known.  To see this, set
\begin{align}
    t_{n,\epsilon,\alpha}' \coloneqq \sqrt{1\vee\{\log(n\epsilon^2)-\log\log(12/\alpha)\}/2}, \label{eq:t'-defn}
\end{align}
and observe that the coverage property of Proposition~\ref{prop:all-param-known}\emph{(a)} (with $\alpha$ replaced by $\alpha/3$) implies that with probability at least $1-\alpha/3$, the interval $\hat{\mathrm{CI}}_{n,\sigma,\alpha}^{(1)}(t_{n,\epsilon,\alpha}')$ is non-empty, or equivalently
\begin{align}
    Q_n^-\biggl(1-\Phi(t_{n,\epsilon,\alpha}')+\sqrt{\frac{\log(12/\alpha)}{2n}}\biggr) - Q_n^+\biggl(1-\Phi(t_{n,\epsilon,\alpha}') + \sqrt{\frac{\log(12/\alpha)}{2n}}\biggr) + 2\sigma t_{n,\epsilon,\alpha}' \geq 0. \label{eq:sigma-lb}
\end{align}
The definition of $t_{n,\epsilon,\alpha}'$ involves a maximum with one as opposed to with zero in the definition of~$t_{n,\epsilon,\alpha}$ in Proposition~\ref{prop:all-param-known}.  Nevertheless, Lemma~\ref{lemma:sigma-ub} below (which is a slight modification of Proposition~\ref{prop:all-param-known}\emph{(b)}) shows that with probability at least $1-\alpha/6$,
\begin{align}
    &Q_n^-\biggl(1-\Phi(t_{n,\epsilon,\alpha}')+\sqrt{\frac{\log(12/\alpha)}{2n}}\biggr) - Q_n^+\biggl(1-\Phi(t_{n,\epsilon,\alpha}') + \sqrt{\frac{\log(12/\alpha)}{2n}}\biggr) + 2\sigma t_{n,\epsilon,\alpha}' \nonumber\\
    &\hspace{9cm} \leq 2C_{\alpha/3}\sigma\biggl\{\frac{1}{\sqrt{n}} + \frac{\log\bigl(1/(1-\epsilon)\bigr)}{\sqrt{1\vee \log(n\epsilon^2)}}\biggr\}, \label{eq:sigma-ub}
\end{align}
where $C_{\alpha}\geq 1$ is taken from Proposition~\ref{prop:all-param-known}.

\begin{Lemma} \label{lemma:sigma-ub}
    Let $n\in\mathbb{N}$, $\sigma>0$, $\alpha\in(0,1]$, $\epsilon \in \bigl[0,1-\frac{4\log_2(24/\alpha)}{n}\bigr]$, and $Z_1,\ldots,Z_n \overset{\mathrm{iid}}{\sim} (1-\epsilon)N(\theta,\sigma^2)+\epsilon Q$,  where $Q\in\mathsf{MNAR}_{N(\theta,\sigma^2)}$. There exists a universal constant $C>0$ such that if $\alpha \in \bigl[ \frac{C}{n}, 1\bigr]$, then~\eqref{eq:sigma-ub} holds with probability at least $1-\alpha/6$.
\end{Lemma}
\begin{proof}
    We follow the same arguments and notation as in the proof of Proposition~\ref{prop:all-param-known}\emph{(b)}, but with $\alpha$ replaced by $\alpha/3$ and $t_* = t_{n,\epsilon,\alpha}'$. We will assume $C$ is large enough that $n\geq 100 e^2\log(12/\alpha)$. 

    First suppose that $\epsilon \leq \sqrt{\frac{e^2\log(12/\alpha)}{n}}$, and let $b \coloneqq 10\sqrt{\frac{e^2\log(12/\alpha)}{n}}$. In this case, we have $t_*=1$ and $b\leq 1$, so by Lemma~\ref{lemma:gaussian-cdf-ratio}\emph{(a)},
    \begin{align*}
        (1-\epsilon)\frac{1 - \Phi(t_*-b)}{1-\Phi(t_*)} \geq \biggl(1 - \sqrt{\frac{e^2\log(12/\alpha)}{n}}\biggr) \biggl(1 + \frac{10}{\sqrt{e}}\sqrt{\frac{e^2\log(12/\alpha)}{n}}\biggr) \geq 1 + 6e^{t_*^2}\sqrt{\frac{\log(12/\alpha)}{2n}}.
    \end{align*}
    This verifies~\eqref{eq:sufficient-cond-gaussian-known-eps} with $\alpha$ replaced by $\alpha/3$. Thus, following the argument in the proof of Proposition~\ref{prop:all-param-known}\emph{(b)}, in this case, we have $|\hat{\mathrm{CI}}_{n,\sigma,\alpha}(t_*)| \leq 2\sigma b = 20e\sigma \sqrt{\frac{\log(12/\alpha)}{n}}$ with probability at least $1-\alpha/6$.

    For $\epsilon > \sqrt{\frac{e^2\log(12/\alpha)}{n}}$, the same proof as in Cases~2--4 of Proposition~\ref{prop:all-param-known}\emph{(b)} holds, except that now $\alpha$ is replaced by $\alpha/3$.
\end{proof}

Thus, if we define
\begin{equation}
\hat{\sigma}_{\mathrm{L}}^0 \coloneqq \frac{Q_n^+\Bigl(1-\Phi(t_{n,\epsilon,\alpha}') + \sqrt{\frac{\log(12/\alpha)}{2n}}\Bigr) - Q_n^-\Bigl(1-\Phi(t_{n,\epsilon,\alpha}')+\sqrt{\frac{\log(12/\alpha)}{2n}}\Bigr)}{2t_{n,\epsilon,\alpha}'}, \label{eq:sigma-hat}
\end{equation}
then combining the two inequalities~\eqref{eq:sigma-lb} and~\eqref{eq:sigma-ub} immediately yields the following corollary.

\begin{Corollary} \label{cor:sigma-CI}
Let $n\in\mathbb{N}$, $\sigma>0$, $\alpha\in(0,1]$, $\epsilon \in \bigl[0,1-\frac{4\log_2(24/\alpha)}{n}\bigr]$, and $Z_1,\ldots,Z_n \overset{\mathrm{iid}}{\sim} (1-\epsilon)N(\theta,\sigma^2)+\epsilon Q$ where $Q\in\mathsf{MNAR}_{N(\theta,\sigma^2)}$. There exists a universal constant $C>0$ such that if $\alpha \in \bigl[ \frac{C}{n}, 1\bigr]$, then with probability at least $1-\alpha/2$, we have
\begin{equation*}
0\leq \frac{\sigma-\hat{\sigma}^0_{\mathrm{L}}}{\sigma}\leq \frac{C_{\alpha/3}}{t'_{n,\epsilon,\alpha}} \biggl\{\frac{1}{\sqrt{n}} + \frac{\log\bigl(1/(1-\epsilon)\bigr)}{\sqrt{1\vee \log(n\epsilon^2)}}\biggr\} \eqqcolon r_{n,\epsilon,\alpha}. 
\end{equation*}
In particular, if we define
\begin{align*}
    \hat{\sigma}_{\mathrm{R}}^0 \coloneqq \frac{\hat{\sigma}^0_{\mathrm{L}}}{(1-r_{n,\epsilon,\alpha})_+}, 
\end{align*}
then $\mathbb{P}(\hat{\sigma}_{\mathrm{L}}^0 \leq \sigma \leq \hat{\sigma}_{\mathrm{R}}^0) \geq 1-\alpha/2$.
\end{Corollary}

Corollary~\ref{cor:sigma-CI} motivates the construction of a confidence interval that replaces the unknown~$\sigma$ in~\eqref{eq:GST-warmup} with $\hat{\sigma}_{\mathrm{R}}$ to yield
\begin{align*}
    &\hat{\mathrm{CI}}^{(2)}_{n,\epsilon,\alpha}\coloneqq\\
    &\begin{cases}
        \biggl[Q_n^+\biggl(1-\Phi(t_{n,\epsilon,\alpha}') + \sqrt{\frac{\log(12/\alpha)}{2n}}\biggr)-\hat{\sigma}_{\mathrm{R}} t_{n,\epsilon,\alpha}',\, Q_n^-\biggl(1-\Phi(t_{n,\epsilon,\alpha}')+\sqrt{\frac{\log(12/\alpha)}{2n}}\biggr)+ \hat{\sigma}_{\mathrm{R}} t_{n,\epsilon,\alpha}'\biggr]\\
        \hspace{13.3cm} \text{if } r_{n,\epsilon,\alpha} \leq 0.9\\
        [Z_{\min}, Z_{\max}] \quad\text{if } r_{n,\epsilon,\alpha} > 0.9. 
    \end{cases} 
\end{align*}
Proposition~\ref{prop:unknown-sigma} below shows that this adaptive confidence interval retains the same coverage and length guarantees as its non-adaptive counterpart.

\begin{Proposition} \label{prop:unknown-sigma}
Let $n\in\mathbb{N}$, $\sigma>0$, $\alpha\in(0,1]$, $\epsilon \in \bigl[0,1-\frac{4\log_2(24/\alpha)}{n}\bigr]$, and $Z_1,\ldots,Z_n \overset{\mathrm{iid}}{\sim} (1-\epsilon)N(\theta,\sigma^2)+\epsilon Q$ where $Q\in\mathsf{MNAR}_{N(\theta,\sigma^2)}$. Then 
\begin{align*}
    \mathbb{P}\bigl(\theta \in \hat{\mathrm{CI}}^{(2)}_{n,\epsilon,\alpha}\bigr) \geq 1-\alpha.
\end{align*}
Moreover, there exists a universal constant $C>0$ such that if $\alpha \in \bigl[ \frac{C}{n}, 1\bigr]$, then with probability at least $1-\alpha$, we have 
\begin{align*}
    \bigl|\hat{\mathrm{CI}}^{(2)}_{n,\epsilon,\alpha} \bigr| \leq C_{\alpha}'\sigma\biggl\{\frac{1}{\sqrt{n}} + \frac{\log\bigl(1/(1-\epsilon)\bigr)}{\sqrt{1\vee \log(n\epsilon^2)}}\biggr\},
\end{align*}
where $C_{\alpha}'>0$ depends only on $\alpha$.
\end{Proposition}

\begin{proof}
    When $r_{n,\epsilon,\alpha} \leq 0.9$, the interval $\hat{\mathrm{CI}}^{(2)}_{n,\epsilon,\alpha}$ is the same as $\hat{\mathrm{CI}}^{(1)}_{n,\sigma,\alpha/3}(t'_{n,\epsilon,\alpha})$ but with $\sigma$ in the latter interval replaced by $\hat{\sigma}_{\mathrm{R}}^0$, so by Corollary~\ref{cor:sigma-CI} and Proposition~\ref{prop:all-param-known}\emph{(a)},
    \begin{align*}
        &\mathbb{P}\bigl(\theta \notin \hat{\mathrm{CI}}^{(2)}_{n,\epsilon,\alpha}\bigr) \leq \frac{\alpha}{2} + \mathbb{P}\bigl(\theta \notin \hat{\mathrm{CI}}^{(1)}_{n,\sigma,\alpha/3}(t'_{n,\epsilon,\alpha})\bigr) \leq \alpha.
    \end{align*}
    When $r_{n,\epsilon,\alpha} > 0.9$, the coverage follows from~\eqref{eq:Zmax<theta} and the line following that bound.

    Regarding the length guarantee, suppose first that $r_{n,\epsilon,\alpha} \leq 0.9$. Then by Lemma~\ref{lemma:sigma-ub} and Corollary~\ref{cor:sigma-CI}, we have with probability at least $1-\alpha$ that
    \begin{align*}
        Q_n^-\biggl(1-\Phi(t_{n,\epsilon,\alpha}')&+\sqrt{\frac{\log(12/\alpha)}{2n}}\biggr) - Q_n^+\biggl(1-\Phi(t_{n,\epsilon,\alpha}') + \sqrt{\frac{\log(12/\alpha)}{2n}}\biggr) +2\hat{\sigma}_{\mathrm{R}}^0 t_{n,\epsilon,\alpha}'\\
        &\leq 2C_{\alpha/3}\sigma \biggl\{\frac{1}{\sqrt{n}} + \frac{\log\bigl(1/(1-\epsilon)\bigr)}{\sqrt{1\vee \log(n\epsilon^2)}}\biggr\} + 2(\hat{\sigma}_{\mathrm{R}}^0 - \sigma) t_{n,\epsilon,\alpha}'\\
        &\leq 2C_{\alpha/3}\sigma \biggl\{\frac{1}{\sqrt{n}} + \frac{\log\bigl(1/(1-\epsilon)\bigr)}{\sqrt{1\vee \log(n\epsilon^2)}}\biggr\} + 2\biggl(\frac{\sigma}{1-r_{n,\epsilon,\alpha}} - \sigma\biggr) t_{n,\epsilon,\alpha}'\\
        &\leq 2C_{\alpha/3}\sigma \biggl\{\frac{1}{\sqrt{n}} + \frac{\log\bigl(1/(1-\epsilon)\bigr)}{\sqrt{1\vee \log(n\epsilon^2)}}\biggr\} + 20\sigma r_{n,\epsilon,\alpha} t_{n,\epsilon,\alpha}'\\
        &= 22C_{\alpha/3} \sigma \biggl\{\frac{1}{\sqrt{n}} + \frac{\log\bigl(1/(1-\epsilon)\bigr)}{\sqrt{1\vee \log(n\epsilon^2)}}\biggr\}.
    \end{align*}
    If $r_{n,\epsilon,\alpha} > 0.9$, then there exist $c_\alpha > 0$, depending only on $\alpha$, and a universal constant $C > 0$, such that $\epsilon \geq 1-n^{-c_{\alpha}}$ for all $\alpha \in [C/n,1]$. We may therefore apply Lemma~\ref{lemma:max-min} to deduce that with probability at least $1-\alpha/2$,
    \begin{align*}
        \bigl| \hat{\mathrm{CI}}^{(2)}_{n,\epsilon,\alpha} \bigr| \leq 2\sigma\sqrt{2\log(4n/\alpha)} \leq C_{\alpha}'\sigma \biggl\{\frac{1}{\sqrt{n}} + \frac{\log\bigl(1/(1-\epsilon)\bigr)}{\sqrt{1\vee \log(n\epsilon^2)}}\biggr\},
    \end{align*}
    for sufficiently large $C'_{\alpha}>0$, depending only on $\alpha$.
\end{proof}

\subsection{\texorpdfstring{Unknown $\sigma$ and $\epsilon$}{Unknown sigma and epsilon}} \label{sec:unknown-sigma-eps}
In this section, we assume that $\sigma\in(0,\sigma_{\max}]$ and $\epsilon\in[0,1]$ are both unknown, where $\sigma_{\max}>0$ is a known upper bound on $\sigma$. Recall the definition of $\hat{\mathrm{CI}}^{(1)}_{n,\sigma,\alpha}(t)$ from~\eqref{eq:GST-warmup} and define
\begin{align*}
    \hat{t}_{n,\sigma_{\max},\alpha} \coloneqq \argmin_{t\geq 0} \bigl|\hat{\mathrm{CI}}^{(1)}_{n,\sigma_{\max},\alpha}(t)\bigr|.
\end{align*}
We will take  $\hat{\mathrm{CI}}^{(1)}_{n,\sigma_{\max},\alpha}(\hat{t}_{n,\sigma_{\max},\alpha})$ as our confidence interval.  The following theorem guarantees its coverage and width.

\begin{Proposition}\label{prop:unknown-eps-sigma}
Let $n\in\mathbb{N}$, $\sigma\in(0,\sigma_{\max}]$, $\alpha\in(0,1]$, $\epsilon \in \bigl[0,1-\frac{4\log_2(8/\alpha)}{n}\bigr]$, and $Z_1,\ldots,Z_n \overset{\mathrm{iid}}{\sim} (1-\epsilon)N(\theta,\sigma^2)+\epsilon Q$ where $Q\in\mathsf{MNAR}_{N(\theta,\sigma^2)}$. Then 
\begin{align*}
    \mathbb{P}\bigl(\theta \in \hat{\mathrm{CI}}^{(1)}_{n,\sigma_{\max},\alpha}(\hat{t}_{n,\sigma_{\max},\alpha}) \bigr) \geq 1-\alpha.
\end{align*}
Moreover, there exist universal constants $C,C'>0$ such that if $\alpha\in\bigl[\frac{C}{n}, 1\bigr]$, then with probability at least $1-\alpha/2$, we have
\begin{align*}
    \bigl|\hat{\mathrm{CI}}^{(1)}_{n,\sigma_{\max},\alpha}(\hat{t}_{n,\sigma_{\max},\alpha})\bigr| \leq C'\sigma_{\max}\biggl\{\biggl(\sqrt{\frac{\log(4/\alpha)}{n}} + \epsilon\biggr) \vee \biggl(\sqrt{\log\Bigl(\frac{1}{1-\epsilon}\Bigr)} - 1 \biggr)\biggr\}.
\end{align*}
\end{Proposition}
\begin{proof}
    The coverage guarantee directly follows from Proposition~\ref{prop:all-param-known}\emph{(a)}, since $\sigma_{\max} \geq \sigma$.

    \medskip
    For the length guarantee, we assume throughout the proof that the universal constant $C > 0$ is large enough that $\alpha \geq 4e^{-2n/45^2}$. We follow a similar argument to the proof of Proposition~\ref{prop:all-param-known}\emph{(b)}. In particular, we showed in the proof of Proposition~\ref{prop:all-param-known}\emph{(b)} that if there exist $t,b\geq 0$ such that~\eqref{eq:sufficient-cond-gaussian-known-eps} holds, 
    then with probability at least $1-\alpha/2$, we have $Q_n^-\bigl(1-\Phi(t) + \sqrt{\frac{\log(4/\alpha)}{2n}}\bigr) \leq \theta + \sigma(b-t)$ and $Q_n^+\bigl(1-\Phi(t) + \sqrt{\frac{\log(4/\alpha)}{2n}}\bigr) > \theta + \sigma(t-b)$, so 
    \begin{align*}
        |\hat{\mathrm{CI}}^{(1)}_{n,\sigma_{\max},\alpha}(\hat{t}_{n,\sigma_{\max},\alpha})| &\leq |\hat{\mathrm{CI}}^{(1)}_{n,\sigma_{\max},\alpha}(t)| \\
        &= Q_n^-\biggl(1-\Phi(t) + \sqrt{\frac{\log(4/\alpha)}{2n}}\biggr) - Q_n^+\biggl(1-\Phi(t) + \sqrt{\frac{\log(4/\alpha)}{2n}}\biggr) + 2\sigma_{\max} t\\
        &\leq 2\sigma b + 2(\sigma_{\max} - \sigma)t.
    \end{align*}

    First consider the case where $\epsilon \leq 1/5$. We take $t\coloneqq 0$ and $b\coloneqq 18\sqrt{\frac{\log(4/\alpha)}{2n}} + 3\epsilon \in [0,1]$. Then
    \begin{align*}
        (1-\epsilon)\frac{1 - \Phi(t-b)}{1-\Phi(t)} &= 2(1-\epsilon)\Phi(b) \geq 2(1-\epsilon)\biggl( \frac{1}{2}+ b\cdot\phi(1)\biggr) \\
        &\geq 1-\epsilon + \frac{8e^{-1/2}}{5\sqrt{2\pi}} \biggl(18\sqrt{\frac{\log(4/\alpha)}{2n}} + 3\epsilon\biggr) \geq 1 + 6e^{t^2}\sqrt{\frac{\log(4/\alpha)}{2n}},
    \end{align*}
    where the second inequality uses the mean value inequality.  Hence, in this case, we deduce that $|\hat{\mathrm{CI}}^{(1)}_{n,\sigma_{\max},\alpha}(\hat{t}_{n,\sigma_{\max},\alpha})| \leq 36\sigma\sqrt{\frac{\log(4/\alpha)}{2n}} + 6\sigma\epsilon$ with probability at least $1-\alpha/2$.

    Next, we consider the case where $1/5 <\epsilon \leq 1-\bigl(\frac{\log(4/\alpha)}{n}\bigr)^{1/32}$. We take $t = b \coloneqq 4\sqrt{\log\bigl(\frac{1}{1-\epsilon}\bigr)} \geq 9/5$. Then
    \begin{align*}
        (1-\epsilon)\frac{1 - \Phi(t-b)}{1-\Phi(t)} &= \frac{1-\epsilon}{2\bigl(1-\Phi(t)\bigr)} \geq \frac{1-\epsilon}{2} \cdot \frac{t}{\phi(t)}\\
        &\geq \frac{9\sqrt{2\pi}}{10} \cdot \frac{1}{(1-\epsilon)^7} \geq 10 \geq 1 + 6e^{t^2}\sqrt{\frac{\log(4/\alpha)}{2n}},
    \end{align*}
    where the first inequality uses the Mills ratio bound $1-\Phi(x) \leq \phi(x)/x$ for $x>0$, and the final inequality uses $\epsilon \leq 1-\bigl(\frac{\log(4/\alpha)}{n}\bigr)^{1/32}$. Hence, in this case, we have $|\hat{\mathrm{CI}}^{(1)}_{n,\sigma_{\max},\alpha}(\hat{t}_{n,\sigma_{\max},\alpha})| \leq 8\sigma_{\max}\sqrt{\log\bigl(\frac{1}{1-\epsilon}\bigr)}$ with probability at least $1-\alpha/2$.

    Finally, we consider the case where $1-\bigl(\frac{\log(4/\alpha)}{n}\bigr)^{1/32}<\epsilon \leq 1-\frac{4\log_2(8/\alpha)}{n}$. By Lemma~\ref{lemma:max-min}, we have with probability at least $1-\alpha/2$ that
    \begin{align*}
            |\hat{\mathrm{CI}}^{(1)}_{n,\sigma_{\max},\alpha}(\hat{t}_{n,\sigma_{\max},\alpha})| &\leq |\hat{\mathrm{CI}}^{(1)}_{n,\sigma_{\max},\alpha}(0)| \leq Z_{\max}-Z_{\min} \leq 2\sigma\sqrt{2\log(4n/\alpha)}\\
        &\leq 2\sigma\sqrt{2\log(4Cn^2)} \leq C'\sigma_{\max}\sqrt{\log\biggl(\frac{1}{1-\epsilon}\biggr)},
    \end{align*}
    for some universal constant $C' > 0$.

    The final conclusion follows by combining the three cases.
\end{proof}

\section{Proof of Theorem~\ref{thm:gaussian-minimax-rates}} \label{sec:proof-of-gaussian-rates}
In this section we present the proof of Theorem~\ref{thm:gaussian-minimax-rates}. In each of the four cases, the proof involves an upper bound showing that the confidence intervals described in Section~\ref{sec:gaussian-CI-construction-general} have desired coverage and length guarantees, and a matching lower bound showing that the minimax length of any adaptive confidence interval must have the same rate as the length guarantee.
\subsection{\texorpdfstring{Known $\sigma,\epsilon_1,\epsilon_2$}{Known sigma, epsilon1, epsilon2}}
\subsubsection{Upper bound}
Recall the generalised Gaussian Scheff{\'e}--Tukey interval from~\eqref{eq:GST}. Define 
\begin{align}
    \widehat{\mathrm{CI}}_{n,\sigma,\epsilon_1,\alpha}^{(3)}(t) &\coloneqq \hat{\mathrm{CI}}_n^{(\mathrm{GST})} \Bigl(t;\; 1-\epsilon_1,\, 18(1-\epsilon_1),\, \frac{\alpha}{4},\, \sigma\Bigr)\nonumber\\
    &\phantom{:}=  \Biggl[Q_n^{+}\biggl((1-\epsilon_1)\{1-\Phi(t)\}+3\sqrt{\frac{(1-\epsilon_1)\log(16/\alpha)}{n}}\biggr)-\sigma t, \nonumber\\
    &\hspace{3.5cm} Q_n^{-}\biggl((1-\epsilon_1)\{1-\Phi(t)\}+3\sqrt{\frac{(1-\epsilon_1)\log(16/\alpha)}{n}}\biggr)+\sigma t\Biggr].\label{eq:gaussian-general-CI-a}
\end{align}
as the confidence interval constructed in Section~\ref{sec:gaussian-CI-construction-general} part~\emph{(a)}.

\begin{Proposition}
\label{prop:all-param-known-2}
    Let $n\in\mathbb{N}$, $\sigma>0$, $\alpha\in(0,1]$ and $\epsilon_1,\epsilon_2 \in [0,1]$ be such that $\epsilon_1+\epsilon_2 \leq 1 - \frac{4\log_2(8/\alpha)}{n}$ and $Z_1,\ldots,Z_n \overset{\mathrm{iid}}{\sim} M \in \mathcal{M}\bigl(N(\theta,\sigma^2),\epsilon_1,\epsilon_2\bigr)$.
    
    (a) The confidence interval $\widehat{\mathrm{CI}}_{n,\sigma,\epsilon_1,\alpha}^{(3)}(t)$ in~\eqref{eq:gaussian-general-CI-a} satisfies simultaneous coverage for all $t\geq 0$, i.e.
    \begin{align*}
        \mathbb{P}\biggl( \theta\in \bigcap_{t\geq 0} \hat{\mathrm{CI}}^{(3)}_{n,\sigma,\epsilon_1,\alpha}(t) \biggr) \geq 1-\alpha.
    \end{align*}
    
    (b) Let $t_{n,\epsilon_1,\epsilon_2,\alpha} \coloneqq \sqrt{0\vee\bigl\{\log\bigl(n\epsilon_2^2/(1-\epsilon_1)\bigr)-\log\log(16/\alpha)\bigr\}/2}$. There exists a universal constant $C>0$ such that if $\alpha \in \bigl[ \frac{C}{n(1-\epsilon_1)}, 1\bigr]$, then with probability at least $1-\alpha/2$, we have
    \begin{align*}
        \bigl| \hat{\mathrm{CI}}^{(3)}_{n,\sigma,\epsilon_1,\alpha} (t_{n,\epsilon_1,\epsilon_2,\alpha}) \bigr| \leq C_{\alpha} \sigma \biggl\{\frac{1}{\sqrt{n(1-\epsilon_1)}} + \frac{\log\bigl(\frac{1-\epsilon_1}{1-\epsilon_1-\epsilon_2}\bigr)}{\sqrt{1\vee \log\{n\epsilon_2^2/(1-\epsilon_1)\}}}\biggr\},
    \end{align*}
    where $C_{\alpha} \geq 1$ depends only on $\alpha$.
\end{Proposition}

\begin{proof}
    \emph{(a)} We may write $M=(1-\epsilon_1-\epsilon_2)N(\theta,\sigma^2) + (\epsilon_1+\epsilon_2)Q$ where $Q\in\mathsf{MNAR}_{N(\theta,\sigma^2)}$. For $i\in[n]$, we may assume that\footnote{Here we recall our convention that $0\cdot\star\coloneqq 0$.} $Z_i \coloneqq L_i X_i + (1-L_i)(X_i \ostar \Omega_i)$ where $(X_1,L_1,\Omega_1),\ldots,(X_n,L_n,\Omega_n)$ are independent triples with $X_i \sim N(\theta,\sigma^2)$, $\Omega_i$ a binary random variable such that $X_i \ostar \Omega_i \sim Q$ and $L_i \sim \mathrm{Bern}(1-\epsilon_1-\epsilon_2)$ independent of $(X_i,\Omega_i)$. In other words, if $L_i=1$, then $Z_i = X_i \sim N(\theta,\sigma^2)$ and if $L_i=0$, then $Z_i \sim Q$. Let $\mathcal{I} \coloneqq \{i\in[n] : L_i = 1\}$ so that $|\mathcal{I}| \sim \mathrm{Bin}(n, 1-\epsilon_1-\epsilon_2)$. By the multiplicative Chernoff bound \citep[Theorem~2.3(c)]{McDiarmid1998}, and since $1-\epsilon_1-\epsilon_2 \geq \frac{4\log_2(8/\alpha)}{n}$, we deduce that
    \begin{align*}
        \mathbb{P}\bigl(|\mathcal{I}| \leq \log_2(8/\alpha)\bigr) \leq \exp\bigl(-9\log_2(8/\alpha)/8\bigr) \leq \frac{\alpha}{8}.
    \end{align*}
    Therefore,
    \begin{align}
        \mathbb{P}(Z_{\max} \leq \theta) \leq \mathbb{P}\Bigl(\max_{i\in\mathcal{I}}X_i \leq \theta\Bigr) &\leq \mathbb{P}\Bigl(\max_{i\in\mathcal{I}}X_i \leq \theta \,\Big|\, |\mathcal{I}| > \log_2(8/\alpha)\Bigr) + \mathbb{P}\Bigl(|\mathcal{I}| \leq \log_2(8/\alpha)\Bigr) \nonumber\\
        &\leq 2^{-\log_2(8/\alpha)} + \frac{\alpha}{8} \leq \frac{\alpha}{4}. \label{eq:Zmax<theta-2}
    \end{align}
    Similarly, $\mathbb{P}(Z_{\min} \geq \theta) \leq \alpha/4$.
    Thus, we deduce that
    \begin{align*}
        \mathbb{P}\biggl\{&Q_n^{-}\biggl((1-\epsilon_1)\{1-\Phi(t)\} + 3\sqrt{\frac{(1-\epsilon_1)\log(16/\alpha)}{n}}\biggr) < \theta - \sigma t \text{ for some } t\geq 0\biggr\}\\
        & \leq \mathbb{P}\biggl(\bigcup_{t \geq 0}\biggl\{\frac{1}{n}\sum_{i=1}^n \indi_{\{Z_i \leq \theta - \sigma t\}} \geq (1-\epsilon_1)\{1-\Phi(t)\} + 3\sqrt{\frac{(1-\epsilon_1)\log(16/\alpha)}{n}}\biggr\}\biggr) + \mathbb{P}(Z_{\max} \leq \theta)\\
        & \leq \mathbb{P}\biggl(\sup_{t \geq 0}\biggl\{ \frac{1}{n}\sum_{i=1}^n \indi_{\{Z_i \leq \theta - \sigma t\}} - \mathbb{P}_M(Z_1\leq \theta-\sigma t) \biggr\} \geq 3\sqrt{\frac{(1-\epsilon_1)\log(16/\alpha)}{n}} \biggr) + \mathbb{P}(Z_{\max} \leq \theta)\\
        &\leq \frac{\alpha}{2},
    \end{align*}
    where the first inequality follows from Lemma~\ref{lemma:empirical-quantile-property}, the second inequality follows from Lemma~\ref{lemma:realisability-characterisation} and the third inequality follows from Lemma~\ref{lemma:dkw-missing} (with the fact that $\log_2(8/\alpha) > \log(16/\alpha)$) and~\eqref{eq:Zmax<theta-2}.
    Similarly,  
    \begin{align*}
        &\mathbb{P}\biggl\{Q_n^{+}\biggl((1-\epsilon_1)\{1-\Phi(t)\}-3\sqrt{\frac{(1-\epsilon_1)\log(16/\alpha)}{n}}\biggr) > \theta + \sigma t \text{ for some } t\geq 0\biggr\} \leq \frac{\alpha}{2}.
    \end{align*}
    Combining these two bounds yields that
    \begin{align*}
        \mathbb{P}\biggl( \theta\in \bigcap_{t\geq 0} \hat{\mathrm{CI}}^{(3)}_{n,\sigma,\epsilon_1,\alpha}(t) \biggr) \geq 1-\alpha,
    \end{align*}
    as required.

    \medskip
    \emph{(b)} By Lemma~\ref{lemma:realisability-characterisation}, we have $\mathbb{P}(Z_1 \leq \theta + \sigma x) \geq (1-\epsilon_1-\epsilon_2)\bigl\{1-\Phi(-x)\bigr\}$ and $\mathbb{P}(Z_1 > \theta + \sigma x) \geq (1-\epsilon_1-\epsilon_2)\bigl\{1-\Phi(x)\bigr\}$ for all $x\in\mathbb{R}$. Thus, by Lemma~\ref{lemma:dkw-missing}, the event
    \begin{align*}
        \mathcal{E}_1 &\coloneqq \biggl\{\frac{1}{n} \sum_{i=1}^n \indi_{\{Z_i \leq \theta + \sigma(b-t)\}} > (1-\epsilon_1-\epsilon_2)\bigl\{ 1 - \Phi(t-b)\bigr\} - 3\sqrt{\frac{(1-\epsilon_1)\log(16/\alpha)}{n}} \;\;\forall\; b,t\in\mathbb{R}\biggr\}\\
        &\hspace{0.7cm} \bigcap\biggl\{\frac{1}{n} \sum_{i=1}^n \indi_{\{Z_i > \theta + \sigma(t-b)\}} > (1-\epsilon_1-\epsilon_2)\bigl\{ 1 - \Phi(t-b)\bigr\} - 3\sqrt{\frac{(1-\epsilon_1)\log(16/\alpha)}{n}} \;\;\forall\; b,t\in\mathbb{R}\biggr\}
    \end{align*}
    has probability at least $1-\alpha/2$. Now suppose that, for some $t_*,b\geq 0$, we have
    \begin{align}
        \frac{1-\epsilon_1-\epsilon_2}{1-\epsilon_1} \cdot \frac{1 - \Phi(t_*-b)}{1-\Phi(t_*)} \geq 1 + 18e^{t_*^2}\sqrt{\frac{\log(16/\alpha)}{n(1-\epsilon_1)}}. \label{eq:sufficient-cond-gaussian-known-eps-2}
    \end{align}
    By Lemma~\ref{lemma:mills-ratio}, $1-\Phi(t_*) \geq \frac{2}{\sqrt{2\pi}} \cdot \frac{1}{t_*+\sqrt{4+t_*^2}} e^{-t_*^2/2} \geq \frac{1}{3}e^{-t_*^2}$. Thus, 
    \begin{align*}
        (1 \!-\! \epsilon_1 \!-\! \epsilon_2)\bigl\{1 - \Phi(t_*-b)\bigr\} - 3\sqrt{\frac{(1-\epsilon_1)\log(16/\alpha)}{n}} \geq (1 \! - \! \epsilon_1) \{1 - \Phi(t_*)\} + 3\sqrt{\frac{(1-\epsilon_1)\log(16/\alpha)}{n}}.
    \end{align*}
    Hence, on the event $\mathcal{E}_1$, we have both $\frac{1}{n} \sum_{i=1}^n \indi_{\{Z_i \leq \theta + \sigma(b-t_*)\}} >  (1-\epsilon_1) \{1 - \Phi(t_*)\} + 3\sqrt{\frac{(1-\epsilon_1)\log(16/\alpha)}{n}}$ and $\frac{1}{n} \sum_{i=1}^n \indi_{\{Z_i > \theta + \sigma(t_*-b)\}} > (1-\epsilon_1) \{1 - \Phi(t_*)\} + 3\sqrt{\frac{(1-\epsilon_1)\log(16/\alpha)}{n}}$. Thus by Lemma~\ref{lemma:empirical-quantile-property}, with probability at least $1-\alpha/2$, we have
    \begin{align*}
        Q_n^-\biggl((1-\epsilon_1) \{1 - \Phi(t_*)\} + 3\sqrt{\frac{(1-\epsilon_1)\log(16/\alpha)}{n}}\biggr) \leq \theta + \sigma(b-t_*)
    \end{align*}
    and
    \begin{align*}
        Q_n^+\biggl((1-\epsilon_1) \{1 - \Phi(t_*)\} + 3\sqrt{\frac{(1-\epsilon_1)\log(16/\alpha)}{n}}\biggr) > \theta + \sigma(t_*-b),
    \end{align*}
    so $|\hat{\mathrm{CI}}^{(3)}_{n, \sigma, \epsilon_1, \alpha}(t_*)| \leq 2\sigma b$ whenever~\eqref{eq:sufficient-cond-gaussian-known-eps-2} holds. In the remainder of this proof, we take $t_* \equiv t_{n,\epsilon_1,\epsilon_2,\alpha} \coloneqq \sqrt{0\vee\bigl\{\log\bigl(n\epsilon_2^2/(1-\epsilon_1)\bigr)-\log\log(16/\alpha)\bigr\}/2}$ and assume throughout that $C$ in the statement of the theorem is large enough that $n(1-\epsilon_1) \geq 6400\log(16/\alpha)$ for $\alpha \in \bigl[\frac{C}{n(1-\epsilon_1)},1\bigr]$ and that $\frac{1}{2}\log\{n(1-\epsilon_1)/6400\} \geq \log\log\{16n(1-\epsilon_1)/C\}$ for $n(1-\epsilon_1) \geq C/\alpha \geq C$.

    \textbf{Case 1:} Suppose that $\frac{\epsilon_2}{1-\epsilon_1} \leq \sqrt{\frac{\log(16/\alpha)}{n(1-\epsilon_1)}}$, and let $b \coloneqq 80\sqrt{\frac{\log(16/\alpha)}{n(1-\epsilon_1)}}$. In this case, we have $t_*=0$ and $b\leq 1$, so by Lemma~\ref{lemma:gaussian-cdf-ratio}\emph{(a)},
    \begin{align*}
        \frac{1-\epsilon_1-\epsilon_2}{1-\epsilon_1} \cdot \frac{1 - \Phi(t_*-b)}{1-\Phi(t_*)} &\geq \biggl(1 - \sqrt{\frac{\log(16/\alpha)}{n(1-\epsilon_1)}}\biggr) \biggl(1 + \frac{40}{\sqrt{e}}\sqrt{\frac{\log(16/\alpha)}{n(1-\epsilon_1)}}\biggr)\\
        &\geq 1 + 18e^{t_*^2}\sqrt{\frac{\log(16/\alpha)}{n(1-\epsilon_1)}}.
    \end{align*}
    This verifies~\eqref{eq:sufficient-cond-gaussian-known-eps-2}. Thus, in this case, we have $|\hat{\mathrm{CI}}^{(3)}_{n,\sigma,\epsilon_1,\alpha}(t_*)| \leq 160\sigma \sqrt{\frac{\log(16/\alpha)}{n(1-\epsilon_1)}}$ with probability at least $1-\alpha/2$.

    \textbf{Case 2:} Suppose that $\sqrt{\frac{\log(16/\alpha)}{n(1-\epsilon_1)}} < \frac{\epsilon_2}{1-\epsilon_1} \leq \frac{1}{80}$, and let $b\coloneqq \frac{80\epsilon_2}{(t_*+1)(1-\epsilon_1)}$. In this case, $t_*= \sqrt{\{\log\{n\epsilon_2^2/(1-\epsilon_1)\}-\log\log(16/\alpha)\}/2}$ and $b\leq 1$. Then by Lemma~\ref{lemma:gaussian-cdf-ratio}\emph{(a)},
    \begin{align*}
        \frac{1-\epsilon_1-\epsilon_2}{1-\epsilon_1} \cdot \frac{1 - \Phi(t_*-b)}{1-\Phi(t_*)} &\geq \biggl(1-\frac{\epsilon_2}{1-\epsilon_1}\biggr)\biggl(1+\frac{40}{\sqrt{e}} \cdot \frac{\epsilon_2}{1-\epsilon_1} \biggr)\\ &\geq 1+\biggl(\frac{40}{\sqrt{e}} - 1 - \frac{1}{2\sqrt{e}}\biggr) \frac{\epsilon_2}{1-\epsilon_1} \geq 1 + 18e^{t_*^2}\sqrt{\frac{\log(16/\alpha)}{n(1-\epsilon_1)}}.
    \end{align*}
    This verifies~\eqref{eq:sufficient-cond-gaussian-known-eps-2}. Thus, in this case, we have
    \begin{align}
        |\hat{\mathrm{CI}}^{(3)}_{n,\sigma,\epsilon_1,\alpha}(t_*)| \leq 160\sigma\cdot \frac{\epsilon_2/(1-\epsilon_1)}{1+\sqrt{\{\log\{n\epsilon_2^2/(1-\epsilon_1)\}-\log\log(16/\alpha)\}/2}} \label{eq:CI-case-2-2}
    \end{align}
    with probability at least $1-\alpha/2$. Moreover, when $\log\{n\epsilon_2^2/(1-\epsilon_1)\} \leq 2\log\log(16/\alpha)$, we have
    \begin{align}
        1+\sqrt{\frac{\log\{n\epsilon_2^2/(1-\epsilon_1)\} - \log\log(16/\alpha)}{2}} \geq 1 \geq \sqrt{\frac{\log\{n\epsilon_2^2/(1-\epsilon_1)\}}{2\log\log(16/\alpha)}}, \label{eq:denom-lb-1-2}
    \end{align}
    and when $\log\{n\epsilon_2^2/(1-\epsilon_1)\} > 2\log\log(16/\alpha)$, we have
    \begin{align}
        1+\sqrt{\frac{\log\{n\epsilon_2^2/(1-\epsilon_1)\} - \log\log(16/\alpha)}{2}} \geq \sqrt{\frac{\log\{n\epsilon_2^2/(1-\epsilon_1)\}}{4}}. \label{eq:denom-lb-2-2}
    \end{align}
    Hence, by~\eqref{eq:CI-case-2-2},~\eqref{eq:denom-lb-1-2},~\eqref{eq:denom-lb-2-2}, and the fact that $x\leq \log\bigl(\frac{1}{1-x}\bigr)$ for $x \in (0,1)$, we have
    \begin{align*}
        |\hat{\mathrm{CI}}^{(3)}_{n,\sigma,\epsilon_1,\alpha}(t_*)| &\leq 160\bigl\{2 \vee \sqrt{2\log\log(16/\alpha)}\bigr\}\sigma \cdot \frac{\epsilon_2/(1-\epsilon_1)}{\sqrt{\log\{n\epsilon_2^2/(1-\epsilon_1)\}}}\\
        &\lesssim_{\alpha} \sigma \cdot \frac{\log\bigl(\frac{1-\epsilon_1}{1-\epsilon_1-\epsilon_2}\bigr)}{\sqrt{\log\{n\epsilon_2^2/(1-\epsilon_1)\}}}
    \end{align*}
    with probability at least $1-\alpha/2$.

    \textbf{Case 3:} Suppose that $\frac{1}{80} < \frac{\epsilon_2}{1-\epsilon_1} \leq 1- \bigl(\frac{\log(16/\alpha)}{n\epsilon_2^2/(1-\epsilon_1)}\bigr)^{1/200}$, and let $b\coloneqq 100\log\bigl(\frac{1-\epsilon_1}{1-\epsilon_1-\epsilon_2}\bigr) / t_*$. In this case, $t_* = \sqrt{\{\log\{n\epsilon_2^2/(1-\epsilon_1)\}-\log\log(16/\alpha)\}/2}$ and $b\leq t_*$. Hence by Lemma~\ref{lemma:gaussian-cdf-ratio}\emph{(b)},
    \begin{align*}
        \frac{1-\epsilon_1-\epsilon_2}{1-\epsilon_1} \cdot \frac{1 - \Phi(t_*-b)}{1-\Phi(t_*)} &\geq \frac{1-\epsilon_1-\epsilon_2}{\sqrt{2}(1-\epsilon_1)} \exp\biggl\{50\log\biggl(\frac{1-\epsilon_1}{1-\epsilon_1-\epsilon_2}\biggr)\biggr\}\\
        &= \frac{1}{\sqrt{2}} \biggl(\frac{1}{1-\epsilon_2/(1-\epsilon_1)}\biggr)^{49} \geq 1+ 18\frac{\epsilon_2}{1-\epsilon_1} = 1 + 18e^{t_*^2}\sqrt{\frac{\log(16/\alpha)}{n(1-\epsilon_1)}},
    \end{align*}
    where the second inequality uses the fact that $\frac{1}{\sqrt{2}}\bigl(\frac{1}{1-x}\bigr)^{49} \geq 1+18x$ for $x\geq 1/80$.
    This verifies~\eqref{eq:sufficient-cond-gaussian-known-eps-2}, so we have
    \begin{align*}
        |\hat{\mathrm{CI}}^{(3)}_{n,\sigma,\epsilon_1,\alpha}(t_*)| \leq 2\sigma b \leq 200\sigma\cdot \frac{\log\bigl(\frac{1-\epsilon_1}{1-\epsilon_1-\epsilon_2}\bigr)}{(1+t^*)/81} \lesssim_{\alpha} \sigma \cdot \frac{\log\bigl(\frac{1-\epsilon_1}{1-\epsilon_1-\epsilon_2}\bigr)}{\sqrt{\log\{n\epsilon_2^2/(1-\epsilon_1)\}}}
    \end{align*}
    with probability at least $1-\alpha/2$, where the second inequality follows since $t^*\geq b\geq \frac{1}{80}$, and the final inequality follows from~\eqref{eq:denom-lb-1-2} and~\eqref{eq:denom-lb-2-2}.

    \textbf{Case 4:} Suppose that $\frac{\epsilon_2}{1-\epsilon_1} > \bigl\{1- \bigl(\frac{\log(16/\alpha)}{n\epsilon_2^2/(1-\epsilon_1)}\bigr)^{1/200}\bigr\} \vee \frac{1}{80}$. We have
    \begin{align}
        &\log\{n\epsilon_2^2/(1-\epsilon_1)\} - \log\log(16/\alpha)\nonumber\\
        &\quad \geq \frac{\log\{n\epsilon_2^2/(1-\epsilon_1)\}}{2} + \frac{\log\{n(1-\epsilon_1)/6400\}}{2} - \log\log\{16n(1-\epsilon_1)/C\} \geq \frac{\log\{n\epsilon_2^2/(1-\epsilon_1)\}}{2}, \label{eq:denom-bound}
    \end{align}
    where the first inequality follows since $\frac{\epsilon_2}{1-\epsilon_1}\geq \frac{1}{80}$ and $\alpha \geq \frac{C}{n(1-\epsilon_1)}$, and the second inequality follows since $C>0$ is chosen large enough that $\frac{1}{2}\log\{n(1-\epsilon_1)/6400\} \geq \log\log\{16n(1-\epsilon_1)/C\}$ for $n(1-\epsilon_1) \geq C$.
    By Lemma~\ref{lemma:max-min-2}, with probability at least $1-\alpha/2$, 
    \begin{align*}
    |\hat{\mathrm{CI}}^{(3)}_{n,\sigma,\epsilon_1,\alpha}(t_*)| &\leq Z_{\max} + \sigma t_* - \bigl(Z_{\min} - \sigma t_*\bigr)\\
        &\leq 2\sigma\sqrt{2\log\{12n(1-\epsilon_1)/\alpha\}} + \sigma \sqrt{2\bigl\{\log\{n\epsilon_2^2/(1-\epsilon_1)\}-\log\log(16/\alpha)\bigr\}}\\
        &\lesssim \sigma\sqrt{\log \{n\epsilon_2^2/(1-\epsilon_1)\}}\\
        &\lesssim \sigma \cdot\frac{\log \{n\epsilon_2^2/(1-\epsilon_1)\} - \log\log(16/\alpha)}{\sqrt{\log\{n\epsilon_2^2/(1-\epsilon_1)\}}}\\
        &\lesssim \sigma \cdot \frac{\log\bigl(\frac{1-\epsilon_1}{1-\epsilon_1-\epsilon_2}\bigr)}{\sqrt{\log\{n\epsilon_2^2/(1-\epsilon_1)\}}},
    \end{align*}
    where the penultimate inequality follows from~\eqref{eq:denom-bound}, and the final inequality uses the fact that $\frac{\epsilon_2}{1-\epsilon_1} > 1- \bigl(\frac{\log(16/\alpha)}{n\epsilon_2^2/(1-\epsilon_1)}\bigr)^{1/200}$.

    The final result follows by combining all the four cases.
\end{proof}

\subsubsection{Lower bound}
In order to study lower bounds on the length of valid confidence intervals, we next present a lower bound for point estimation of $\theta$ in the model~\eqref{eq:model-mcar-mnar} with $P = N(\theta,\sigma^2)$, for which it is convenient to work in the minimax quantile framework of~\citet{ma2024high}.  Specifically, writing $\hat{\Theta}$ for the set of Borel measurable functions from $\mathbb{R}_\star^n$ to $\mathbb{R}$, for $\alpha \in (0,1)$ we define the \emph{lower minimax $(1-\alpha)$th quantile} for estimating $\theta$ by  
\begin{align*}
    L_{\mathrm{est}}(n,\sigma,\epsilon_1,\epsilon_2) \coloneqq \inf\Bigl\{r\geq 0 : \inf_{\hat{\theta}\in\hat{\Theta}} \sup_{\theta\in\mathbb{R}} \sup_{M\in\mathcal{M}(N(\theta,\sigma^2),\epsilon_1,\epsilon_2)} \mathbb{P}_M(|\hat{\theta}-\theta|>r) \leq \alpha\Bigr\}.
\end{align*}
Since the midpoint of any confidence interval for $\theta$ may be regarded as a point estimate, Proposition~\ref{prop:point-estimate-lb} below immediately implies that $L_{n,\sigma,\epsilon_1,\epsilon_2}\bigl(\hat{\mathcal{CI}}(\sigma, \epsilon_1, \epsilon_2)\bigr) \geq 2L_{\mathrm{est}}(n,\sigma,\epsilon_1,\epsilon_2)$.
    
\begin{Proposition} \label{prop:point-estimate-lb}
    Let $n\in\mathbb{N}$, $\sigma>0$ and  $\epsilon_1,\epsilon_2\in[0,1)$ be such that $\epsilon_1+\epsilon_2 \leq 1- 2/n$, and let $\alpha\in(0,1/4]$.  
    Then
    \begin{align*}
        L_{\mathrm{est}}(n,\sigma,\epsilon_1,\epsilon_2) \geq c \sigma \biggl\{\frac{1}{\sqrt{n(1-\epsilon_1)}} + \frac{\log\bigl(\frac{1-\epsilon_1}{1-\epsilon_1-\epsilon_2}\bigr)}{\sqrt{1\vee \log\{n\epsilon_2^2/(1-\epsilon_1)\}}}\biggr\}, 
    \end{align*}
    where $c>0$ is a universal constant.
\end{Proposition}
\begin{proof}
    For $a>0$ to be specified later, let
    \begin{align}
        \tau \coloneqq \frac{\sigma^2 \log\bigl(\frac{1-\epsilon_1}{1-\epsilon_1-\epsilon_2}\bigr)}{2a} \label{eq:tau-def}
    \end{align}
    denote the unique point in $\mathbb{R}$ where $(1-\epsilon_1) \phi_{(-a,\sigma)}(\tau) = (1-\epsilon_1-\epsilon_2) \phi_{(a,\sigma)}(\tau)$.  Next, define the function $f_1: \mathbb{R} \rightarrow \mathbb{R}$ by
    \begin{align} \label{eq:f1-definition-univariate-lb}
    f_1(x) \coloneqq \begin{cases}
        (1-\epsilon_1-\epsilon_2)  \phi_{(-a,\sigma)}(x) & \text{ if } x \leq 0\\
        (1-\epsilon_1-\epsilon_2)  \phi_{(a,\sigma)}(x) & \text{ if } 0 < x \leq \tau\\
        (1-\epsilon_1) \phi_{(-a,\sigma)}(x) & \text{ if } x > \tau.
    \end{cases} 
    \end{align}
    Similarly, we note that $-\tau$ is the unique point satisfying $(1-\epsilon_1) \phi_{(a,\sigma)}(-\tau) = (1-\epsilon_1-\epsilon_2) \phi_{(-a,\sigma)}(-\tau)$ and define the function $f_2: \mathbb{R} \rightarrow \mathbb{R}$ by 
    \begin{align} \label{eq:f2-definition-univariate-lb}
    f_2(x) \coloneqq \begin{cases}
        (1-\epsilon_1) \phi_{(a,\sigma)}(x) & \text{ if } x \leq -\tau\\
        (1-\epsilon_1-\epsilon_2)  \phi_{(-a,\sigma)}(x) & \text{ if } -\tau < x \leq 0\\
        (1-\epsilon_1-\epsilon_2) \phi_{(a,\sigma)}(x) & \text{ if } x > 0.
    \end{cases}
    \end{align}
    Note that $\int_{\mathbb{R}} f_1(x) \, \mathrm{d}x = \int_{\mathbb{R}} f_2(x) \, \mathrm{d}x \leq 1-\epsilon_1 \leq 1$, so we may construct $M_1, M_2 \in \mathcal{P}(\mathbb{R}_{\star})$ with Radon--Nikodym derivatives
    \[
    \frac{\mathrm{d}M_{\ell}}{\mathrm{d} \lambda_{\star}}(z) \coloneqq f_{\ell}(z) \mathbbm{1}_{\{z \in \mathbb{R}\}} + \biggl(1 - \int_{\mathbb{R}} f_{\ell}(x)\, \mathrm{d}x\biggr) \mathbbm{1}_{\{z = \star\}} \quad \text{for} \quad \ell \in \{1, 2\},
    \]
    where $\lambda_{\star}$ denotes the extension of the Lebesgue measure to $\mathbb{R}_{\star}$. Then by Lemma~\ref{lemma:realisability-characterisation}, we have $M_{1} \in \mathcal{M}\bigl(N(-a, \sigma^2), \epsilon_1, \epsilon_2\bigr)$ and $M_{2} \in \mathcal{M}\bigl(N(a, \sigma^2), \epsilon_1, \epsilon_2\bigr)$. Since $M_{1}(\{\star\}) = M_{2}(\{\star\})$ and $f_1(x) = f_2(x)$ for $x \in [-\tau, \tau]$, we compute
    \begin{align} 
        \mathrm{KL}(&M_1, M_2) = \int_{-\infty}^{-\tau} (1-\epsilon_1-\epsilon_2) \phi_{(-a,\sigma)}(x) \log\biggl(\frac{(1-\epsilon_1-\epsilon_2) \phi_{(-a,\sigma)}(x)}{(1-\epsilon_1)\phi_{(a,\sigma)}(x)}\biggr)\, \mathrm{d}x  \nonumber\\
        & \qquad \qquad+ \int_{\tau}^{\infty} (1-\epsilon_1) \phi_{(-a,\sigma)}(x) \log\biggl(\frac{(1-\epsilon_1) \phi_{(-a,\sigma)}(x)}{(1-\epsilon_1-\epsilon_2) \phi_{(a,\sigma)}(x)}\biggr)\, \mathrm{d} x \nonumber\\
        &= (1-\epsilon_1-\epsilon_2) \biggl\{ \frac{2a^2}{\sigma^2} - \log\biggl(1+ \frac{\epsilon_2}{(1-\epsilon_1-\epsilon_2)} \biggr) \biggr\} \bigl\{1 - \Phi_{(0,\sigma)}(\tau-a)\bigr\}\nonumber\\
        &\qquad\qquad + (1-\epsilon_1) \biggl\{ \frac{2a^2}{\sigma^2} + \log\biggl(1+ \frac{\epsilon_2}{(1-\epsilon_1-\epsilon_2)} \biggr) \biggr\}\bigl\{1-\Phi_{(0,\sigma)}(\tau+a)\bigr\} \nonumber\\
        &\qquad\qquad + 2a\bigl\{ (1-\epsilon_1-\epsilon_2) \phi_{(0,\sigma)}(\tau - a) - (1-\epsilon_1) \phi_{(0,\sigma)}(\tau+a) \bigr\} \nonumber\\
        &= \frac{2a(1-\epsilon_1-\epsilon_2)}{\sigma^2}(a-\tau)\bigl\{1 - \Phi_{(0,\sigma)}(\tau-a)\bigr\} + \frac{2a(1-\epsilon_1)}{\sigma^2}(a+\tau)\bigl\{1-\Phi_{(0,\sigma)}(\tau+a)\bigr\} \nonumber\\
        &\qquad\qquad + 2a\bigl\{ (1-\epsilon_1-\epsilon_2) \phi_{(0,\sigma)}(\tau - a) - (1-\epsilon_1) \phi_{(0,\sigma)}(\tau+a) \bigr\}. \label{eq:KL-expression}
    \end{align}
    By the Mills ratio bound $1 - \Phi_{(0,\sigma)}(x) \leq \sigma^2 \phi_{(0,\sigma)}(x)/x$ for $x > 0$, we have
    \begin{align}
        &\frac{2a(1-\epsilon_1-\epsilon_2)}{\sigma^2}(a-\tau)\bigl\{1 - \Phi_{(0,\sigma)}(\tau-a)\bigr\} + \frac{2a(1-\epsilon_1)}{\sigma^2}(a+\tau)\bigl\{1-\Phi_{(0,\sigma)}(\tau+a)\bigr\} \nonumber\\
        &\hspace{0.4cm}= \frac{2a(1-\epsilon_1-\epsilon_2)}{\sigma^2}\cdot \tau \cdot \bigl\{\Phi_{(0,\sigma)}(\tau-a) - \Phi_{(0,\sigma)}(\tau+a)\bigr\} + \frac{2a(1-\epsilon_1-\epsilon_2)}{\sigma^2}\cdot a \cdot \bigl\{1-\Phi_{(0,\sigma)}(\tau-a)\bigr\} \nonumber\\
        &\hspace{2cm}+ \frac{2a(1-\epsilon_1-\epsilon_2)}{\sigma^2}\cdot a \cdot \bigl\{1- \Phi_{(0,\sigma)}(\tau+a)\bigr\} + \frac{2a\epsilon_2}{\sigma^2}\cdot (a+\tau) \cdot \bigl\{1-\Phi_{(0,\sigma)}(\tau+a)\bigr\} \nonumber\\
        &\hspace{0.4cm} \leq \frac{2a^2(1-\epsilon_1-\epsilon_2)}{\tau-a} \phi_{(0,\sigma)}(\tau-a) + \frac{2a^2(1-\epsilon_1-\epsilon_2)}{\tau+a} \phi_{(0,\sigma)}(\tau+a) + 2a\epsilon_2\phi_{(0,\sigma)}(\tau+a). \label{eq:KL-ub-1}
    \end{align}
    Moreover, by the mean value theorem, assuming that $\tau \geq a$,
    \begin{align}
        &2a\bigl\{ (1-\epsilon_1-\epsilon_2) \phi_{(0,\sigma)}(\tau - a) - (1-\epsilon_1) \phi_{(0,\sigma)}(\tau+a) \bigr\}\nonumber\\
        &\hspace{3cm}= 2a(1-\epsilon_1-\epsilon_2) \bigl\{\phi_{(0,\sigma)}(\tau - a) - \phi_{(0,\sigma)}(\tau+a)\bigr\} - 2a\epsilon_2\phi_{(0,\sigma)}(\tau+a)\nonumber\\
        &\hspace{3cm}\leq \frac{4a^2(1-\epsilon_1-\epsilon_2)(\tau+a)}{\sigma^2}\phi_{(0,\sigma)}(\tau - a) - 2a\epsilon_2\phi_{(0,\sigma)}(\tau+a). \label{eq:KL-ub-2}
    \end{align}
    Therefore, by~\eqref{eq:KL-expression}, \eqref{eq:KL-ub-1} and \eqref{eq:KL-ub-2}, we deduce that for $\tau\geq a$,
    \begin{align}
        \mathrm{KL}(M_1, M_2) &\leq 2a^2(1-\epsilon_1-\epsilon_2)\biggl\{ \frac{\phi_{(0,\sigma)}(\tau-a)}{\tau-a} + \frac{\phi_{(0,\sigma)}(\tau+a)}{\tau+a} + \frac{2(\tau+a)}{\sigma^2}\phi_{(0,\sigma)}(\tau - a) \biggr\}\nonumber\\
        &\leq 2a^2(1-\epsilon_1-\epsilon_2) \biggl\{ \frac{1}{\tau-a} + \frac{1}{\tau+a} + \frac{2(\tau+a)}{\sigma^2} \biggr\}\phi_{(0,\sigma)}(\tau - a). \label{eq:KL-ub-final}
    \end{align}

    We first consider the case where $\frac{\epsilon_2}{1-\epsilon_1} \geq \frac{2}{\sqrt{n(1-\epsilon_1)}}$, which implies $\frac{n\epsilon_2^2}{1-\epsilon_1} \geq 4$. Let
    \begin{align}
        a\coloneqq \frac{\sigma\log\bigl(\frac{1-\epsilon_1}{1-\epsilon_1-\epsilon_2}\bigr)}{4\sqrt{\log\{n\epsilon_2^2/(1-\epsilon_1)\}}}, \label{eq:a-def}
    \end{align}
    so that by substituting this into the definition of~$\tau$, we obtain $\tau = 2\sigma\sqrt{\log\{n\epsilon_2^2/(1-\epsilon_1)\}} > 2\sigma$, and by Lemma~\ref{lemma:a-tau-relation}, we have $a\leq \tau/8$. Therefore,~\eqref{eq:KL-ub-final} yields that
    \begin{align*}
        \mathrm{KL}(M_1, M_2) &\leq \frac{\sigma^2\log^2\bigl(\frac{1-\epsilon_1}{1-\epsilon_1-\epsilon_2}\bigr)}{8\log\{n\epsilon_2^2/(1-\epsilon_1)\}} \cdot (1-\epsilon_1-\epsilon_2) \cdot \frac{6\sqrt{\log\{n\epsilon_2^2/(1-\epsilon_1)\}}}{\sigma} \cdot \phi_{(0,\sigma)}(7\tau/8)\\
        &\leq \frac{3}{4\sqrt{2\pi}} \cdot \frac{\epsilon_2^2/(1-\epsilon_1)}{\sqrt{\log\{n\epsilon_2^2/(1-\epsilon_1-\epsilon_2)\}}} \cdot \biggl(\frac{1-\epsilon_1}{n\epsilon_2^2}\biggr)^{3/2}\\
        &\leq \frac{1}{3} \cdot \frac{(1-\epsilon_1)^{1/2}}{n^{3/2}\epsilon_2} \leq \frac{1}{6n},
    \end{align*}
    where the second inequality uses $\log\bigl(\frac{1-\epsilon_1}{1-\epsilon_1-\epsilon_2}\bigr) \leq \frac{\epsilon_2}{1-\epsilon_1}$.
    Thus, $\mathrm{KL}(M_1^{\otimes n}, M_2^{\otimes n}) \leq 1/6 < \log\bigl( \frac{1}{4\alpha(1-\alpha)} \bigr)$ for $\alpha\in(0,1/4]$, so by \citet[Corollary~6]{ma2024high}, we deduce that for $\alpha\in(0,1/4]$ and when $\frac{\epsilon_2}{1-\epsilon_1}\geq \frac{2}{\sqrt{n(1-\epsilon_1)}}$,
    \begin{align} 
        L_{\mathrm{est}}(n,\sigma,\epsilon_1,\epsilon_2) \geq a = \frac{\sigma\log\bigl(\frac{1-\epsilon_1}{1-\epsilon_1-\epsilon_2}\bigr)}{4\sqrt{\log\{n\epsilon_2^2/(1-\epsilon_1)\}}}. \label{eq:all-known-lb-1}
    \end{align}

    Next, since $(1-\epsilon_1)N(\theta,\sigma^2) + \epsilon_1\delta_{\{\star\}} \in \mathcal{M}\bigl(N(\theta, \sigma^2), \epsilon_1, \epsilon_2\bigr)$, we may use the MCAR lower bound from \citet[Proposition~S33(a)]{ma2024estimation} to deduce that
    \begin{align}
        L_{\mathrm{est}}(n,\sigma,\epsilon_1,\epsilon_2) \geq \frac{\sigma}{\sqrt{20n(1-\epsilon_1)}}. \label{eq:all-known-lb-2}
    \end{align}
    
    Hence, when $\frac{\epsilon_2}{1-\epsilon_1} \geq \frac{2}{\sqrt{n(1-\epsilon_1)}}$, we deduce the claim of the theorem by combining~\eqref{eq:all-known-lb-1} and~\eqref{eq:all-known-lb-2}. When $\frac{\epsilon_2}{1-\epsilon_1} < \frac{2}{\sqrt{n(1-\epsilon_1)}}$, we have $\log\bigl(\frac{1-\epsilon_1}{1-\epsilon_1-\epsilon_2}\bigr) \leq \frac{\epsilon_2}{1-\epsilon_1} < \frac{2}{\sqrt{n(1-\epsilon_1)}}$, so by~\eqref{eq:all-known-lb-2}, 
    \begin{align*}
        L_{\mathrm{est}}(n,\sigma,\epsilon_1,\epsilon_2) \geq \frac{\sigma}{\sqrt{20n(1-\epsilon_1)}} \geq \frac{\sigma}{3\sqrt{20}} \biggl\{\frac{1}{\sqrt{n(1-\epsilon_1)}} + \frac{\log\bigl(\frac{1-\epsilon_1}{1-\epsilon_1-\epsilon_2}\bigr)}{\sqrt{1\vee \log\{n\epsilon_2^2/(1-\epsilon_1)\}}}\biggr\}, 
    \end{align*}
    as required.
\end{proof}

\begin{Lemma} \label{lemma:a-tau-relation}
    Consider the setting of the proof of Proposition~\ref{prop:point-estimate-lb} and let $a$ and $\tau$ be defined as in~\eqref{eq:a-def} and~\eqref{eq:tau-def} respectively. If $\frac{\epsilon_2}{1-\epsilon_1} \geq \frac{2}{\sqrt{n(1-\epsilon_1)}}$ and $\epsilon_1+\epsilon_2 \leq 1-\frac{2}{n}$, then $a\leq \tau/8$.
\end{Lemma}
\begin{proof}
    By substituting the definitions of $a$ and $\tau$, it suffices to show that
    \begin{align*}
        \frac{1-\epsilon_1}{1-\epsilon_1-\epsilon_2} \leq \frac{n\epsilon_2^2}{1-\epsilon_1}.
    \end{align*}
    First note that $\frac{\epsilon_2}{1-\epsilon_1} \geq \frac{2}{\sqrt{n(1-\epsilon_1)}}$ implies $\frac{n\epsilon_2^2}{1-\epsilon_1} \geq 4$. Thus the claim is true when $\frac{1-\epsilon_1}{1-\epsilon_1-\epsilon_2} \leq 4$. Now suppose that $\frac{1-\epsilon_1}{1-\epsilon_1-\epsilon_2} > 4$. Then $\epsilon_2 > \frac{3}{4}(1-\epsilon_1)$. Hence,
    \begin{align*}
        n\epsilon_2^2 > \frac{n(1-\epsilon_1)^2}{2} \geq \frac{(1-\epsilon_1)^2}{1-\epsilon_1-\epsilon_2},
    \end{align*}
    where the final inequality uses $\epsilon_1+\epsilon_2 \leq 1-\frac{2}{n}$. This completes the proof.
\end{proof}

\subsection{\texorpdfstring{Unknown $\sigma$, known $\epsilon_1,\epsilon_2$}{Unknown sigma, known epsilon1, epsilon2}}
\subsubsection{Upper bound}
Let
\begin{align}
    t_{n,\epsilon_1,\epsilon_2,\alpha}' \coloneqq \sqrt{1\vee\bigl\{\log\bigl(n\epsilon_2^2/(1-\epsilon_1)\bigr)-\log\log(48/\alpha)\bigr\}/2}, \label{eq:t'-defn-2}
\end{align}
and observe that the coverage property of Proposition~\ref{prop:all-param-known-2}\emph{(a)} (with $\alpha$ replaced by $\alpha/3$) certainly implies that with probability at least $1-\alpha/3$, the interval $\hat{\mathrm{CI}}_{n,\sigma,\epsilon_1,\alpha}^{(3)}(t_{n,\epsilon_1,\epsilon_2,\alpha}')$ is non-empty; equivalently,
\begin{align}
    & Q_n^{-}\biggl((1-\epsilon_1)\{1-\Phi(t_{n,\epsilon_1,\epsilon_2,\alpha}')\}+3\sqrt{\frac{(1-\epsilon_1)\log(48/\alpha)}{n}}\biggr) \nonumber\\
    &\hspace{2cm} - Q_n^{+}\biggl((1-\epsilon_1)\{1-\Phi(t_{n,\epsilon_1,\epsilon_2,\alpha}')\}+3\sqrt{\frac{(1-\epsilon_1)\log(48/\alpha)}{n}}\biggr) + 2\sigma t_{n,\epsilon_1,\epsilon_2,\alpha}' \geq 0. \label{eq:sigma-lb-2}
\end{align}
The definition of $t_{n,\epsilon_1,\epsilon_2,\alpha}'$ involves a maximum with one as opposed to with zero in the definition of~$t_{n,\epsilon_1,\epsilon_2,\alpha}$ in Proposition~\ref{prop:all-param-known-2}.  Nevertheless, we have the following lemma, which is a slight modification of Proposition~\ref{prop:all-param-known-2}\emph{(b)}. 
\begin{Lemma} \label{lemma:sigma-ub-2}
    Let $n\in\mathbb{N}$, $\sigma>0$, $\alpha\in(0,1]$, $\epsilon_1,\epsilon_2 \in [0,1]$ be such that $\epsilon_1+\epsilon_2 \leq 1-\frac{4\log_2(24/\alpha)}{n}$, and $Z_1,\ldots,Z_n \overset{\mathrm{iid}}{\sim} M \in \mathcal{M}\bigl(N(\theta,\sigma^2),\epsilon_1,\epsilon_2\bigr)$. There exists a universal constant $C>0$ such that if $\alpha \in \bigl[ \frac{C}{n(1-\epsilon_1)}, 1\bigr]$, then with probability at least $1-\alpha/6$,
    \begin{align}
    & Q_n^{-}\biggl((1-\epsilon_1)\{1-\Phi(t_{n,\epsilon_1,\epsilon_2,\alpha}')\}+3\sqrt{\frac{(1-\epsilon_1)\log(48/\alpha)}{n}}\biggr)\nonumber\\
    &\hspace{1cm} - Q_n^{+}\biggl((1-\epsilon_1)\{1-\Phi(t_{n,\epsilon_1,\epsilon_2,\alpha}')\}+3\sqrt{\frac{(1-\epsilon_1)\log(48/\alpha)}{n}}\biggr) + 2\sigma t_{n,\epsilon_1,\epsilon_2,\alpha}' \nonumber\\
    &\hspace{6cm} \leq 2C_{\alpha/3}\sigma\biggl\{\frac{1}{\sqrt{n(1-\epsilon_1)}} + \frac{\log\bigl(\frac{1-\epsilon_1}{1-\epsilon_1-\epsilon_2}\bigr)}{\sqrt{1\vee \log\{n\epsilon_2^2/(1-\epsilon_1)\}}}\biggr\}, \label{eq:sigma-ub-2}
    \end{align}
    where $C_{\alpha}\geq 1$ is taken from Proposition~\ref{prop:all-param-known-2}.
\end{Lemma}
\begin{proof}
    We follow the same arguments and notation as in the proof of Proposition~\ref{prop:all-param-known-2}\emph{(b)}, but with $\alpha$ replaced by $\alpha/3$ and $t_* = t_{n,\epsilon_1,\epsilon_2,\alpha}'$. We will assume $C > 0$ is large enough that $n(1-\epsilon_1)\geq 1600 e^2\log(48/\alpha)$ and that $\frac{1}{2}\log\{n(1-\epsilon_1)/6400\} \geq \log\log\{48n(1-\epsilon_1)/C\}$ for $n(1-\epsilon_1) \geq C$. 

    First suppose that $\frac{\epsilon_2}{1-\epsilon_1} \leq \sqrt{\frac{e^2\log(48/\alpha)}{n(1-\epsilon_1)}}$, and let $b \coloneqq 40\sqrt{\frac{e^2\log(48/\alpha)}{n(1-\epsilon_1)}}$. In this case, we have $t_*=1$ and $b\leq 1$, so by Lemma~\ref{lemma:gaussian-cdf-ratio}\emph{(a)},
    \begin{align*}
        \frac{1-\epsilon_1-\epsilon_2}{1-\epsilon_1} \cdot \frac{1 - \Phi(t_*-b)}{1-\Phi(t_*)} &\geq \biggl(1 - \sqrt{\frac{e^2\log(48/\alpha)}{n(1-\epsilon_1)}}\biggr) \biggl(1 + 40\sqrt{\frac{e\log(48/\alpha)}{n(1-\epsilon_1)}}\biggr)\\
        &\geq 1 + 18e^{t_*^2}\sqrt{\frac{\log(48/\alpha)}{n(1-\epsilon_1)}}.
    \end{align*}
    This verifies~\eqref{eq:sufficient-cond-gaussian-known-eps-2} with $\alpha$ replaced by $\alpha/3$. Thus, following the argument in the proof of Proposition~\ref{prop:all-param-known-2}\emph{(b)}, in this case, we have $|\hat{\mathrm{CI}}_{n,\sigma,\alpha}(t_*)| \leq 2\sigma b = 80e\sigma \sqrt{\frac{\log(48/\alpha)}{n(1-\epsilon_1)}}$ with probability at least $1-\alpha/6$.

    For $\frac{\epsilon_2}{1-\epsilon_1} > \sqrt{\frac{e^2\log(48/\alpha)}{n(1-\epsilon_1)}}$, the same proof as in Cases~2--4 of Proposition~\ref{prop:all-param-known-2}\emph{(b)} holds, except that now $\alpha$ is replaced by $\alpha/3$.
\end{proof}

Thus, if we define
\begin{align}
&\hat{\sigma}_{\mathrm{L}} \coloneqq \frac{1}{{2t_{n,\epsilon_1,\epsilon_2,\alpha}'}} \cdot \biggl\{Q_n^{+}\biggl((1-\epsilon_1)\{1-\Phi(t_{n,\epsilon_1,\epsilon_2,\alpha}')\}+3\sqrt{\frac{(1-\epsilon_1)\log(48/\alpha)}{n}}\biggr)\nonumber\\
&\hspace{4cm} - Q_n^{-}\biggl((1-\epsilon_1)\{1-\Phi(t_{n,\epsilon_1,\epsilon_2,\alpha}')\}+3\sqrt{\frac{(1-\epsilon_1)\log(48/\alpha)}{n}}\biggr)\biggr\}, \label{eq:sigma-hat-2}
\end{align}
then combining the two inequalities~\eqref{eq:sigma-lb-2} and~\eqref{eq:sigma-ub-2} immediately yields the following corollary.

\begin{Corollary} \label{cor:sigma-CI-2}
Let $n\in\mathbb{N}$, $\sigma>0$, $\alpha\in(0,1]$, $\epsilon_1,\epsilon_2 \in [0,1]$ be such that $\epsilon_1+\epsilon_2 \leq 1-\frac{4\log_2(24/\alpha)}{n}$, and $Z_1,\ldots,Z_n \overset{\mathrm{iid}}{\sim} M \in \mathcal{M}\bigl(N(\theta,\sigma^2),\epsilon_1,\epsilon_2\bigr)$. Define
\begin{align}
    r_{n,\epsilon_1,\epsilon_2,\alpha} \coloneqq \frac{C_{\alpha/3}}{t'_{n,\epsilon_1,\epsilon_2,\alpha}} \biggl\{\frac{1}{\sqrt{n(1-\epsilon_1)}} + \frac{\log\bigl(\frac{1-\epsilon_1}{1-\epsilon_1-\epsilon_2}\bigr)}{\sqrt{1\vee \log\{n\epsilon_2^2/(1-\epsilon_1)\}}}\biggr\}. \label{eq:r-n-eps12-alpha}
\end{align}
There exists a universal constant $C>0$ such that if $\alpha \in \bigl[ \frac{C}{n(1-\epsilon_1)}, 1\bigr]$, then with probability at least $1-\alpha/2$, we have
\begin{equation*}
0\leq \frac{\sigma-\hat{\sigma}_{\mathrm{L}}}{\sigma}\leq r_{n,\epsilon_1,\epsilon_2,\alpha}. 
\end{equation*}
In particular, if we define
\begin{align}
    \hat{\sigma}_{\mathrm{R}} \coloneqq \frac{\hat{\sigma}_{\mathrm{L}}}{(1-r_{n,\epsilon_1,\epsilon_2,\alpha})_+}, \label{eq:hat-sigma-R}
\end{align}
then $\mathbb{P}(\hat{\sigma}_{\mathrm{L}} \leq \sigma \leq \hat{\sigma}_{\mathrm{R}}) \geq 1-\alpha/2$.
\end{Corollary}

Corollary~\ref{cor:sigma-CI-2} motivates the construction of a confidence interval that replaces the unknown~$\sigma$ in~\eqref{eq:gaussian-general-CI-a} with $\hat{\sigma}_{\mathrm{R}}$ to yield
\begin{align}
    \hat{\mathrm{CI}}^{(4)}_{n,\epsilon_1,\epsilon_2,\alpha} \coloneqq \begin{cases}
        \hat{\mathrm{CI}}_n^{(\mathrm{GST})} \bigl(t'_{n,\epsilon_1,\epsilon_2,\alpha};\; 1-\epsilon_1,\, 18(1-\epsilon_1),\, \alpha/12,\, \hat{\sigma}_{\mathrm{R}}\bigr)\quad&\text{if } r_{n,\epsilon_1,\epsilon_2,\alpha} \leq 0.9\\
        [Z_{\min}, Z_{\max}] \qquad&\text{if } r_{n,\epsilon_1,\epsilon_2,\alpha} > 0.9. 
    \end{cases}  \label{eq:gaussian-general-CI-b}
\end{align}

\begin{Proposition} \label{prop:unknown-sigma-2}
Let $n\in\mathbb{N}$, $\sigma>0$, $\alpha\in(0,1]$, $\epsilon_1,\epsilon_2 \in [0,1]$ be such that $\epsilon_1+\epsilon_2 \leq 1-\frac{4\log_2(24/\alpha)}{n}$, and $Z_1,\ldots,Z_n \overset{\mathrm{iid}}{\sim} M \in \mathcal{M}\bigl(N(\theta,\sigma^2),\epsilon_1,\epsilon_2\bigr)$. Then 
\begin{align*}
    \mathbb{P}\bigl(\theta \in \hat{\mathrm{CI}}^{(4)}_{n,\epsilon_1,\epsilon_2,\alpha}\bigr) \geq 1-\alpha.
\end{align*}
Moreover, there exists a universal constant $C>0$ such that if $\alpha \in \bigl[ \frac{C}{n(1-\epsilon_1)}, 1\bigr]$, then with probability at least $1-\alpha$, we have 
\begin{align*}
    \bigl|\hat{\mathrm{CI}}^{(4)}_{n,\epsilon_1, \epsilon_2,\alpha} \bigr| \leq C_{\alpha}'\sigma \biggl\{\frac{1}{\sqrt{n(1-\epsilon_1)}} + \frac{\log\bigl(\frac{1-\epsilon_1}{1-\epsilon_1-\epsilon_2}\bigr)}{\sqrt{1\vee \log\{n\epsilon_2^2/(1-\epsilon_1)\}}}\biggr\},
\end{align*}
where $C_{\alpha}'>0$ depends only on $\alpha$.
\end{Proposition}
\begin{proof}
    When $r_{n,\epsilon_1, \epsilon_2,\alpha} \leq 0.9$, the interval $\hat{\mathrm{CI}}^{(4)}_{n,\epsilon_1, \epsilon_2,\alpha}$ is the same as $\hat{\mathrm{CI}}^{(3)}_{n,\sigma,\epsilon_1,\alpha/3}(t'_{n,\epsilon_1, \epsilon_2,\alpha})$ but with $\sigma$ in the latter interval replaced by $\hat{\sigma}_{\mathrm{R}}$, so by Corollary~\ref{cor:sigma-CI-2} and Proposition~\ref{prop:all-param-known-2}\emph{(a)},
    \begin{align*}
        \mathbb{P}\bigl(\theta \notin \hat{\mathrm{CI}}^{(4)}_{n,\epsilon_1, \epsilon_2,\alpha}\bigr) &\leq \mathbb{P}(\sigma > \hat{\sigma}_{\mathrm{R}}) + \mathbb{P}\Bigl(\bigl\{\theta \notin \hat{\mathrm{CI}}^{(4)}_{n,\epsilon_1, \epsilon_2,\alpha}\bigr\} \cap \{\sigma \leq \hat{\sigma}_{\mathrm{R}}\}\Bigr)\\
        &\leq \frac{\alpha}{2} + \mathbb{P}\bigl(\theta \notin \hat{\mathrm{CI}}^{(3)}_{n,\sigma,\epsilon_1,\alpha/3}(t'_{n,\epsilon_1, \epsilon_2,\alpha})\bigr) \leq \alpha.
    \end{align*}
    When $r_{n,\epsilon_1, \epsilon_2,\alpha} > 0.9$, the coverage follows from~\eqref{eq:Zmax<theta-2} and the line following that bound.

    Regarding the length guarantee, suppose first that $r_{n,\epsilon_1, \epsilon_2,\alpha} \leq 0.9$. Then by Lemma~\ref{lemma:sigma-ub-2} and Corollary~\ref{cor:sigma-CI-2}, we have with probability at least $1-\alpha$ that
    \begin{align*}
        &Q_n^{-}\biggl((1-\epsilon_1)\{1-\Phi(t'_{n,\epsilon_1,\epsilon_2,\alpha})\}+3\sqrt{\frac{(1-\epsilon_1)\log(48/\alpha)}{n}}\biggr)\\
        &\hspace{2cm} - Q_n^+\biggl((1-\epsilon_1)\{1-\Phi(t'_{n,\epsilon_1,\epsilon_2,\alpha})\}-3\sqrt{\frac{(1-\epsilon_1)\log(48/\alpha)}{n}}\biggr) +2\hat{\sigma}_{\mathrm{R}} t_{n,\epsilon_1, \epsilon_2,\alpha}'\\
        &\leq 2C_{\alpha/3}\sigma \biggl\{\frac{1}{\sqrt{n(1-\epsilon_1)}} + \frac{\log\bigl(\frac{1-\epsilon_1}{1-\epsilon_1-\epsilon_2}\bigr)}{\sqrt{1\vee \log\{n\epsilon_2^2/(1-\epsilon_1)\}}}\biggr\} + 2(\hat{\sigma}_{\mathrm{R}} - \sigma) t_{n,\epsilon_1, \epsilon_2,\alpha}'\\
        &\leq 2C_{\alpha/3}\sigma \biggl\{\frac{1}{\sqrt{n(1-\epsilon_1)}} + \frac{\log\bigl(\frac{1-\epsilon_1}{1-\epsilon_1-\epsilon_2}\bigr)}{\sqrt{1\vee \log\{n\epsilon_2^2/(1-\epsilon_1)\}}}\biggr\} + 2\biggl(\frac{\sigma}{1-r_{n,\epsilon_1, \epsilon_2,\alpha}} - \sigma\biggr) t_{n,\epsilon_1, \epsilon_2,\alpha}'\\
        &\leq 2C_{\alpha/3}\sigma \biggl\{\frac{1}{\sqrt{n(1-\epsilon_1)}} + \frac{\log\bigl(\frac{1-\epsilon_1}{1-\epsilon_1-\epsilon_2}\bigr)}{\sqrt{1\vee \log\{n\epsilon_2^2/(1-\epsilon_1)\}}}\biggr\} + 20\sigma r_{n,\epsilon_1, \epsilon_2,\alpha} \cdot t_{n,\epsilon_1, \epsilon_2,\alpha}'\\
        &= 22C_{\alpha/3} \sigma \biggl\{\frac{1}{\sqrt{n(1-\epsilon_1)}} + \frac{\log\bigl(\frac{1-\epsilon_1}{1-\epsilon_1-\epsilon_2}\bigr)}{\sqrt{1\vee \log\{n\epsilon_2^2/(1-\epsilon_1)\}}}\biggr\}.
    \end{align*}
    If $r_{n,\epsilon_1,\epsilon_2,\alpha} > 0.9$, then $\frac{1}{\sqrt{n(1-\epsilon_1)}} + \frac{\log\bigl(\frac{1-\epsilon_1}{1-\epsilon_1-\epsilon_2}\bigr)}{\sqrt{1\vee \log\{n\epsilon_2^2/(1-\epsilon_1)\}}} > 0.9 C_{\alpha/3}^{-1} t_{n,\epsilon_1,\epsilon_2,\alpha}'$. By inspecting the proof of Proposition~\ref{prop:all-param-known-2}\emph{(b)}, we see that $C_{\alpha/3} \lesssim \sqrt{\log(e/\alpha)}$. Thus, $r_{n,\epsilon_1,\epsilon_2,\alpha} > 0.9$ implies
    either $\frac{1}{\sqrt{n(1-\epsilon_1)}} \gtrsim t_{n,\epsilon_1,\epsilon_2,\alpha}'/ \sqrt{\log(e/\alpha)}$ or $\frac{\log\bigl(\frac{1-\epsilon_1}{1-\epsilon_1-\epsilon_2}\bigr)}{\sqrt{1\vee \log\{n\epsilon_2^2/(1-\epsilon_1)\}}} \gtrsim t_{n,\epsilon_1,\epsilon_2,\alpha}'/\sqrt{\log(e/\alpha)}$.  In the first case, since $\epsilon_1+\epsilon_2 \leq 1-\frac{4\log_2(24/\alpha)}{n}$, we deduce that $n(1-\epsilon_1) \asymp_{\alpha} 1$.
    Then, by Lemma~\ref{lemma:max-min-2}, we have with probability at least $1-\alpha/2$, that
    \begin{align*}
        \bigl| \hat{\mathrm{CI}}^{(4)}_{n,\epsilon_1,\epsilon_2,\alpha} \bigr| &\leq 2\sigma\sqrt{2\log\{12n(1-\epsilon_1)/\alpha\}} \lesssim_{\alpha} \sigma \lesssim_{\alpha} \sigma \biggl\{\frac{1}{\sqrt{n(1-\epsilon_1)}} + \frac{\log\bigl(\frac{1-\epsilon_1}{1-\epsilon_1-\epsilon_2}\bigr)}{\sqrt{1\vee \log\{n\epsilon_2^2/(1-\epsilon_1)\}}}\biggr\}.
    \end{align*}
    In the second case, $\log\bigl(\frac{1-\epsilon_1}{1-\epsilon_1-\epsilon_2}\bigr) \gtrsim 1/\sqrt{\log(e/\alpha)}$, so $\frac{\epsilon_2}{1-\epsilon_1} \gtrsim 1$. Then, by Lemma~\ref{lemma:max-min-2}, we have with probability at least $1-\alpha/2$, that
    \begin{align*}
        \bigl| \hat{\mathrm{CI}}^{(4)}_{n,\epsilon_1,\epsilon_2,\alpha} \bigr| &\leq 2\sigma\sqrt{2\log\{12n(1-\epsilon_1)/\alpha\}} \lesssim \sigma\sqrt{\log\{n(1-\epsilon_1)\}}\\
        &\lesssim \sigma \sqrt{2\log\{n\epsilon_2^2/(1-\epsilon_1)\} - 2\log\log(48/\alpha)} \leq 2\sigma t_{n,\epsilon_1,\epsilon_2,\alpha}'\\
        &\lesssim_{\alpha} \sigma \biggl\{\frac{1}{\sqrt{n(1-\epsilon_1)}} + \frac{\log\bigl(\frac{1-\epsilon_1}{1-\epsilon_1-\epsilon_2}\bigr)}{\sqrt{1\vee \log\{n\epsilon_2^2/(1-\epsilon_1)\}}}\biggr\},
    \end{align*}
    where the second and third inequalities use the fact that $\frac{C}{n(1-\epsilon_1)} \leq \alpha \leq 1$ for a large enough universal constant $C > 0$.
\end{proof}

\subsubsection{Lower bound}
The lower bound for this case again follows from Proposition~\ref{prop:point-estimate-lb}.

\subsection{\texorpdfstring{Known $\sigma$, unknown $\epsilon_1,\epsilon_2$}{Known sigma, unknown epsilon1, epsilon2}}
\subsubsection{Upper bound}
Note that $\mathcal{M}\bigl(P,\epsilon_1,\epsilon_2\bigr) \subseteq \mathcal{M}\bigl(P,0,\epsilon_+\bigr)$ where $\epsilon_+=\epsilon_1+\epsilon_2$. Thus, we may use the confidence interval $\hat{\mathrm{CI}}^{(1)}_{n,\sigma,\alpha}(\hat{t}_{n,\sigma,\alpha})$ from Appendix~\ref{sec:known-sigma-unknown-eps}, which depends on $\sigma$ but not $(\epsilon_1,\epsilon_2)$, and apply Proposition~\ref{prop:known-sigma-unknown-epsilon} with $\epsilon=\epsilon_+$ to deduce the following result.
\begin{Proposition}
Let $n\in\mathbb{N}$, $\sigma>0$, $\alpha\in(0,1]$, and let $\epsilon_1,\epsilon_2 \in [0,1)$ be such that $\epsilon_1 + \epsilon_2 \leq 1-\frac{4\log_2(8/\alpha)}{n}$. Let $Z_1,\ldots,Z_n \overset{\mathrm{iid}}{\sim} M \in \mathcal{M}\bigl(N(\theta,\sigma),\epsilon_1,\epsilon_2\bigr)$. Then
  \begin{align*}
      \mathbb{P}\bigl(\theta \in \hat{\mathrm{CI}}^{(1)}_{n,\sigma,\alpha}(\hat{t}_{n,\sigma,\alpha})\bigr) \geq 1-\alpha. 
  \end{align*}
  Moreover, there exists a universal constant $C>0$ such that if $\alpha\in\bigl[\frac{C}{n}, 1\bigr]$, then with probability at least $1-\alpha/2$,
  \begin{align*}
      \bigl|\hat{\mathrm{CI}}^{(1)}_{n,\sigma,\alpha}(\hat{t}_{n,\sigma,\alpha})\bigr| \leq  C_{\alpha}\sigma\biggl\{\frac{1}{\sqrt{n}} + \frac{\log\bigl(\frac{1}{1-\epsilon_1-\epsilon_2}\bigr)}{\sqrt{1\vee \log\{n(\epsilon_1+\epsilon_2)^2\}}}\biggr\},
  \end{align*}
  where $C_{\alpha}\geq 1$ is defined as in Proposition~\ref{prop:all-param-known}.
\end{Proposition}

\subsubsection{Lower bound}
\begin{Proposition} \label{prop:unknown-eps-q-lb}
    Let $n \in \mathbb{N}$, $\sigma>0$, $(\epsilon_1,\epsilon_2)\in\mathcal{E}_n'$ and $\alpha\in(0,\frac{1}{9})$. Then
    \begin{align*}
        L_{n,\sigma,\epsilon_1,\epsilon_2}\bigl(\hat{\mathcal{CI}}(\sigma,\mathcal{E}_n')\bigr) \gtrsim \sigma \biggl\{\frac{1}{\sqrt{n}} + \frac{\log\bigl(\frac{1}{1-\epsilon_1-\epsilon_2}\bigr)}{\sqrt{1\vee \log\{n(\epsilon_1+\epsilon_2)^2\}}}\biggr\}. 
    \end{align*}
\end{Proposition}
\begin{proof}
    For simplicity of exposition, we will assume without loss of generality that $\sigma=1$, since the general case follows directly by scaling. We also write $L_{n,\sigma,\epsilon_1,\epsilon_2}\equiv L_{n,\sigma,\epsilon_1,\epsilon_2}\bigl(\hat{\mathcal{CI}}(\sigma,\mathcal{E}_n')\bigr)$ in this proof.
    For $r>0$, $M_1\in \mathcal{M}\bigl(N(0,1),0,1-\frac{1}{2n}\bigr)$ and $M_2 \in \mathcal{M}\bigl(N(r,1), \epsilon_1, \epsilon_2\bigr)$ to be chosen later,
    consider the testing problem $H_0: Z_1,\ldots,Z_n \overset{\mathrm{iid}}{\sim} M_1$ and $H_1: Z_1,\ldots,Z_n \overset{\mathrm{iid}}{\sim} M_2$. Let $\xi>0$ and let $\hat{\mathrm{CI}} \in \hat{\mathcal{CI}}$ be such that
    \begin{align*}
        \sup_{\theta\in\mathbb{R}} \;\sup_{M\in\mathcal{M}(N(\theta,\sigma^2),\epsilon_1,\epsilon_2)} \mathbb{P}_M\bigl(|\hat{\mathrm{CI}}(Z_1,\ldots,Z_n)| \geq (1+\xi)L_{n,\sigma,\epsilon_1,\epsilon_2} \bigr) \leq \alpha',
    \end{align*}
    where $\alpha'\coloneqq 1/3 - 2\alpha > \alpha$.  Define $T\coloneqq \mathbbm{1}_{\{0\notin \hat{\mathrm{CI}}(Z_1,\ldots,Z_n)\}}$.
    If $(1+\xi)L_{n,\sigma,\epsilon_1,\epsilon_2} \leq r$, then
    \begin{align}
        \mathbb{P}_{M_1}(T=1) &+ \mathbb{P}_{M_2}(T=0)\nonumber\\
        &\leq \alpha + \mathbb{P}_{M_2}\bigl(r\notin\hat{\mathrm{CI}}(Z_1,\ldots,Z_n)\bigr) + \mathbb{P}_{M_2}\bigl(|\hat{\mathrm{CI}}(Z_1,\ldots,Z_n)| \geq (1+\xi)L_{n,\sigma,\epsilon_1,\epsilon_2}\bigr) \nonumber\\
        &\leq 2\alpha+\alpha' = \frac{1}{3}. \label{eq:testing-lb-unkown-epsilon}
    \end{align}
    We now consider two cases and make specific choices of $r$, $M_1$ and $M_2$ in each to prove that~\eqref{eq:testing-lb-unkown-epsilon} cannot hold.  This will establish a contradiction, allowing us to conclude, since $\xi>0$ is arbitrary, that $L_{n,\sigma,\epsilon_1,\epsilon_2} \geq r$.

    \medskip
    \textbf{Case 1:} Suppose that $\frac{1}{2} \leq \epsilon_1+\epsilon_2 \leq 1-\frac{1}{2n}$ Let $\tau\coloneqq \frac{r}{2} + \frac{\log(\frac{1}{1-\epsilon_1-\epsilon_2})}{r}$ be the unique solution to $\phi(\tau) = (1-\epsilon_1-\epsilon_2)\phi(\tau-r)$. Let $r\coloneqq \frac{\log(\frac{1}{1-\epsilon_1-\epsilon_2})}{2\sqrt{\log n}} > 0$, and let $\gamma \in (-\infty, \tau)$ be the unique solution to $\frac{\phi(\gamma)}{2n} = (1-\epsilon_1-\epsilon_2) \phi(\gamma-r)$.
    Define $f_1:\mathbb{R} \rightarrow [0,\infty)$ by
    \begin{align*}
        f_1(x) \coloneqq
        \begin{cases}
            \frac{1}{2n} \cdot \phi(x) \quad&\text{if } x\leq \gamma\\
            (1-\epsilon_1-\epsilon_2) \phi(x-r)  \quad&\text{if } \gamma< x \leq \tau\\
            \phi(x) &\text{if } x> \tau.
        \end{cases}
    \end{align*}
    Further define $M_1,M_2 \in \mathcal{P}(\mathbb{R}_{\star})$ with Radon--Nikodym derivatives
    \begin{align*}
        \frac{\mathrm{d} M_1}{\mathrm{d} \lambda_{\star}} (z) \coloneqq f_1(z) \indi_{\{z\in\mathbb{R}\}} + \Bigl(1 - \int_{\mathbb{R}} f_1(x) \,\mathrm{d} x\Bigr)\indi_{\{z=\star\}},
    \end{align*}
    and
    \begin{align*}
        \frac{\mathrm{d} M_2}{\mathrm{d} \lambda_{\star}} (z) \coloneqq (1-\epsilon_1-\epsilon_2)\phi(z-r) \indi_{\{z\in\mathbb{R}\}} + (\epsilon_1+\epsilon_2)\indi_{\{z=\star\}}
    \end{align*}
    respectively.  Then by Lemma~\ref{lemma:realisability-characterisation}, $M_1 \in \mathcal{M}\bigl(N(0,1),0,1-1/(2n)\bigr)$ and $M_2\in \mathcal{M}\bigl(N(r,1),\epsilon_1,\epsilon_2\bigr)$.
    Moreover,
    \begin{align*}
    \mathrm{TV}(M_1,M_2) \leq \frac{1}{2n} + (1-\epsilon_1-\epsilon_2)\bigl(1-\Phi(\tau-r)\bigr) - \bigl(1-\Phi(\tau)\bigr).
    \end{align*}
    Now, 
    \begin{align*}
    \frac{(1-\epsilon_1-\epsilon_2)\bigl(1-\Phi(\tau-r)\bigr)}{1-\Phi(\tau)} &\leq (1-\epsilon_1-\epsilon_2) \cdot \frac{\tau+1/\tau}{\tau-r} \cdot \frac{\phi(\tau-r)}{\phi(\tau)} \\
    &\leq 3(1-\epsilon_1-\epsilon_2) \cdot \frac{\phi(\tau-r)}{\phi(\tau)} \leq 1 + \frac{e^{\tau^2/2}}{2n}
    \end{align*}
    where the first inequality follows by Mills ratio bounds (Lemma~\ref{lemma:mills-ratio}), the second inequality follows since $r\leq \tau/2$ when $\epsilon_1+\epsilon_2 \leq 1-\frac{1}{2n}$ and since $\tau \geq 2\sqrt{\log n} \geq 2$, and the third inequality follows by substituting $r$ and $\tau$ in the left-hand side and using the fact that $\tau \geq 2\sqrt{\log n}$ on the right-hand side. 
    Hence
    \begin{align*}
        (1-\epsilon_1-\epsilon_2)\bigl(1-\Phi(\tau-r)\bigr) - \bigl(1-\Phi(\tau)\bigr) \leq \frac{e^{\tau^2/2}}{2n} \cdot \{1-\Phi(\tau)\} \leq \frac{1}{2n},
    \end{align*}
    so $\mathrm{TV}(M_1,M_2) \leq \frac{1}{2n}$ and by \citet[Exercise~8.6\emph{(b)}]{samworth2026statistics}, 
    \begin{align*}
        \mathrm{TV}(M_1^{\otimes n}, M_2^{\otimes n}) \leq 1-\biggl(1-\frac{1}{2n}\biggr)^n < \frac{2}{3}.
    \end{align*}
    Therefore, 
    \begin{align*}
        \mathbb{P}_{M_1}(T=1) + \mathbb{P}_{M_2}(T=0) \geq 1 - \mathrm{TV}(M_1^{\otimes n},M_2^{\otimes n}) > \frac{1}{3}, 
    \end{align*}
    contradicting~\eqref{eq:testing-lb-unkown-epsilon}. We deduce that in this case, $L_{n,\sigma,\epsilon_1,\epsilon_2} \geq \frac{\log(\frac{1}{1-\epsilon_1-\epsilon_2})}{2\sqrt{\log n}}$.

    \begin{figure}[ht]
    \centering
    \begin{subfigure}{0.75\linewidth}
        \includegraphics[width=\linewidth]{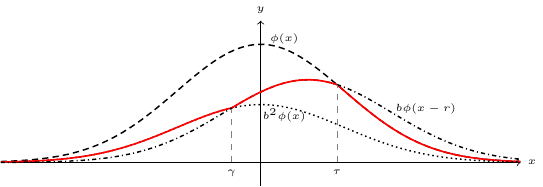}
        \caption{The red curve illustrates $m_1(x)$.}
    \end{subfigure}

    \begin{subfigure}{0.75\linewidth}
        \includegraphics[width=\linewidth]{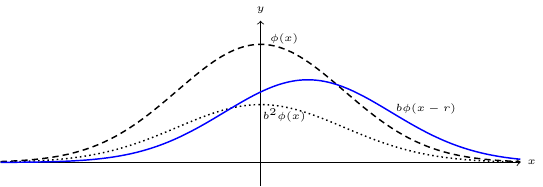}
        \caption{The blue curve illustrates $m_2(x)$.}
    \end{subfigure}
    
    \caption{Construction of the lower bound in the proof of case 2 in Proposition \ref{prop:unknown-eps-q-lb}.}\label{fig:unknown-eps-q-lb}
    \end{figure}

    \medskip
    \textbf{Case 2:} Suppose that $\frac{2}{\sqrt{n}} \leq \epsilon_1+\epsilon_2 < \frac{1}{2}$. Define $b\coloneqq 1-\epsilon_1-\epsilon_2$ and for $x\in\mathbb{R}$, define
    \begin{align*}
        m_1(x) \coloneqq b^2\phi(x) \vee b\phi(x-r) \wedge \phi(x)  \quad\text{and}\quad  m_2(x) \coloneqq b\phi(x-r),
    \end{align*}
    where $r>0$ will be chosen later.
    Let $\gamma\coloneqq \frac{r}{2} - \frac{\log (1/b)}{r}$ be the unique solution to $b^2\phi(\gamma) = b\phi(\gamma-r)$ and let $\tau\coloneqq \frac{r}{2} + \frac{\log (1/b)}{r}$ be the unique solution to $\phi(\tau) = b\phi(\tau-r)$. Let $M_1$ and $M_2$ be distributions on~$\mathbb{R}_{\star}$ whose Radon--Nikodym derivatives with respect to $\lambda_\star$, when restricted to $\mathbb{R}$, are given by~$m_1$ and $m_2$ respectively; see Figure~\ref{fig:unknown-eps-q-lb}.  By Lemma~\ref{lemma:realisability-characterisation} and since $b^2 \geq 1/4 > 1/(2n)$, we have $M_1\in  \mathcal{M}\bigl(N(0,1), 0, 1-1/(2n)\bigr)$ and $M_2\in \mathcal{M}\bigl(N(r,1), \epsilon_1, \epsilon_2\bigr)$. 
    Moreover, we have 
    \begin{align*}
        M_2(\mathbb{R}) - M_1(\mathbb{R}) &= \int_{\mathbb{R}} \bigl\{b\phi(x-r) - \phi(x)\bigr\}_+ \,\d x - \int_{\mathbb{R}} \bigl\{b^2\phi(x) - b\phi(x-r)\bigr\}_+ \,\d x \geq 0,
    \end{align*}
    so $M_2(\{\star\}) \leq M_1(\{\star\})$.
    Thus, if $\frac{\log (1/b)}{r} > \frac{r}{2}$, then
    \begin{align}
        \mathrm{KL}(M_2&,M_1) \leq \int_{\mathbb{R}} m_2(x)\log\biggl(\frac{m_2(x)}{m_1(x)}\biggr) \,\mathrm{d}x \nonumber\\
        & = \int_{-\infty}^\gamma b\phi(x-r) \log\biggl( \frac{b\phi(x-r)}{b^2\phi(x)} \biggr) \,\mathrm{d}x + \int_{\tau}^{\infty} b\phi(x-r) \log\biggl( \frac{b\phi(x-r)}{\phi(x)} \biggr) \,\mathrm{d}x \nonumber\\
        &= \int_{-\infty}^\gamma b\phi(x-r)\bigl\{rx-r^2/2 + \log (1/b)\bigr\} \,\d x + \int_{\tau}^{\infty} b\phi(x-r) \bigl\{ rx - r^2/2 - \log (1/b) \bigr\} \,\mathrm{d}x \nonumber\\
        &= b\bigl(r^2/2 + \log (1/b)\bigr) \Phi(\gamma-r) - br\phi(r-\gamma) + b\bigl(r^2/2 - \log(1/b)\bigr)\Phi(r-\tau) + br\phi(\tau-r) \nonumber\\
        &\leq \frac{br^2}{2}\bigl\{\Phi(\gamma-r) + \Phi(r-\tau)\bigr\} + br\bigl\{\phi(\tau-r) - \phi(r-\gamma)\bigr\} \nonumber\\
        &\leq br^2 \cdot \frac{\phi\bigl(\log(1/b)/r - r/2\bigr)}{\log(1/b)/r - r/2} + br^2\bigl(\log(1/b)/r - r/2\bigr) \phi\bigl(\log(1/b)/r - r/2\bigr), \label{eq:unknown-eps-q-lb-regime-2}
    \end{align}
    where the penultimate inequality follows since $\Phi(\gamma-r) \leq \Phi(r-\tau)$; the final inequality follows by applying the Mills ratio bound (Lemma~\ref{lemma:mills-ratio}) to the first term, using the mean value theorem for the second term, and substituting the definitions of $\tau$ and $\gamma$. Now let $r\coloneqq \frac{\log (1/b)}{4\sqrt{\log\{n(1-b)^2\}}}$. Note that since $b\in[1/2,\, 1-2/\sqrt{n}]$, we have $\log\{n(1-b)^2\} > \log(1/b)$, $0 < r \leq \sqrt{\log\{n(1-b)^2\}}$ and 
    \begin{align}
        3\sqrt{\log\{n(1-b)^2\}} \leq \frac{\log(1/b)}{r} - \frac{r}{2} \leq 4\sqrt{\log\{n(1-b)^2\}} \leq \frac{\log(1/b)}{r} + \frac{r}{2} \leq 5\sqrt{\log\{n(1-b)^2\}}. \label{eq:logb-r-ineq}
    \end{align}
    Since $\frac{\log(1/b)}{r} > \frac{r}{2}$, we may apply~\eqref{eq:logb-r-ineq} to~\eqref{eq:unknown-eps-q-lb-regime-2} to deduce that
    \begin{align*}
        \mathrm{KL}(M_2,M_1) &\leq \frac{\log^2 (1/b)}{16\log\{n(1-b)^2\}} \cdot \frac{\phi\bigl(3\sqrt{\log\{n(1-b)^2\}}\bigr)}{4\sqrt{\log\{n(1-b)^2\}}}\\
        &\hspace{6cm} + \frac{\log^2 (1/b)}{4\sqrt{\log\{n(1-b)^2\}}}\cdot \phi\bigl(3\sqrt{\log\{n(1-b)^2\}}\bigr)\\
        &\leq \frac{(1-b)^2}{\{n(1-b)^2\}^4} = \frac{1}{n} \cdot \frac{1}{\{n(1-b)^2\}^3} \leq \frac{1}{64n},
    \end{align*}
    where the second inequality uses the fact that $\log^2 (1/b) \leq 3(1-b)^2$ for $b\in[1/2,1]$ and the final inequality uses the fact that $b\leq 1-2/\sqrt{n}$. Then, by Pinsker's inequality, 
    \begin{align*}
        \mathbb{P}_{M_1}(T=1) + \mathbb{P}_{M_2}(T=0)  \geq 1 - \mathrm{TV}(M_1^{\otimes n},M_2^{\otimes n}) \geq 1 - \sqrt{\frac{n}{2} \mathrm{KL}(M_2,M_1)}> \frac{1}{3}, 
    \end{align*}
    again contradicting \eqref{eq:testing-lb-unkown-epsilon}. We deduce that in this case, $L_{n,\sigma,\epsilon_1,\epsilon_2} \geq \frac{\log(\frac{1}{1-\epsilon_1-\epsilon_2})}{4\sqrt{\log \{n(\epsilon_1+\epsilon_2)^2\}}}$.

    \medskip
    Finally, when $\epsilon_1+\epsilon_2 \leq \frac{2}{\sqrt{n}}$, we have $L_{n,\sigma,\epsilon_1,\epsilon_2} \geq L_{n,\sigma,\epsilon_1+\epsilon_2,0} \gtrsim \frac{1}{\sqrt{n}}$ by the MCAR lower bound \citep[Proposition~S33(a)]{ma2024estimation}.
    The final claim follows by combining the three cases.
\end{proof}

\subsection{\texorpdfstring{Unknown $\sigma,\epsilon_1,\epsilon_2$}{Unknown sigma, epsilon1, epsilon2}}
\subsubsection{Upper bound}
Again, since $\mathcal{M}\bigl(P,\epsilon_1,\epsilon_2\bigr) \subseteq \mathcal{M}\bigl(P,0,\epsilon_+\bigr)$ where $\epsilon_+=\epsilon_1+\epsilon_2$, we may use the confidence interval $\hat{\mathrm{CI}}^{(1)}_{n,\sigma_{\max},\alpha}(\hat{t}_{n,\sigma_{\max},\alpha})$ from Appendix~\ref{sec:unknown-sigma-eps}, which does not depend on $\sigma,\epsilon_1,\epsilon_2$, and apply Proposition~\ref{prop:unknown-eps-sigma} with $\epsilon=\epsilon_+$ to deduce the following result.

\begin{Proposition}
    Let $n\in\mathbb{N}$, $\sigma\in(0,\sigma_{\max}]$, $\alpha\in(0,1]$, and let $\epsilon_1,\epsilon_2\in[0,1)$ be such that $ \epsilon_1+\epsilon_2 \leq 1-\frac{4\log_2(8/\alpha)}{n}$. Let $Z_1,\ldots,Z_n \overset{\mathrm{iid}}{\sim} M \in \mathcal{M}\bigl(N(\theta,\sigma^2),\epsilon_1,\epsilon_2\bigr)$. Then 
    \begin{align*}
        \mathbb{P}\bigl(\theta \in \hat{\mathrm{CI}}^{(1)}_{n,\sigma_{\max},\alpha}(\hat{t}_{n,\sigma_{\max},\alpha}) \bigr) \geq 1-\alpha.
    \end{align*}
    Moreover, there exist universal constants $C,C'>0$ such that if $\alpha\in\bigl[\frac{C}{n}, 1\bigr]$, then with probability at least $1-\alpha/2$, we have
    \begin{align*}
    \bigl| \hat{\mathrm{CI}}^{(1)}_{n,\sigma_{\max},\alpha}(\hat{t}_{n,\sigma_{\max},\alpha}) \bigr| \leq C'\sigma_{\max}\biggl\{\biggl(\sqrt{\frac{\log(4/\alpha)}{n}} + \epsilon_1 + \epsilon_2 \biggr) \vee \biggl(\sqrt{\log\Bigl(\frac{1}{1-\epsilon_1 - \epsilon_2}\Bigr)} - 1 \biggr)\biggr\}.
    \end{align*}
\end{Proposition}

\subsubsection{Lower bound}
\begin{Proposition} \label{prop:all-unknown-lb}
    Let $n \in \mathbb{N}$, $\sigma>0$, $(\epsilon_1,\epsilon_2)\in\mathcal{E}_n'$ and $\alpha\in(0,\frac{1}{9})$. Then
    \begin{align*}
        L_{n,\sigma,\epsilon_1,\epsilon_2}\bigl( \widehat{\mathcal{CI}}(\mathcal{S},\mathcal{E}_{n}') \bigr) \gtrsim \sigma \biggl\{\biggl(\frac{1}{\sqrt{n}} + \epsilon_1 + \epsilon_2 \biggr) \vee \biggl(\sqrt{\log\Bigl(\frac{1}{1-\epsilon_1 - \epsilon_2}\Bigr)} - 1 \biggr)\biggr\}.
    \end{align*}
\end{Proposition}
\begin{proof}
    We write $L_{n,\sigma,\epsilon_1,\epsilon_2}\equiv L_{n,\sigma,\epsilon_1,\epsilon_2}\bigl( \widehat{\mathcal{CI}}(\mathcal{S},\mathcal{E}_{n}') \bigr)$ in this proof. For $r>0$, $\gamma\in(0,\sigma_{\max}]$, $M_1\in \mathcal{M}\bigl(N(0,\sigma^2),0,1-1/(2n)\bigr)$ and $M_2 \in \mathcal{M}\bigl(N(r,\gamma^2), \epsilon_1, \epsilon_2\bigr)$ to be chosen later,
    consider the testing problem $H_0:Z_1,\ldots,Z_n \overset{\mathrm{iid}}{\sim} M_1$ against $H_1: Z_1,\ldots,Z_n \overset{\mathrm{iid}}{\sim} M_2$.
    Let $\xi>0$ and let $\hat{\mathrm{CI}} \in \hat{\mathcal{CI}}$ be such that
    \begin{align*}
        \sup_{\theta\in\mathbb{R}} \;\sup_{M\in\mathcal{M}(N(\theta,\gamma^2),\epsilon_1,\epsilon_2)} \mathbb{P}_M\bigl(|\hat{\mathrm{CI}}(Z_1,\ldots,Z_n)| \geq (1+\xi)L(n,\sigma,\epsilon_1,\epsilon_2) \bigr) \leq \alpha',
    \end{align*}
    where $\alpha'\coloneqq 1/3-2\alpha > \alpha$.
    Define a test statistic $T\coloneqq \mathbbm{1}_{\{0\notin \hat{\mathrm{CI}}(Z_1,\ldots,Z_n)\}}$. If $(1+\xi)L_{n,\sigma,\epsilon_1,\epsilon_2} \leq r$, then
    \begin{align}
        \mathbb{P}_{M_1}(T=1) &+ \mathbb{P}_{M_2}(T=0)\nonumber\\
        &\leq \alpha + \mathbb{P}_{M_2}\bigl(r\notin\hat{\mathrm{CI}}(Z_1,\ldots,Z_n)\bigr) + \mathbb{P}_{M_2}\bigl(|\hat{\mathrm{CI}}(Z_1,\ldots,Z_n)| \geq (1+\xi)L_{n,\sigma,\epsilon_1,\epsilon_2}\bigr) \nonumber\\
        &\leq 2\alpha+\alpha' = \frac{1}{3}. \label{eq:testing-lb-unkown-epsilon-sigma}
    \end{align}
    We now make specific choices of $r$, $M_1$ and $M_2$, and later consider two different choices of $\gamma \geq \sigma(1-\epsilon_1 - \epsilon_2)$, to prove that~\eqref{eq:testing-lb-unkown-epsilon-sigma} cannot hold.  This will establish a contradiction, allowing us to conclude, since $\xi>0$ is arbitrary, that $L_{n,\sigma,\epsilon_1,\epsilon_2} \geq r$.
    
    Let 
    \begin{align*}
        r\coloneqq \sqrt{2(\sigma^2 - \gamma^2)\log\biggl(\frac{\gamma}{\sigma(1-\epsilon_1-\epsilon_2)}\biggr)},
    \end{align*}
    so that $(1-\epsilon_1-\epsilon_2)\phi_{(r,\gamma)}(x) \leq \phi_{(0,\sigma)}(x)$ for all $x\in\mathbb{R}$ and $\gamma \geq \sigma(1-\epsilon_1 - \epsilon_2)$. Define 
    \begin{align*}
        m_1 \coloneqq \frac{1}{2n}\phi_{(0,\sigma)} \vee (1-\epsilon_1-\epsilon_2)\phi_{(r,\gamma)} \quad\text{and}\quad m_2 \coloneqq (1-\epsilon_1-\epsilon_2)\phi_{(r,\gamma)}.
    \end{align*}
    Let $M_1$ and $M_2$ be distributions on $\mathbb{R}_{\star}$ whose densities restricted to $\mathbb{R}$ are given by $m_1$ and $m_2$ respectively. Then, by Lemma~\ref{lemma:realisability-characterisation}, we have $M_1\in \mathcal{M}\bigl(N(0,\sigma^2),0,1-\frac{1}{2n}\bigr)$ and $M_2\in \mathcal{M}\bigl(N(r,\gamma^2),\epsilon_1,\epsilon_2\bigr)$. Moreover, $\mathrm{TV}(M_1,M_2) \leq 1/(2n)$, so
    \begin{align*}
        \mathbb{P}_{H_0}(T=1) + \mathbb{P}_{H_1}(T=0) \geq 1 - \mathrm{TV}(M_1^{\otimes n},M_2^{\otimes n}) \geq \biggl(1 - \frac{1}{2n}\biggr)^n \geq \frac{1}{2},
    \end{align*}
    which contradicts~\eqref{eq:testing-lb-unkown-epsilon-sigma}. Thus $L_{n,\sigma,\epsilon_1,\epsilon_2} \geq r = \sqrt{2(\sigma^2 - \gamma^2)\log\bigl(\frac{\gamma}{\sigma(1-\epsilon_1-\epsilon_2)}\bigr)}$.  It remains to choose $\gamma \geq \sigma(1-\epsilon_1 - \epsilon_2)$.  When $\epsilon_1+\epsilon_2 > 0.2$, we choose $\gamma = 0.9\sigma \geq \sigma(1-\epsilon_1 - \epsilon_2)$ so that $L_{n,\sigma,\epsilon_1,\epsilon_2} \gtrsim \sigma \sqrt{\log\bigl(\frac{1}{1-\epsilon_1-\epsilon_2}\bigr)}$. When $\epsilon_1+\epsilon_2 \leq 0.2$, we choose $\gamma = \bigl(1-\frac{\epsilon_1+\epsilon_2}{2}\bigr)\sigma \geq \sigma(1-\epsilon_1 - \epsilon_2)$. Then
    \begin{align*}
        L_{n,\sigma,\epsilon_1,\epsilon_2} &\geq \sigma \sqrt{2(\epsilon_1+\epsilon_2) - \frac{(\epsilon_1+\epsilon_2)^2}{2}} \cdot \sqrt{\log\biggl(\frac{1-(\epsilon_1 + \epsilon_2)/2}{1-(\epsilon_1 + \epsilon_2)}\biggr)}\\
        &\geq \sigma\sqrt{1.9(\epsilon_1 + \epsilon_2)} \cdot \sqrt{\log\biggl(1 + \frac{\epsilon_1 + \epsilon_2}{2}\biggr)} \geq \sigma(\epsilon_1+\epsilon_2)/2,
    \end{align*}
    where the last inequality uses $\log(1+x) \geq 0.9x$ for all $x\in[0,0.1]$. Finally, by the MCAR lower bound \citep[Proposition~S33(a)]{ma2024estimation}, we have that $L_{n,\sigma,\epsilon_1,\epsilon_2} \gtrsim \frac{\sigma}{\sqrt{n}}$. Thus, $L_{n,\sigma,\epsilon_1,\epsilon_2} \gtrsim \sigma\bigl(\frac{1}{\sqrt{n}} + \epsilon_1+\epsilon_2\bigr)$ when $\epsilon_1+\epsilon_2\leq 0.2$.  The final conclusion follows by combining the two cases.
\end{proof}

\section{Proof of Theorem \ref{thm:nonparametric-minimax-rates}} \label{sec:proof-of-nonparametric-rates}
Again, we will provide a proof for the upper bound and a proof for the lower bound in each case.

\subsection{Bounded random variables}

\begin{Lemma} \label{lemma:bounded-rv-bias}
    Let $\theta\in\mathbb{R}$ and $P\in \mathcal{P}_{\mathrm{BD}}(\theta,[a,b])$.
    Let $Z \sim M \in \mathsf{MNAR}_P$ and $p\coloneqq \mathbb{P}(Z \neq \star)$. Then
    \begin{align*}
        p\mathbb{E}(Z\,|\,Z\neq\star) + (1-p)a \leq \theta \leq p\mathbb{E}(Z\,|\,Z\neq\star) + (1-p)b.
    \end{align*}
\end{Lemma}
\begin{proof}
    There exist $X\sim P$ and a binary random variable $\Omega$ such that $\mathbb{P}(\Omega=1)=p$ and $X\ostar\Omega \overset{\mathrm{d}}{=} Z$. Thus, we have
    \begin{align*}
        p\mathbb{E}(Z\,|\,Z\neq\star) + (1-p)a &= \mathbb{P}(Z \neq \star) \mathbb{E}(Z\,|\,Z\neq\star) + \mathbb{P}(Z = \star)a\\
        &= \mathbb{E}(Z\indi_{\{Z\neq\star\}} + a\indi_{\{Z=\star\}}) = \mathbb{E}(X\indi_{\{\Omega=1\}} + a\indi_{\{\Omega=0\}}) \leq \mathbb{E}(X) = \theta.
    \end{align*}
    Similarly,
    \begin{align*}
        p\mathbb{E}(Z\,|\,Z\neq\star) + (1-p)b = \mathbb{E}(X\indi_{\{\Omega=1\}} + b\indi_{\{\Omega=0\}}) \geq \mathbb{E}(X) = \theta,
    \end{align*}
    as required.
\end{proof}

Recalling $\hat{\mathrm{CI}}^{(\mathrm{BD})}_{n,a,b,\alpha}$ from~\eqref{eq:CI-BD}, we have the following proposition.
\begin{Proposition}
    Let $(\epsilon_1,\epsilon_2) \in \mathcal{E}$ and $P\in\mathcal{P}_{\mathrm{BD}}(\theta,[a,b])$. Let $Z_1,\ldots,Z_n \overset{\mathrm{iid}}{\sim} M \in \mathcal{M}(P,\epsilon_1,\epsilon_2)$. Then
    \begin{align*}
        \mathbb{P}\Bigl(\theta \in \hat{\mathrm{CI}}^{(\mathrm{BD})}_{n,a,b,\alpha}\Bigr) \geq 1-\alpha.
    \end{align*}
    Moreover, if $\epsilon_1 + \epsilon_2 \leq 1-\frac{8\log(4/\alpha)}{n}$, then with probability at least $1-\alpha/2$,
    \begin{align*}
        \bigl| \hat{\mathrm{CI}}^{(\mathrm{BD})}_{n,a,b,\alpha} \bigr| \leq C_{\alpha} (b-a) \biggl(\frac{1}{\sqrt{n}} + \epsilon_1 + \epsilon_2\biggr),
    \end{align*}
    where $C_{\alpha}>0$ depends only on $\alpha$.
\end{Proposition}
\begin{proof}
    \textbf{Coverage:} Let $p\coloneqq \mathbb{P}(Z\neq\star)$. Define the event
    \begin{align}
    \label{Eq:E1}
        E_1 \coloneqq \biggl\{ |\narrowhat{p}_n - p| \leq \sqrt{\frac{2\narrowhat{p}_n(1-\narrowhat{p}_n)\log(6/\alpha)}{n}} + \frac{3\log(6/\alpha)}{n} \biggr\}.
    \end{align}
    By the empirical Bernstein inequality \citep[Theorem~1]{audibert2009exploration}, we have $\mathbb{P}(E_1) \geq 1-\alpha/2$. Moreover, on $E_1$, we have $\narrowhat{p}_{n,\alpha}^- \leq p$. Next, define
    \begin{align*}
        E_2 \coloneqq \biggl\{ \bigl|\bar{Z}_n - \mathbb{E}(Z_1\,|\, Z_1 \neq \star)\bigr| \leq (b-a)\sqrt{\frac{\log(4/\alpha)}{2n\narrowhat{p}_n}} \biggr\}.
    \end{align*}
    Let $Y_1,\ldots,Y_n \overset{\mathrm{iid}}{\sim} \mathsf{Law}(Z_1 \,|\, Z_1 \neq\star)$.
    By Hoeffding's inequality, we deduce that for $k \in [n]$,
    \begin{align*}
        \mathbb{P}(E_2\,|\,n\narrowhat{p}_n = k) = \mathbb{P}\biggl( \biggl|\frac{1}{k}\sum_{i=1}^k Y_i - \mathbb{E}(Y_1)\biggr| \leq (b-a)\sqrt{\frac{\log(4/\alpha)}{2k}} \biggr) \geq 1-\frac{\alpha}{2}.
    \end{align*}
    Moreover, when $k=0$, our convention that $\bar{Z}_n=0$ means that this bound continues to hold.  Hence, $\mathbb{P}(E_2) = \mathbb{E}\bigl\{\mathbb{P}(E_2\,|\,n\narrowhat{p}_n)\bigr\} \geq 1- \alpha/2$.
    Therefore, by Lemma~\ref{lemma:bounded-rv-bias}, we have on $E_1 \cap E_2$ that 
    \begin{align*}
        \narrowhat{p}_{n,\alpha}^- \bar{Z}_n + (1-\narrowhat{p}_{n,\alpha}^-)a - (b-a)\sqrt{\frac{\log(4/\alpha)}{2n\narrowhat{p}_n}} &\leq p \bar{Z}_n + (1-p)a - (b-a)\sqrt{\frac{\log(4/\alpha)}{2n\narrowhat{p}_n}}\\
        &\leq p\mathbb{E}(Z\,|\,Z\neq\star) + (1-p)a \leq \theta.
    \end{align*}
    Similarly, on $E_1\cap E_2$, we have $\narrowhat{p}_{n,\alpha}^- \bar{Z}_n + (1-\narrowhat{p}_{n,\alpha}^-)b + (b-a)\sqrt{\frac{\log(4/\alpha)}{2n\narrowhat{p}_n}} \geq p\mathbb{E}(Z\,|\,Z\neq\star) + (1-p)b \geq \theta$.
    Thus, by a union bound, $\mathbb{P}\bigl(\theta\in \hat{\mathrm{CI}}^{(\mathrm{BD})}_{n,a,b,\alpha}\bigr) \geq 1-\alpha$.

    \medskip
    \textbf{Length:} Since $p \geq 1 - \epsilon_1 - \epsilon_2 \geq \frac{8\log(4/\alpha)}{n}$, we have by a multiplicative Chernoff bound that the event 
    \begin{align*}
        E_3 \coloneqq \biggl\{\narrowhat{p}_n \geq \frac{p}{2} \biggr\}
    \end{align*}
    has probability at least $1-\alpha/4$. Moreover, by Hoeffding's inequality, the event
    \begin{align*}
        E_4 \coloneqq \biggl\{\narrowhat{p}_n \geq p - \sqrt{\frac{\log(4/\alpha)}{2n}}\biggr\}
    \end{align*}
    has probability at least $1-\alpha/4$. Thus, on $E_3 \cap E_4$, we have
    \begin{align*}
        \bigl| \hat{\mathrm{CI}}^{(\mathrm{BD})}_{n,a,b,\alpha} \bigr| &= (b-a)(1-\narrowhat{p}^-_{n,\alpha}) + 2(b-a)\sqrt{\frac{\log(4/\alpha)}{2n\narrowhat{p}_n}}\\
        &\lesssim_{\alpha} (b-a) \biggl(1-\narrowhat{p}_n + \frac{1}{\sqrt{n\narrowhat{p}_n}}\biggr)
         \lesssim_{\alpha} (b-a) \biggl(1-p + \frac{1}{\sqrt{np}}\biggr) \lesssim_{\alpha} (b-a) \biggl(\frac{1}{\sqrt{n}} + \epsilon_1 + \epsilon_2\biggr),
    \end{align*}
    where the last inequality uses the fact that $p\geq 1-\epsilon_1-\epsilon_2\geq \frac{8\log(4/\alpha)}{n}$.
\end{proof}

In fact, for fixed $\epsilon_+ < 1$, as $n\to\infty$, we have $\narrowhat{p}_n \to p$ and $\narrowhat{p}^-_{n,\alpha} \to p$ in probability uniformly over~$\mathcal{P}_{\mathrm{BD}}$. Moreover, since $p>1-\epsilon_+>0$ and the confidence interval viewed as a function of $\narrowhat{p}_n$ and $\narrowhat{p}^-_{n,\alpha}$ is uniformly continuous on $[\frac{1-\epsilon_+}{2}, 1]$, we deduce that
\begin{align}
    \lim_{n \rightarrow \infty} \sup_{P\in\mathcal{P}_{\mathrm{BD}}} \sup_{M\in\mathcal{M}(P,\epsilon_1,\epsilon_2)} \mathbb{P}_M\bigl(|\hat{\mathrm{CI}}_{n,a,b,\alpha}^{(\mathrm{BD})}| > (b-a)\epsilon_+ \bigr) = 0.  \label{eq:BD-CI-asym-length}
\end{align}

\subsubsection{Lower bound}
\begin{Proposition}
    Let $n\in\mathbb{N}$, $(\epsilon_1,\epsilon_2) \in\mathcal{E}$ and $\alpha\in(0,\frac{1}{4}]$.
    Then 
    \begin{align*}
        L_{n,\epsilon_1,\epsilon_2}\bigl(\hat{\mathcal{CI}}(\mathcal{P}_{\mathrm{BD}},\mathcal{E})\bigr) \geq \frac{c(b-a)}{\sqrt{n}} \vee (b-a)\epsilon_+,
    \end{align*}
    where $c>0$ is a universal constant.
\end{Proposition}
\begin{proof}
    Write $L_{n,\epsilon_1,\epsilon_2}\equiv L_{n,\epsilon_1,\epsilon_2}\bigl(\hat{\mathcal{CI}}(\mathcal{P}_{\mathrm{BD}},\mathcal{E}_n')\bigr)$ in this proof.
    For $\theta_1,\theta_2\in[a,b]$, $P_1\in \mathcal{P}_{\mathrm{BD}}(\theta_1,[a,b])$, $P_2\in \mathcal{P}_{\mathrm{BD}}(\theta_2,[a,b])$, $M_1\in \mathcal{M}\bigl(P_1,0,1\bigr)$ and $M_2 \in \mathcal{M}\bigl(P_2, \epsilon_1, \epsilon_2\bigr)$ to be chosen later, consider the problem of testing  $H_0: Z_1,\ldots,Z_n \overset{\mathrm{iid}}{\sim} M_1$ against $H_1: Z_1,\ldots,Z_n \overset{\mathrm{iid}}{\sim} M_2$. Let $\xi>0$ and let $\hat{\mathrm{CI}} \in \hat{\mathcal{CI}}$ be such that
    \begin{align*}
        \sup_{\theta\in\mathbb{R},\, P\in\mathcal{P}_{\mathrm{BD}}(\theta,[a,b]),\, M\in\mathcal{M}(P,\epsilon_1,\epsilon_2)} \mathbb{P}_M\bigl(|\hat{\mathrm{CI}}(Z_1,\ldots,Z_n)| \geq (1+\xi)L (n,\epsilon_1,\epsilon_2)\bigr) \leq \alpha',
    \end{align*}
    where $\alpha'\coloneqq 9/10 - 2\alpha > \alpha$.  Define $T\coloneqq \mathbbm{1}_{\{\theta_1\notin \hat{\mathrm{CI}}(Z_1,\ldots,Z_n)\}}$.
    If $(1+\xi)L_{n,\epsilon_1,\epsilon_2} \leq |\theta_1-\theta_2|$, then
    \begin{align}
        \mathbb{P}_{M_1}(T=1) &+ \mathbb{P}_{M_2}(T=0)\nonumber\\
        &\leq \alpha + \mathbb{P}_{M_2}\bigl(\theta_2\notin\hat{\mathrm{CI}}(Z_1,\ldots,Z_n)\bigr) + \mathbb{P}_{M_2}\bigl(|\hat{\mathrm{CI}}(Z_1,\ldots,Z_n)| \geq (1+\xi)L_{n,\epsilon_1,\epsilon_2}\bigr) \nonumber\\
        &\leq 2\alpha+\alpha' = \frac{9}{10}. \label{eq:testing-lb-unkown-epsilon-bd}
    \end{align}
    We now make specific choices of $\theta_1$, $\theta_2$, $P_1$, $P_2$ $M_1$ and $M_2$ to prove that~\eqref{eq:testing-lb-unkown-epsilon-bd} cannot hold.  This will establish a contradiction, allowing us to conclude, since $\xi>0$ is arbitrary, that $L_{n,\epsilon_1,\epsilon_2} \geq |\theta_1-\theta_2|$.

    Let $\theta_1 \coloneqq (1-\epsilon_1-\epsilon_2)a + (\epsilon_1+\epsilon_2)b$ and
    define $P_1\in \mathcal{P}_{\mathrm{BD}}(\theta_1,[a,b])$ by
    \begin{align*}
        P_1(\{a\}) \coloneqq 1-\epsilon_1 - \epsilon_2 \quad\text{and}\quad P_1(\{b\}) \coloneqq \epsilon_1 + \epsilon_2.
    \end{align*}
    Let $\theta_2\coloneqq a$ and define $P_2\in \mathcal{P}_{\mathrm{BD}}(\theta_2,[a,b])$ by
    \begin{align*}
        P_2(\{a\}) \coloneqq 1.
    \end{align*}
    Further define a distribution $M$ on $\mathbb{R}_{\star}$ by
    \begin{align*}
        M(\{a\}) \coloneqq 1-\epsilon_1-\epsilon_2 \quad\text{and}\quad M(\{\star\}) \coloneqq \epsilon_1 + \epsilon_2.
    \end{align*}
    Then $\frac{M(\{a\})}{P_1(\{a\})}, \frac{M(\{b\})}{P_1(\{b\})}, \frac{M(\{a\})}{P_2(\{a\})} \in [0, 1]$ so by Lemma~\ref{lemma:realisability-characterisation}, we have both $M\in\mathcal{M}(P_1,0,1)$ and $M\in\mathcal{M}(P_2,\epsilon_1,\epsilon_2)$. Thus, letting $M_1 \coloneqq M$ and $M_2 \coloneqq M$, we have
    \begin{align*}
        \mathbb{P}_{M_1}(T=1) + \mathbb{P}_{M_2}(T=0) = 1 > \frac{9}{10},
    \end{align*}
    but this contradicts~\eqref{eq:testing-lb-unkown-epsilon-bd}. Hence, $L_{n,\epsilon_1,\epsilon_2} \geq |\theta_1-\theta_2| = (b-a)(\epsilon_1+\epsilon_2)$. Finally, since $\mathcal{M}(P,\epsilon_1+\epsilon_2,0) \subseteq \mathcal{M}(P,\epsilon_1,\epsilon_2)$, we have by the MCAR lower bound of \citet[][Proposition S33\emph{(b)}]{ma2024estimation} that there exists a universal constant $c > 0$ such that $L_{n,\epsilon_1,\epsilon_2} \geq c(b-a)/\sqrt{n}$. Therefore,  $L_{n,\epsilon_1,\epsilon_2} \geq c(b-a)/\sqrt{n} \vee (b-a)(\epsilon_1+\epsilon_2)$.
\end{proof}

\subsection{Sub-Gaussian random variables}
\subsubsection{Upper bound}
\begin{Lemma} \label{lemma:sub-gaussian-bias}
    Let $\theta\in\mathbb{R}$, $\sigma>0$ and $P\in\mathcal{P}_{\mathrm{SG}}(\theta,\sigma)$.
    Let $Z \sim M \in \mathsf{MNAR}_P$, let $p\coloneqq \mathbb{P}(Z \neq \star)$ and let $Y \sim \mathsf{Law}(Z\,|\,Z\neq\star)$ where $\mathsf{Law}(Z\,|\,Z\neq\star)$ is the conditional distribution of $Z$ given $\{Z\neq \star\}$. Then
    \begin{align}
        \bigl|\mathbb{E}(Y) - \theta \bigr| \leq \sigma r^{(\mathrm{SG})}(p) \label{eq:sample-mean-bias}
    \end{align}
    and
    \begin{align}
        \|Y-\mathbb{E}(Y)\|_{\psi_2} \leq 4\sigma\sqrt{\log(2/p)}. \label{eq:SG-norm-bound}
    \end{align}
\end{Lemma}

\begin{proof}
    By translation invariance it suffices to prove the claim for $\theta=0$.
    Let $\lambda$ be a Borel measure on~$\mathbb{R}$ such that $P\ll\lambda$, and let $f\coloneqq \frac{\mathrm{d}P}{\mathrm{d}\lambda}$. We may assume without loss of generality that $Z$ is generated via $X\sim P$ and a binary random variable $\Omega$ with $\mathbb{P}(\Omega=1) = \mathbb{P}(Z \neq \star)$ in such a way that $Z=X\ostar \Omega$. Define $m:\mathbb{R}\to [0,1]$ by $m(x) \coloneqq \mathbb{P}(\Omega=1 \,|\, X=x)$. Then $\int_{\mathbb{R}} m(x)f(x) \,\mathrm{d}\lambda (x) = p$ and $\frac{\mathrm{d}M}{\mathrm{d}\lambda_{\star}}(x) = m(x)f(x)$ for $x\in\mathbb{R}$, where $\lambda_\star$ is the extended measure on $\mathbb{R}_\star$ corresponding to $\lambda$. 
    Let $Q$ be the conditional distribution of $Y$. Then
    \begin{align*}
        \frac{\mathrm{d}Q}{\mathrm{d}\lambda}(x) = \frac{m(x)f(x)}{p}  \quad\text{and}\quad
        \frac{\mathrm{d}Q}{\mathrm{d}P}(x) = \frac{m(x)}{p}.
    \end{align*}
    Therefore, by Jensen's inequality,
    \begin{align*}
        \mathrm{KL}(Q,P) = \int_{\mathbb{R}} \log\biggl(\frac{m(x)}{p}\biggr) \,\mathrm{d}Q(x) \leq \log\biggl(\frac{\int_{\mathbb{R}} m(x) \,\mathrm{d}Q(x)}{p}\biggr) \leq \log(1/p).
    \end{align*}
    Hence, by the Donsker--Varadhan variational characterisation of KL divergence \citep[e.g.][Corollary~4.15]{boucheron2003concentration}, we have for any $t\in\mathbb{R}$,
    \begin{align*}
        \mathbb{E}_{Y\sim Q}(tY) \leq \mathrm{KL}(Q,P) + \log \mathbb{E}_{X\sim P} (e^{tX}) \leq \log(1/p) + \frac{\sigma^2 t^2 \log2}{2},
    \end{align*}
    where the final inequality uses the fact that if $\|X\|_{\psi_2} \leq \sigma$, then $\mathbb{E}(e^{tX}) \leq e^{(\sigma^2 t^2 \log2)/2}$ for all $t \in \mathbb{R}$ \citep[Theorem~1.1]{leskela2026sharp}.
    Taking $t=\sigma^{-1}\sqrt{\frac{2\log(1/p)}{\log2}} > 0$ yields $\mathbb{E}(Y) \leq \sigma\sqrt{\log4 \cdot \log(1/p)}$.
    Similarly, taking $t= -\sigma^{-1}\sqrt{\frac{2\log(1/p)}{\log2}} > 0$ yields $\mathbb{E}(Y) \geq -\sigma\sqrt{\log4 \cdot \log(1/p)}$, so 
    \begin{align*}
        \bigl|\mathbb{E}(Y) - 0 \bigr| \leq \sigma\sqrt{\log4 \cdot \log(1/p)}.
    \end{align*}
    On the other hand, we also have
    \begin{align*}
        \bigl|\mathbb{E}(Y) - 0 \bigr| &= \frac{\bigl|\mathbb{E}\bigl\{Xm(X)\bigr\}\bigr|}{p} = \frac{\bigl|\mathbb{E}\bigl\{X\bigl(1-m(X)\bigr)\bigr\}\bigr|}{p}\\
        &\leq \inf_{t>0} \frac{1}{p} \Bigl\{\bigl|\mathbb{E}\bigl\{X\bigl(1-m(X)\bigr) \mathbbm{1}_{\{|X|\leq t\}} \bigr\}\bigr| + \mathbb{E}\bigl(|X|\mathbbm{1}_{\{|X|> t\}} \bigr)\bigr|\Bigr\}\\
        &\leq \inf_{t>0} \frac{1}{p} \Bigl\{ t(1-p) + t\mathbb{P}(|X|\geq t) + \int_t^{\infty} \mathbb{P}(|X| \geq s) \,\mathrm{d}s \Bigr\}\\
        &\leq \inf_{t>0} \frac{1}{p} \Bigl\{ t(1-p) + 2te^{-t^2/\sigma^2} + 2\int_t^{\infty} e^{-s^2/\sigma^2} \,\mathrm{d}s \Bigr\}\\
        &= \inf_{t>0} \frac{1}{p} \Bigl\{ t(1-p) + 2te^{-t^2/\sigma^2} + 2\sigma\sqrt{\pi}\bigl\{ 1-\Phi(\sqrt{2}t/\sigma) \bigr\} \Bigr\}\\
        &\leq \frac{\sigma(1-p)}{p} \biggl\{ 2\log^{1/2}\biggl(\frac{2}{1-p}\biggr) + \frac{1}{2}\log^{-1/2}\biggl(\frac{2}{1-p}\biggr)\biggr\},
    \end{align*}
    where in the penultimate inequality, we used $\mathbb{P}(|X| \geq s) = \mathbb{P}(e^{X^2/\sigma^2} \geq e^{s^2/\sigma^2}) \leq 2e^{-s^2/\sigma^2}$ for $s\geq 0$, and in the final inequality, we take $t = \sigma\sqrt{\log(\frac{2}{1-p})}$ and use the Mills ratio bound. Thus,
    \begin{align*}
        \bigl|\mathbb{E}(Y) - 0 \bigr|
        &\leq \sigma\sqrt{\log4 \cdot \log(1/p)} \wedge \frac{\sigma(1-p)}{p} \biggl\{ 2\log^{1/2}\biggl(\frac{2}{1-p}\biggr) + \frac{1}{2}\log^{-1/2}\biggl(\frac{2}{1-p}\biggr)\biggr\}\\
        &\leq \sigma\sqrt{\log4 \cdot \log(1/p)} \wedge \frac{11\sigma(1-p)}{5p} \log^{1/2}\biggl(\frac{2}{1-p}\biggr),
    \end{align*}
    where the final inequality follows since $\frac{1}{2}\log^{-1/2}\bigl(\frac{2}{1-p}\bigr) \leq \frac{1}{5} \log^{1/2}\bigl(\frac{2}{1-p}\bigr)$ for $p\geq 0.9$, and since the first term attains the minimum in the penultimate expression when $p<0.9$.

    \medskip
    For the second claim, note that
    \begin{align}
        \mathbb{P}\bigl(|Y| \geq s\bigr) = \frac{1}{p} \int_{|x|\geq s} m(x)f(x) \,\mathrm{d}\lambda(x) \leq \frac{1}{p} \mathbb{P}\bigl(|X| \geq s\bigr) \leq 1 \wedge \frac{2\exp(-s^2/\sigma^2)}{p} \label{eq:Y-tail-bound}
    \end{align}
    for all $s\geq 0$. Thus, 
    \begin{align*}
        \mathbb{E}\biggl\{\exp\biggl(\frac{Y^2}{6\sigma^2\log(2/p)}\biggr)\biggr\} &= \int_0^\infty \mathbb{P}\biggl\{\exp\biggl(\frac{Y^2}{6\sigma^2\log(2/p)}\biggr) \geq t\biggr\} \,\mathrm{d}t\\
        &= \int_0^\infty \mathbb{P}\Bigl(|Y| \geq \sqrt{6\sigma^2 \log(2/p)\log t}\Bigr) \,\mathrm{d}t\\
        &\leq e^{1/2} + \frac{2}{p} \int_{e^{1/2}}^\infty \exp\Bigl(-2\log(2/p)\log t - 4\log(2/p)\log t\Bigr)\,\mathrm{d}t\\
        &\leq e^{1/2} + \frac{2}{p} \int_{e^{1/2}}^\infty \exp\Bigl(-\log(2/p) - 4\log2\cdot\log t\Bigr)\,\mathrm{d}t\\
        &= e^{1/2} + \frac{e^{1/2}}{4(4\log 2-1)} \leq 2,
    \end{align*}
    where the first inequality follows from~\eqref{eq:Y-tail-bound}, so $\|Y\|_{\psi_2} \leq \sigma\sqrt{6\log(2/p)}$. Finally, since $\|\cdot\|_{\psi_2}$ is a norm \citep[Exercise~2.5.7]{vershynin2018high}, we deduce that
    \begin{align*}
        \|Y - \mathbb{E}(Y)\|_{\psi_2} \leq \|Y\|_{\psi_2} + \|\mathbb{E}(Y)\|_{\psi_2} \leq \sigma\sqrt{6\log(2/p)} + \frac{|\mathbb{E}(Y)|}{\sqrt{\log 2}} \leq 4\sigma\sqrt{\log(2/p)},
    \end{align*}
    where the final inequality uses the fact that $|\mathbb{E}(Y)| \leq \sigma\sqrt{\log 4 \cdot \log(1/p)}$.
\end{proof}

Recalling $\hat{\mathrm{CI}}^{(\mathrm{SG})}_{n,\sigma_{\max},\alpha}$ from~\eqref{eq:CI-SG}, we have the following proposition.
\begin{Proposition}\label{prop:sub-gaussian-CI}
    Let $(\epsilon_1,\epsilon_2) \in \mathcal{E}$, $\theta\in\mathbb{R}$, $\sigma\in(0,\sigma_{\max}]$ and $P\in\mathcal{P}_{\mathrm{SG}}(\theta,\sigma)$. Let $Z_1,\ldots,Z_n \overset{\mathrm{iid}}{\sim} M\in \mathcal{M}(P,\epsilon_1,\epsilon_2)$. Then
    \begin{align*}
        \mathbb{P}\Bigl(\theta \in \hat{\mathrm{CI}}^{(\mathrm{SG})}_{n,\sigma_{\max},\alpha} \Bigr) \geq 1-\alpha.
    \end{align*}
    Moreover, there exists a universal constant $c>0$ such that if $\alpha\in\bigl[6e^{-cn}, 1\bigr]$ and $\epsilon_1 + \epsilon_2 \leq 1-\frac{27\log(6/\alpha)}{n}$, then with probability at least $1-\alpha/2$,
    \begin{align*}
        \bigl| \hat{\mathrm{CI}}^{(\mathrm{SG})}_{n,\sigma_{\max},\alpha} \bigr| \leq C_{\alpha}\sigma_{\max} \Biggl\{ \frac{1}{\sqrt{n}} + \sqrt{\log\biggl(\frac{1}{1-\epsilon_1-\epsilon_2}\biggr)} \wedge \frac{\epsilon_1+\epsilon_2}{1-\epsilon_1-\epsilon_2} \sqrt{\log\biggl(\frac{1}{\epsilon_1+\epsilon_2}\biggr)}\Biggr\},
    \end{align*}
    where $C_\alpha>0$ depends only on $\alpha$.
\end{Proposition}
\begin{proof}
    \textbf{Coverage:} Let $p\coloneqq \mathbb{P}(Z_1 \neq \star)$. Recall the definition of the event $E_1$ from~\eqref{Eq:E1}, as well as the facts that $\mathbb{P}(E_1) \geq 1-\alpha/2$ and $\narrowhat{p}_{n,\alpha}^- \leq p$ on $E_1$. Define
    \begin{align*}
        E_2 \coloneqq \biggl\{\bigl|\bar{Z}_n - \mathbb{E}(Z_1\,|\, Z_1 \neq \star)\bigr| \leq 4\sigma_{\max}\sqrt{\frac{\log(4)\log(2/p)\log(4/\alpha)}{n \narrowhat{p}_n}}\biggr\}.
    \end{align*}
    Let $Y_1,\ldots,Y_n \overset{\mathrm{iid}}{\sim} \mathsf{Law}(Z_1 \,|\, Z_1 \neq\star)$.
    By~\eqref{eq:SG-norm-bound} and \citet[Theorem~1.1]{leskela2026sharp}, we deduce that for $k \in [n]$,
    \begin{align*}
        \mathbb{P}(E_2\,|\,n\narrowhat{p}_n = k) = \mathbb{P}\biggl( \biggl|\frac{1}{k}\sum_{i=1}^k Y_i - \mathbb{E}(Y_1)\biggr| \leq 4\sigma_{\max}\sqrt{\frac{\log(4)\log(2/p)\log(4/\alpha)}{k}} \biggr) \geq 1-\frac{\alpha}{2}.
    \end{align*}
    Moreover, when $k=0$, our convention that $\bar{Z}_n=0$ means that this bound continues to hold.  Hence, $\mathbb{P}(E_2) = \mathbb{E}\bigl\{\mathbb{P}(E_2\,|\,n\narrowhat{p}_n)\bigr\} \geq 1- \alpha/2$.
    Therefore, by~\eqref{eq:sample-mean-bias} and since $r^{(\mathrm{SG})}(\cdot)$ is decreasing, we deduce that on $E_1 \cap E_2$, we have $\theta\in \hat{\mathrm{CI}}^{(\mathrm{SG})}_{n,\sigma_{\max},\alpha}$. Hence $\mathbb{P}\bigl(\theta\in \hat{\mathrm{CI}}^{(\mathrm{SG})}_{n,\sigma_{\max},\alpha}\bigr) \geq 1-\alpha$ by a union bound.

    \medskip
    \textbf{Length:} Let $c_0\in(0,1)$ be the unique solution to the following equation
    \begin{align*}
        \sqrt{\log4 \cdot \log(1/c_0)} =  \frac{11(1-c_0)}{5c_0} \log^{1/2}\biggl(\frac{2}{1-c_0}\biggr).
    \end{align*}
    First, suppose that $p > 1- 1/n$. By Bernstein's inequality \citep[e.g.][Theorem~A.6.12 or Exercise~A.6.10]{samworth2026statistics}, we deduce that the event $E_3 \coloneqq \bigl\{\narrowhat{p}_n \geq 1-\frac{3\log(2/\alpha)}{n}\bigr\}$ has probability at least $1-\alpha/2$. Thus, on $E_3$, we have $1-\narrowhat{p}_{n,\alpha}^- \leq 9\log(6/\alpha)/n$. Therefore, since $z \mapsto z\sqrt{\log(2/z)}$ is increasing on $[0,1]$, we deduce that on $E_3$, and provided we choose the universal constant $c$ in the statement to lie in $(0,1/18]$,
    \begin{align*}
        r^{(\mathrm{SG})}(\narrowhat{p}_{n,\alpha}^-) \lesssim (1-\narrowhat{p}_{n,\alpha}^-) \sqrt{\log\biggl(\frac{2}{1-\narrowhat{p}_{n,\alpha}^-}\biggr)} \leq \frac{9\log(6/\alpha)}{n} \sqrt{\log\biggl(\frac{2n}{9\log(6/\alpha)}\biggr)} \lesssim_{\alpha} \frac{1}{\sqrt{n}}.
    \end{align*}
    Hence, we have with probability at least $1-\alpha/2$ that
    \begin{align*}
        \Bigl| \hat{\mathrm{CI}}^{(\mathrm{SG})}_{n,\sigma,\alpha} \Bigr| \lesssim_{\alpha} \frac{\sigma_{\max}}{\sqrt{n}}.
    \end{align*}
    
    Next, suppose that $c_0 \leq p \leq 1- 1/n$.
    By choosing the universal constant $c > 0$ in the theorem statement small enough, we have $c_0/2\leq \narrowhat{p}_{n,\alpha}^-\leq p$ on the event $E_1$.
    Let $f(x) \coloneqq (1-x)\log^{1/2}\bigl(\frac{2}{1-x}\bigr)$ for $x \in (0,1)$.  By the mean value theorem and the concavity of $f$, on $E_1$,
    \begin{align*}
        r^{(\mathrm{SG})}(\narrowhat{p}_{n,\alpha}^-) &\leq \frac{11(1-\narrowhat{p}_{n,\alpha}^-)}{5\narrowhat{p}_{n,\alpha}^-} \log^{1/2}\biggl(\frac{2}{1-\narrowhat{p}_{n,\alpha}^-}\biggr) \lesssim f(\narrowhat{p}_{n,\alpha}^-)\\
        &\leq f(p) + \Biggl\{ \sqrt{\log\biggl(\frac{2}{1-p}\biggr)} - \frac{1}{2\sqrt{\log\bigl(\frac{2}{1-p}\bigr)}} \Biggr\} (p - \narrowhat{p}_{n,\alpha}^-)\\
        &\leq f(p) + \Biggl\{ \sqrt{\log\biggl(\frac{2}{1-p}\biggr)} - \frac{1}{2\sqrt{\log\bigl(\frac{2}{1-p}\bigr)}} \Biggr\} \cdot 2\biggl\{ \sqrt{\frac{2\narrowhat{p}_n(1-\narrowhat{p}_n)\log(6/\alpha)}{n}} + \frac{3\log(6/\alpha)}{n} \biggr\}\\
        &\lesssim_{\alpha} f(p) + \sqrt{\log\biggl(\frac{2}{1-p}\biggr) \cdot \frac{1-\narrowhat{p}_n}{n}} + \frac{1}{n}\sqrt{\log\biggl(\frac{2}{1-p}\biggr)}\\
        &\lesssim f(p) + \sqrt{\log\biggl(\frac{2}{1-p}\biggr) \cdot \frac{1-p + n^{-1/2}}{n}} + \frac{\sqrt{\log n}}{n}\\
        &\lesssim f(p) + \sqrt{\log\biggl(\frac{2}{1-p}\biggr) \cdot \frac{1-p}{n}} + \sqrt{\frac{\log n}{n^{3/2}}} + \frac{1}{\sqrt{n}} \lesssim f(p) + \frac{1}{\sqrt{n}} \leq r^{(\mathrm{SG})}(p) + \frac{1}{\sqrt{n}}.
    \end{align*}
    Hence, in this case, we have with probability at least $1-\alpha/2$ that
    \begin{align*}
        \Bigl| \hat{\mathrm{CI}}^{(\mathrm{SG})}_{n,\sigma_{\max},\alpha} \Bigr| \lesssim_{\alpha} \sigma_{\max} \biggl\{ \frac{1}{\sqrt{n}} + r^{(\mathrm{SG})}(p) \biggr\}.
    \end{align*}

    Finally, suppose that $0<p<c_0$. Since $\epsilon_1+\epsilon_2 \leq 1-\frac{27\log(6/\alpha)}{n}$, we have $p\geq \frac{27\log(6/\alpha)}{n}$. By a multiplicative Chernoff bound, we have that the event 
    \begin{align*}
        E_4 \coloneqq \biggl\{\frac{2p}{3} \leq \narrowhat{p}_n \leq \frac{4p}{3} \biggr\}
    \end{align*}
    has probability at least $1-\alpha/3$. Thus, we have on $E_4$ that 
    \begin{align*}
        \narrowhat{p}_{n,\alpha}^- &\geq \frac{2p}{3} - \sqrt{\frac{2\narrowhat{p}_n(1-\narrowhat{p}_n)\log(6/\alpha)}{n}} - \frac{3\log(6/\alpha)}{n}\\
        &\geq \frac{2p}{3} - \sqrt{\frac{3p\log(6/\alpha)}{n}} - \frac{3\log(6/\alpha)}{n} \geq \frac{2p}{3} - \frac{p}{3} - \frac{p}{9} = \frac{2p}{9},
    \end{align*}
    where the final inequality uses the fact that $p\geq \frac{27\log(6/\alpha)}{n}$.
    Therefore, we deduce that on $E_4$,
    \begin{align*}
        r^{(\mathrm{SG})}(\narrowhat{p}_{n,\alpha}^-) \lesssim \sqrt{\log(1/\narrowhat{p}_{n,\alpha}^-)} \leq \sqrt{\log(9/2p)} \lesssim r^{(\mathrm{SG})}(p).
    \end{align*}
    Hence, on $E_4$, we have
    \begin{align*}
        \Bigl| \hat{\mathrm{CI}}^{(\mathrm{SG})}_{n,\sigma,\alpha} \Bigr| \lesssim_{\alpha} \sigma_{\max} \biggl\{\sqrt{\frac{\log(9/p)}{np}} + r^{(\mathrm{SG})}(p)\biggr\} \lesssim \sigma_{\max} r^{(\mathrm{SG})}(p).
    \end{align*}

    Combining all the three cases above, we deduce that with probability at least $1-\alpha/2$,
    \begin{align*}
        \bigl| \hat{\mathrm{CI}}^{(\mathrm{SG})}_{n,\sigma,\alpha} \bigr| &\lesssim_{\alpha} \sigma_{\max} \biggl\{ \frac{1}{\sqrt{n}} + r^{(\mathrm{SG})}(p)\biggr\} \leq \sigma_{\max} \biggl\{ \frac{1}{\sqrt{n}} + r^{(\mathrm{SG})}(1-\epsilon_1-\epsilon_2) \biggr\}\\
        &\lesssim \sigma_{\max} \Biggl\{ \frac{1}{\sqrt{n}} + \sqrt{\log\biggl(\frac{1}{1-\epsilon_1-\epsilon_2}\biggr)} \wedge \frac{\epsilon_1+\epsilon_2}{1-\epsilon_1-\epsilon_2} \sqrt{\log\biggl(\frac{1}{\epsilon_1+\epsilon_2}\biggr)}\Biggr\},
    \end{align*}
    where the second inequality follows since $p\geq 1-\epsilon_1 - \epsilon_2$ and since $r^{(\mathrm{SG})}(\cdot)$ is decreasing.
\end{proof}

\subsubsection{Lower bound}
\begin{Proposition}\label{prop:sub-gaussian-lb}
    Let $n\in\mathbb{N}$, $(\epsilon_1,\epsilon_2) \in\mathcal{E}$ and $\alpha\in(0,\frac{1}{4}]$.
    Then  
    \begin{align*}
        L_{n,\epsilon_1,\epsilon_2}\bigl(\hat{\mathcal{CI}}(\mathcal{P}_{\mathrm{SG}}, \mathcal{E})\bigr) \geq \frac{c\sigma_{\max}}{\sqrt{n}} \vee \sigma_{\max}\Biggl\{\sqrt{\log\biggl(\frac{1}{1-\epsilon_+}\biggr)} \wedge \frac{\epsilon_+}{1-\epsilon_+}\sqrt{\log\biggl(\frac{1}{\epsilon_+}\biggr)}\Biggr\},
    \end{align*}
    where $c>0$ is a universal constant.
\end{Proposition}
\begin{proof}
    Since $\mathcal{M}(P,\epsilon_1+\epsilon_2,0) \subseteq \mathcal{M}(P,\epsilon_1,\epsilon_2)$, we have by the MCAR lower bound of \citet[][Proposition S33(a)]{ma2024estimation} that $L_{n,\epsilon_1,\epsilon_2}\bigl(\hat{\mathcal{CI}}(\mathcal{P}_{\mathrm{SG}}, \mathcal{E})\bigr) \geq c\sigma/\sqrt{n}$. The second part of the bound follows by Proposition~\ref{prop:tail-adapt} (with $\psi_1(x)=\psi_2(x)=e^{x^2}-1$) and Lemma~\ref{lemma:a1-a2}, and we therefore conclude the result.
\end{proof}

\subsection{Finite-variance random variables}
\subsubsection{Upper bound}

\begin{Lemma} \label{lemma:finite-variance-bias}
    Let $\theta\in\mathbb{R}$, $\sigma>0$ and $P \in \mathcal{P}_{\mathrm{FV}}(\theta,\sigma)$. Let $Z \sim M \in \mathsf{MNAR}_P$ and $p\coloneqq \mathbb{P}(Z \neq \star) \in [0,1]$. Then
    \begin{align*}
        \bigl|\mathbb{E}(Z \,|\, Z\neq \star) - \theta \bigr| \leq \sigma r^{(\mathrm{FV})}(p),
    \end{align*}
    and
    \begin{align*}
        \mathrm{Var}(Z\,|\,Z\neq \star) \leq \frac{\sigma^2}{p}.
    \end{align*}
\end{Lemma}
\begin{proof}
    By translation invariance, it suffices to prove the claim for $\theta=0$. As in the proof of Lemma~\ref{lemma:sub-gaussian-bias}, there exists a Borel measurable function $m:\mathbb{R}\to [0,1]$ such that $\int_{\mathbb{R}} m(x) \,\mathrm{d}P(x) = p$ and $\frac{\mathrm{d}M}{\mathrm{d}P_{\star}}(x) = m(x)$ for $x\in\mathbb{R}$, where $P_\star$ is the extended measure of $P$ on $\mathbb{R}_{\star}$. Thus, by Cauchy--Schwarz,
    \begin{align*}
        \bigl|\mathbb{E}(Z \,|\, Z\neq \star) - 0 \bigr| &= \frac{1}{p} \bigl|\mathbb{E}_P\{Xm(X)\}\bigr| = \frac{1}{p} \bigl|\mathbb{E}_P\bigl\{X\bigl(m(X)-p\bigr)\bigr\}\bigr|\\
        &\leq \frac{1}{p} \sqrt{\mathbb{E}_P(X^2) \mathrm{Var}_P \, m(X)} \leq \frac{\sigma\sqrt{p(1-p)}}{p} = \sigma\sqrt{\frac{1-p}{p}},
    \end{align*}
    and
    \begin{align*}
        \mathrm{Var}(Z \,|\, Z\neq \star) \leq \mathbb{E}(Z^2 \,|\, Z\neq \star) = \frac{1}{p} \mathbb{E}_P\{X^2m(X)\} \leq \frac{\sigma^2}{p},
    \end{align*}
    as required.
\end{proof}

Recalling $\hat{\mathrm{CI}}^{(\mathrm{FV})}_{n,\sigma_{\max},\alpha}$ from~\eqref{eq:CI-FV}, we have the following.
\begin{Proposition} \label{prop:finite-variance-CI}
    Let $(\epsilon_1,\epsilon_2) \in \mathcal{E}$, $\theta\in\mathbb{R}$, $\sigma\in(0,\sigma_{\max}]$ and $P \in \mathcal{P}_{\mathrm{FV}}(\theta,\sigma)$. Let $Z_1,\ldots,Z_n \overset{\mathrm{iid}}{\sim} M\in \mathcal{M}(P,\epsilon_1,\epsilon_2)$. Then
    \begin{align*}
        \mathbb{P}\Bigl(\theta \in \hat{\mathrm{CI}}^{(\mathrm{FV})}_{n,\sigma_{\max},\alpha} \Bigr) \geq 1-\alpha.
    \end{align*}
    Moreover, if $\epsilon_1 + \epsilon_2 \leq 1 - \frac{27\log(12/\alpha)}{n}$, then with probability at least $1-\alpha$,
    \begin{align*}
        \bigl| \hat{\mathrm{CI}}^{(\mathrm{FV})}_{n,\sigma_{\max},\alpha} \bigr| \leq C_{\alpha}\sigma_{\max} \Biggl\{ \frac{1}{\sqrt{n}} + \sqrt{\frac{\epsilon_+}{1-\epsilon_+}}\Biggr\},
    \end{align*}
    where $C_{\alpha}>0$ depends only on $\alpha$.
\end{Proposition}
\begin{proof}
    \textbf{Coverage:} Again, let $p\coloneqq \mathbb{P}(Z_1 \neq \star)$.  Recall the definition of the event $E_1$ from~\eqref{Eq:E1}, as well as the facts that $\mathbb{P}(E_1) \geq 1-\alpha/2$ and $\narrowhat{p}_{n,\alpha}^- \leq p$ on $E_1$. Define the event
    \begin{align*}
        E_2 \coloneqq \biggl\{\bigl|\bar{Z}_n - \mathbb{E}(Z\,|\,Z\neq\star)\bigr| \leq \sigma_{\max} \sqrt{\frac{2}{n\narrowhat{p}_np\alpha}}\biggr\}.
    \end{align*}
    Let $Y_1,\ldots,Y_n \overset{\mathrm{iid}}{\sim} \mathsf{Law}(Z_1 \,|\, Z_1 \neq\star)$. By Lemma~\ref{lemma:finite-variance-bias} and Chebyshev's inequality, we deduce that for $k \in [n]$,
    \begin{align*}
        \mathbb{P}(E_2\,|\,n\narrowhat{p}_n = k) = \mathbb{P}\biggl( \biggl|\frac{1}{k}\sum_{i=1}^k Y_i - \mathbb{E}(Y_1)\biggr| \leq \sigma_{\max} \sqrt{\frac{2}{kp\alpha}} \biggr) \geq 1-\frac{\alpha}{2}.
    \end{align*}
    Moreover, when $k=0$, our convention that $\bar{Z}_n=0$ means that this bound continues to hold. Hence, $\mathbb{P}(E_2) \geq 1-\alpha/2$
    Therefore, by Lemma~\ref{lemma:finite-variance-bias} and since $r^{(\mathrm{FV})}(\cdot)$ is decreasing, we deduce that on $E_1 \cap E_2$, we have $\theta\in \hat{\mathrm{CI}}^{(\mathrm{FV})}_{n,\sigma_{\max}, \alpha}$.

    \medskip
    \textbf{Length:} By Bernstein's inequality, the event $E_3 \coloneqq \bigl\{|\narrowhat{p}_n - p| \leq \sqrt{\frac{2p(1-p)\log(6/\alpha)}{n}} + \frac{\log(6/\alpha)}{3n}\bigr\}$ has probability at least $1-\alpha/3$. Thus, on $E_3$, we have
    \begin{align}
        p - \narrowhat{p}_{n,\alpha}^- &\leq p - \narrowhat{p}_n + \sqrt{\frac{2\narrowhat{p}_n(1-\narrowhat{p}_n)\log(6/\alpha)}{n}} + \frac{3\log(6/\alpha)}{n}\nonumber\\
        &\lesssim_{\alpha} p - \narrowhat{p}_n + \sqrt{\frac{(1-p) + (p-\narrowhat{p}_n)}{n}} + \frac{1}{n}\nonumber\\
        &\lesssim |\narrowhat{p}_n - p| + \sqrt{\frac{1-p}{n}} + \frac{1}{n} \lesssim_{\alpha} \sqrt{\frac{1-p}{n}} + \frac{1}{n}, \label{eq:p^-lb1}
    \end{align}
    where the penultimate inequality uses the AM--GM inequality to deduce that $\sqrt{\frac{|\narrowhat{p}_n-p|}{n}} \lesssim |\narrowhat{p}_n-p| + \frac{1}{n}$. Moreover, since $\epsilon_1+\epsilon_2 \leq 1 - \frac{27\log(12/\alpha)}{n}$, we have $p\geq \frac{27\log(12/\alpha)}{n}$. Thus, by a multiplicative Chernoff bound, the event $E_4 \coloneqq \bigl\{\frac{2p}{3} \leq \narrowhat{p}_n \leq \frac{4p}{3} \bigr\}$ has probability at least $1-\alpha/6$. On the event $E_4$, we have
    \begin{align}
        \narrowhat{p}_{n,\alpha}^- &\geq \frac{2p}{3} - \sqrt{\frac{2\narrowhat{p}_n(1-\narrowhat{p}_n)\log(6/\alpha)}{n}} - \frac{3\log(6/\alpha)}{n}\nonumber\\
        &\geq \frac{2p}{3} - \sqrt{\frac{3p\log(6/\alpha)}{n}} - \frac{3\log(6/\alpha)}{n} \geq \frac{2p}{3} - \frac{p}{3} - \frac{p}{9} = \frac{2p}{9}. \label{eq:p^-lb2}
    \end{align}
    Therefore, we deduce that on $E_3\cap E_4$,
    \begin{align*}
        r^{(\mathrm{FV})}(\narrowhat{p}_{n,\alpha}^-) = \sqrt{\frac{1-\narrowhat{p}_{n,\alpha}^-}{\narrowhat{p}_{n,\alpha}^-}} &\overset{(i)}{\leq} \sqrt{\frac{9(1-\narrowhat{p}_{n,\alpha}^-)}{2p}} \lesssim \sqrt{\frac{1-p}{p}} + \sqrt{\frac{p-\narrowhat{p}^-_{n,\alpha}}{p}}\\
        &\overset{(ii)}{\lesssim}_{\hspace{-0.06cm}\alpha} \sqrt{\frac{1-p}{p}} + \biggl\{\sqrt{\frac{1-p}{p}}\sqrt{\frac{1}{np}}\biggr\}^{1/2} + \frac{1}{\sqrt{np}} \lesssim \sqrt{\frac{1-p}{p}} + \frac{1}{\sqrt{np}},
    \end{align*}
    where $(i)$ follows from~\eqref{eq:p^-lb2} and $(ii)$ follows from~\eqref{eq:p^-lb1}; moreover, on $E_4$,
    \begin{align*}
        \sqrt{\frac{2}{n\narrowhat{p}_n\narrowhat{p}^-_{n,\alpha}\alpha}} \lesssim_{\alpha} \frac{1}{\sqrt{np^2}}.
    \end{align*}
    Thus, with probability at least $1-\alpha/2$, since $p\geq 1-\epsilon_1-\epsilon_2$,
    \begin{align*}
        \bigl| \hat{\mathrm{CI}}^{(\mathrm{FV})}_{n,\sigma_{\max},\alpha} \bigr| &\lesssim_{\alpha} \sigma_{\max} \Biggl\{ \frac{1}{\sqrt{n(1-\epsilon_1-\epsilon_2)^2}} + \sqrt{\frac{\epsilon_1+\epsilon_2}{1-\epsilon_1-\epsilon_2}}\Biggr\} \lesssim_{\alpha} \sigma_{\max} \Biggl\{ \frac{1}{\sqrt{n}} + \sqrt{\frac{\epsilon_1+\epsilon_2}{1-\epsilon_1-\epsilon_2}}\Biggr\},
    \end{align*}
    where the last inequality follows since $1-\epsilon_1-\epsilon_2 \geq \frac{27\log(12/\alpha)}{n}$.
\end{proof}

Similarly to~\eqref{eq:BD-CI-asym-length}, for fixed $\epsilon_+<1$, we have
\begin{align*}
    \lim_{n \rightarrow \infty} \sup_{P\in\mathcal{P}_{\mathrm{FV}}} \sup_{M\in\mathcal{M}(P,\epsilon_1,\epsilon_2)} \mathbb{P}_M\biggl(|\hat{\mathrm{CI}}_{n,\sigma_{\max},\alpha}^{(\mathrm{FV})}| > \sigma_{\max}\sqrt{\frac{\epsilon_+}{1-\epsilon_+}} \biggr) = 0.
\end{align*}

\subsubsection{Lower bound}
To prove the lower bound for Theorem~\ref{thm:nonparametric-minimax-rates}\emph{(c)}, it suffices to prove Proposition~\ref{thm:finite-variance-lb}; see Appendix~\ref{sec:proof-of-tail-adaptation} below.

\section{Proof of Proposition~\ref{thm:finite-variance-lb}} \label{sec:proof-of-tail-adaptation}

Let $\Psi$ denote the class of convex, lower semi-continuous functions $\psi:[0,\infty) \to [0,\infty)$ with $\psi(0)=0$ and $\psi(x) \to \infty$ as $x\to\infty$. The $\psi$-Orlicz norm of a random variable $X$ is defined as $\|X\|_{\psi}\coloneqq \inf\bigl\{t>0 : \mathbb{E}\psi(|X|/t) \leq 1\bigr\}$. For example, if $\mathbb{E}(X)=0$ and $\psi(x) = x^2$, then $\|X\|_{\psi}$ is the variance of $X$; if $\mathbb{E}(X)=0$ and $\psi(x) = e^{x^2}-1$, then $\|X\|_{\psi}$ is the sub-Gaussian norm of $X$.

Given $\psi_1,\psi_2 \in \Psi$, the following lemma considers adaptive confidence intervals that are valid for all distributions with either $\|X-\mathbb{E}(X)\|_{\psi_1} \leq \sigma_1$ or $\|X-\mathbb{E}(X)\|_{\psi_2} \leq \sigma_2$, and provides a lower bound for the width of the confidence interval restricted to distributions with $\|X-\mathbb{E}(X)\|_{\psi_2} \leq \sigma$, for some $\sigma\leq\sigma_2$ (and similarly for $\psi_1$).

\begin{Proposition} \label{prop:tail-adapt}
    Let $\sigma_1,\sigma_2>0$. For $\theta\in\mathbb{R}$, $\sigma>0$ and $j \in \{1,2\}$, let $\psi_j \in \Psi$ and let $\mathcal{P}_j(\theta,\sigma)$ be the set of all distributions~$P$ on $\mathbb{R}$ such that if $X\sim P$, then $\mathbb{E}(X)=\theta$ and $\|X-\theta\|_{\psi_j} \leq \sigma$. Fix $\sigma_1,\sigma_2 > 0$, and let $\hat{\mathcal{CI}}$ be the set of all confidence intervals $\hat{\mathrm{CI}}(Z_1,\ldots,Z_n)$ satisfying $\mathbb{P}_M\bigl(\theta\in \hat{\mathrm{CI}}(Z_1,\ldots,Z_n)\bigr) \geq 1-\alpha$ for all $\theta\in\mathbb{R}$, $\epsilon_1,\epsilon_2 \geq 0$ with $\epsilon_1+\epsilon_2 \leq 1$, and $Z_1,\ldots,Z_n \overset{\mathrm{iid}}{\sim} M\in\mathcal{M}(P,\epsilon_1,\epsilon_2)$ where $P\in \mathcal{P}_1(\theta,\sigma_1) \cup \mathcal{P}_2(\theta,\sigma_2)$. For $j\in\{1,2\}$ and $\sigma>0$, let
    \begin{align*}
        L_j(n,\sigma,\epsilon_1,\epsilon_2) \coloneqq \inf \biggl\{r>0 : \inf_{\hat{\mathrm{CI}} \in \hat{\mathcal{CI}}}\; \sup_{\theta\in\mathbb{R},\, P\in\mathcal{P}_j(\theta,\sigma),\, M\in\mathcal{M}(P,\epsilon_1,\epsilon_2)} \mathbb{P}_M\bigl(|\hat{\mathrm{CI}}(Z_1,\ldots,Z_n)| \geq r\bigr) \leq \alpha \biggr\}
    \end{align*}
    and let
    \begin{align*}
        a_j(\sigma_j,\epsilon_1,\epsilon_2) \coloneqq \sup \biggl\{a\geq 0 :  (1-\epsilon_1 - \epsilon_2)\psi_j\biggl(\frac{a}{\sigma_j}\biggr) + (\epsilon_1+\epsilon_2)\psi_j\biggl(\frac{1-\epsilon_1 - \epsilon_2}{\epsilon_1+\epsilon_2} \cdot \frac{a}{\sigma_j}\biggr) \leq 1\biggr\},
    \end{align*}
    where we adopt the convention that $0/0 \coloneqq 0$.
    If $\alpha\in(0,\frac{1}{4}]$, then for $j\in\{1,2\}$ and $\sigma\in(0,\sigma_j]$, we have
    \begin{align*}
        L_j(n,\sigma,\epsilon_1,\epsilon_2) \geq a_1(\sigma_1,\epsilon_1,\epsilon_2) \vee a_2(\sigma_2,\epsilon_1,\epsilon_2).
    \end{align*}
\end{Proposition}
\begin{proof}
    We will prove the lower bound for $L_2(n,\sigma,\epsilon_1,\epsilon_2)$; the corresponding result for $L_1(n,\sigma,\epsilon_1,\epsilon_2)$ follows by symmetry. For $a>0$, $P_1\in \mathcal{P}_1(0,\sigma_1) \cup \mathcal{P}_2(0,\sigma_2)$, $P_2\in\mathcal{P}_2(a,\sigma_2)$, $M_1\in \mathcal{M}\bigl(P_1,0,1\bigr)$ and $M_2 \in \mathcal{M}\bigl(P_2, \epsilon_1, \epsilon_2\bigr)$ to be chosen later,
    consider the testing problem $H_0: Z_1,\ldots,Z_n \overset{\mathrm{iid}}{\sim} M_1$ against $H_1: Z_1,\ldots,Z_n \overset{\mathrm{iid}}{\sim} M_2$. Let $\xi>0$ and let $\hat{\mathrm{CI}} \in \hat{\mathcal{CI}}$ be such that
    \begin{align*}
        \sup_{\theta\in\mathbb{R},\, P\in\mathcal{P}_2(\theta,\sigma),\, M\in\mathcal{M}(P,\epsilon_1,\epsilon_2)} \mathbb{P}_M\bigl(|\hat{\mathrm{CI}}(Z_1,\ldots,Z_n)| \geq (1+\xi)L_2 (n,\sigma,\epsilon_1,\epsilon_2)\bigr) \leq \alpha',
    \end{align*}
    where $\alpha'\coloneqq 9/10 - 2\alpha > \alpha$.  Define $T\coloneqq \mathbbm{1}_{\{0\notin \hat{\mathrm{CI}}(Z_1,\ldots,Z_n)\}}$.
    If $(1+\xi)L_2(n,\sigma,\epsilon_1,\epsilon_2) \leq a$, then
    \begin{align}
        \mathbb{P}_{M_1}(T=1) &+ \mathbb{P}_{M_2}(T=0)\nonumber\\
        &\leq \alpha + \mathbb{P}_{M_2}\bigl(a\notin\hat{\mathrm{CI}}(Z_1,\ldots,Z_n)\bigr) + \mathbb{P}_{M_2}\bigl(|\hat{\mathrm{CI}}(Z_1,\ldots,Z_n)| \geq (1+\xi)L_2(n,\sigma,\epsilon_1,\epsilon_2)\bigr) \nonumber\\
        &\leq 2\alpha+\alpha' = \frac{9}{10}. \label{eq:testing-lb-unkown-epsilon-psi}
    \end{align}
    We now make specific choices of $a$, $P_1$, $P_2$ $M_1$ and $M_2$ to prove that~\eqref{eq:testing-lb-unkown-epsilon-psi} cannot hold.  This will establish a contradiction, allowing us to conclude, since $\xi>0$ is arbitrary, that $L_2(n,\sigma,\epsilon_1,\epsilon_2) \geq a$.

    Assume without loss of generality that $a_1(\sigma_1,\epsilon_1,\epsilon_2) \vee a_2(\sigma_2,\epsilon_1,\epsilon_2)>0$, since otherwise the claim is trivial.
    Let $\eta>0$ be small enough that $a\coloneqq a_1(\sigma_1,\epsilon_1,\epsilon_2) \vee a_2(\sigma_2,\epsilon_1,\epsilon_2)-\eta > 0$ and let $b\coloneqq -\frac{1-\epsilon_1-\epsilon_2}{\epsilon_1+\epsilon_2} \cdot a$. We define a distribution~$P_1$ on $\mathbb{R}$ by
    \begin{align*}
        P_1(\{a\}) \coloneqq 1-\epsilon_1-\epsilon_2 \quad\text{and}\quad P_1(\{b\}) \coloneqq \epsilon_1+\epsilon_2,
    \end{align*}
    and define another distribution $P_2$ on $\mathbb{R}$ by
    \begin{align*}
        P_2(\{a\}) \coloneqq 1.
    \end{align*}
    If $X\sim P_1$, then $\mathbb{E}(X)=0$ and
    \[
    \mathbb{E} \psi_j\biggl(\frac{|X|}{\sigma_j}\biggr) = (1-\epsilon_1 - \epsilon_2)\psi_j\biggl(\frac{a}{\sigma_j}\biggr) + (\epsilon_1+\epsilon_2)\psi_j\biggl(\frac{1-\epsilon_1 - \epsilon_2}{\epsilon_1+\epsilon_2} \cdot \frac{a}{\sigma_j}\biggr),
    \]
    so by definition of $a_j(\sigma_j,\epsilon_1,\epsilon_2)$, either $\mathbb{E} \psi_1(|X|/\sigma_1) \leq 1$ or $\mathbb{E} \psi_2(|X|/\sigma_2) \leq 1$.  Hence $P_1\in \mathcal{P}_1(0,\sigma_1) \cup \mathcal{P}_2(0,\sigma_2)$. Moreover, if $Y\sim P_2$, then $\mathbb{E}(Y)=a$ and for $\sigma \in (0,\sigma_2]$, we have $\mathbb{E} \psi_2(|Y-a|/\sigma) = \psi_2(0) = 0$, so $P_2 \in \mathcal{P}_2(a,\sigma)$. We further define a distribution $M$ on $\mathbb{R}_{\star}$ by
    \begin{align*}
        M(\{a\}) \coloneqq 1-\epsilon_1-\epsilon_2 \quad\text{and}\quad M(\{\star\}) \coloneqq \epsilon_1 + \epsilon_2.
    \end{align*}
    Then $\frac{M(\{a\})}{P_1(\{a\})}, \frac{M(\{b\})}{P_1(\{b\})}, \frac{M(\{a\})}{P_2(\{a\})} \in [0, 1]$ so by Lemma~\ref{lemma:realisability-characterisation}, $M\in\mathcal{M}(P_1,0,1)$ and $M\in\mathcal{M}(P_2,\epsilon_1,\epsilon_2)$. Further, letting $M_1\coloneqq M$ and $M_2\coloneqq M$, we have
    \begin{align*}
        \mathbb{P}_{M_1}(T=1) + \mathbb{P}_{M_2}(T=0) = 1 > \frac{9}{10},
    \end{align*}
    but this contradicts~\eqref{eq:testing-lb-unkown-epsilon-psi}. Hence, $L_2(n,\sigma,\epsilon_1,\epsilon_2) \geq a = a_1(\sigma_1,\epsilon_1,\epsilon_2) \vee a_2(\sigma_2,\epsilon_1,\epsilon_2)-\eta$. Since $\eta>0$ was arbitrary, the claim follows.
\end{proof}

\begin{Lemma} \label{lemma:a1-a2}
    Define $\psi_1,\psi_2 \in \Psi$ by $\psi_1(x) \coloneqq x^2$ and $\psi_2(x) \coloneqq e^{x^2}-1$.  Then
    \begin{align*}
        a_1(\sigma_1,\epsilon_1,\epsilon_2) = \sigma_1\sqrt{\frac{\epsilon_1+\epsilon_2}{1-\epsilon_1 - \epsilon_2}},
    \end{align*}
    and
    \begin{align*}
        a_2(\sigma_2,\epsilon_1,\epsilon_2) \geq \sigma_2\Biggl\{\sqrt{\log\biggl(\frac{1}{1-\epsilon_1-\epsilon_2}\biggr)} \wedge \frac{\epsilon_1+\epsilon_2}{1-\epsilon_1-\epsilon_2}\sqrt{\log\biggl(\frac{1}{\epsilon_1+\epsilon_2}\biggr)}\Biggr\}.
    \end{align*}
\end{Lemma}
\begin{proof}
    The first claim follows because
    \begin{align*}
    (1-\epsilon_1 - \epsilon_2)\psi_1\biggl(\frac{a}{\sigma_1}\biggr) + (\epsilon_1+\epsilon_2)\psi_1\biggl(\frac{1-\epsilon_1 - \epsilon_2}{\epsilon_1+\epsilon_2} \cdot \frac{a}{\sigma_1}\biggr) &= \frac{a^2}{\sigma_1^2}\biggl\{1-\epsilon_1 - \epsilon_2 + \frac{(1-\epsilon_1-\epsilon_2)^2}{\epsilon_1+\epsilon_2}\biggr\} \\
    &= \frac{a^2}{\sigma_1^2} \cdot \frac{1-\epsilon_1 - \epsilon_2}{\epsilon_1+\epsilon_2}.
    \end{align*}
    
    For the second claim, write $a\coloneqq \sigma_2 \Bigl\{ \log^{1/2}\bigl(\frac{1}{1-\epsilon_1-\epsilon_2}\bigr) \wedge \frac{\epsilon_1+\epsilon_2}{1-\epsilon_1-\epsilon_2}\log^{1/2}\bigl(\frac{1}{\epsilon_1+\epsilon_2}\bigr) \Bigr\}$. Then, since $\psi_2$ is an increasing function on $[0,\infty)$, we have
    \begin{align*}
        &(1-\epsilon_1 - \epsilon_2)\psi_2\biggl(\frac{a}{\sigma_2}\biggr) + (\epsilon_1+\epsilon_2)\psi_2\biggl(\frac{1-\epsilon_1 - \epsilon_2}{\epsilon_1+\epsilon_2} \cdot \frac{a}{\sigma_2}\biggr)\\
        &\qquad\leq (1-\epsilon_1 - \epsilon_2)\psi_2\biggl(\log^{1/2}\Bigl(\frac{1}{1-\epsilon_1-\epsilon_2}\Bigr)\biggr) + (\epsilon_1+\epsilon_2)\psi_2\biggl(\log^{1/2}\Bigl(\frac{1}{\epsilon_1+\epsilon_2}\Bigr)\biggr) = 1.
    \end{align*}
    The claim follows.
\end{proof}

Now, we are ready to prove Proposition~\ref{thm:finite-variance-lb}.

\begin{proof}[Proof of Proposition~\ref{thm:finite-variance-lb}]
    Since $\mathcal{M}(P,\epsilon_1+\epsilon_2,0) \subseteq \mathcal{M}(P,\epsilon_1,\epsilon_2)$, we have by the MCAR lower bound of \citet[][Proposition S33(a)]{ma2024estimation} that $L(n,\sigma,\epsilon_1,\epsilon_2) \geq c\sigma/\sqrt{n}$. The second part of the bound follows by Proposition~\ref{prop:tail-adapt} (with $\psi_1(x)=x^2$ and $\psi_2(x)=e^{x^2}-1$) and Lemma~\ref{lemma:a1-a2}, and we therefore conclude the result.
\end{proof}

\section{Proof of Theorem~\ref{thm:symmetric}}
\begin{proof}
    \textbf{Coverage:} We may write $M=(1-\epsilon_1-\epsilon_2)P + (\epsilon_1+\epsilon_2)Q$ where $Q\in\mathsf{MNAR}_{P}$. For $i\in[n]$, we may assume that $Z_1,\ldots,Z_n$ are generated via $Z_i \coloneqq L_i X_i + (1-L_i)(X_i \ostar \Omega_i)$ where $(X_1,L_1,\Omega_1),\ldots,(X_n,L_n,\Omega_n)$ are independent triples with $X_i \sim P$, $\Omega_i$ a binary random variable satisfying $X_i \ostar \Omega_i \sim Q$ and $L_i \sim \mathrm{Bern}(1-\epsilon_1-\epsilon_2)$ independent of $(X_i,\Omega_i)$. In other words, $Z_i \, | \, \{L_i=1\} = X_i \sim P$ and $Z_i \, | \, \{L_i = 0\} = X_i \ostar \Omega_i \sim Q$. Let $\mathcal{I} \coloneqq \{i\in[n] : L_i = 1\}$, so that $|\mathcal{I}| \sim \mathrm{Bin}(n, 1-\epsilon_1-\epsilon_2)$. By the multiplicative Chernoff bound, and since $1-\epsilon_1-\epsilon_2 \geq \frac{4\log_2(8/\alpha)}{n}$, we deduce that
    \begin{align*}
        \mathbb{P}\bigl(|\mathcal{I}| \leq \log_2(8/\alpha)\bigr) \leq \exp\bigl(-9\log_2(8/\alpha)/8\bigr) \leq \frac{\alpha}{8}.
    \end{align*}
    Since $X_i \sim P$ is symmetric around $\theta$, we have
    \begin{align*}
        \mathbb{P}(Z_{\max} \leq \theta) &\leq \mathbb{P}\Bigl(\Bigl\{\max_{i\in\mathcal{I}}X_i \leq \theta\Bigr\} \bigcap \Bigl\{|\mathcal{I}| > \log_2(8/\alpha)\Bigr\}\Bigr) + \mathbb{P}\Bigl(|\mathcal{I}| \leq \log_2(8/\alpha)\Bigr) \nonumber\\
        &\leq 2^{-\log_2(8/\alpha)} + \frac{\alpha}{8} \leq \frac{\alpha}{4}. 
    \end{align*}
    Thus, by Lemma \ref{lemma:empirical-quantile-property},     
    \begin{align*}
        \mathbb{P} \biggl\{ Q_n^-\biggl(\frac{1}{2}+\sqrt{\frac{\log(4/\alpha)}{2n}}\biggr) \leq \theta \biggr\} &\leq \mathbb{P} \biggl\{ \frac{1}{n} \sum_{i=1}^n \indi_{\{Z_i \leq \theta\}} \geq \frac{1}{2}+\sqrt{\frac{\log(4/\alpha)}{2n}} \biggr\} + \mathbb{P}(Z_{\max} \leq \theta)\\
        &\leq \frac{\alpha}{4} + \frac{\alpha}{4} = \frac{\alpha}{2},
    \end{align*}
    where the second inequality follows from Hoeffding's inequality since $\mathbb{P}(Z_i \leq \theta) \leq 1/2$. Similarly, 
    \begin{align*}
        \mathbb{P} \biggl\{ Q_n^+\biggl(\frac{1}{2}-\sqrt{\frac{\log(4/\alpha)}{2n}}\biggr) > \theta \biggr\} \leq \frac{\alpha}{2}.
    \end{align*}
    Thus, by a union bound, we deduce that $\mathbb{P}\bigl(\theta \in \hat{\mathrm{CI}}^{(\mathrm{Sym})}_{n,\alpha}\bigr) \geq 1-\alpha$.

    \medskip
    \textbf{Length:} Since $P$ admits a Lebesgue density $f$ that is positive and continuous at~$\theta$, there exists $c_1>0$ such that $f(x) \geq f(\theta)/2$ for all $x\in[\theta-c_1,\theta+c_1]$. Let $C_1\coloneqq 2/f(\theta)>0$ and let $\theta'\coloneqq \theta + C_1\bigl(2\sqrt{\frac{\log(4/\alpha)}{2n}} + \epsilon_1 + \epsilon_2\bigr)$. Then, by Lemma~\ref{lemma:empirical-quantile-property},
    \begin{align}
        \mathbb{P}\Bigl((\theta',\infty) \cap \hat{\mathrm{CI}}^{(\mathrm{Sym})}_{n,\alpha} = \emptyset\Bigr) &\geq \mathbb{P}\bigg\{ \theta' \geq Q_n^-\biggl(\frac{1}{2}+\sqrt{\frac{\log(4/\alpha)}{2n}}\biggr) \biggr\} \nonumber\\
        &\geq \mathbb{P}\bigg\{ \frac{1}{n} \sum_{i=1}^n \indi_{\{Z_i \leq \theta'\}} \geq \frac{1}{2}+\sqrt{\frac{\log(4/\alpha)}{2n}} \biggr\}. \label{eq:theta-notin-CI-sym}
    \end{align}
    Now, by Hoeffding's inequality, the event 
    \begin{align*}
        \mathcal{E} \coloneqq \Biggl\{\frac{1}{n} \sum_{i=1}^n \indi_{\{Z_i \leq \theta'\}} \geq \mathbb{P} ( Z_1 \leq \theta') - \sqrt{\frac{\log(4/\alpha)}{2n}} \Biggr\}
    \end{align*}
    has probability at least $1-\alpha/4$. Assume further that $n$ is large enough and $\epsilon_1+\epsilon_2$ is small enough (both depending only on $f$ and $\alpha$) that $C_1\bigl(2\sqrt{\frac{\log(4/\alpha)}{2n}} + \epsilon_1 + \epsilon_2\bigr) \leq c_1 \wedge C_1/2$. Then
    \begin{align}
        \mathbb{P} ( Z_1 \leq \theta') &\geq (1-\epsilon_1 - \epsilon_2) \mathbb{P} ( X_1 \leq \theta')\nonumber\\
        &= (1-\epsilon_1 - \epsilon_2) \biggl( \frac{1}{2} + \int_\theta^{\theta'} f(x) \,\mathrm{d}x \biggr)\nonumber\\
        &\geq (1-\epsilon_1 - \epsilon_2) \biggl\{ \frac{1}{2} + \frac{f(\theta)}{2} \cdot C_1\biggl(2\sqrt{\frac{\log(4/\alpha)}{2n}} + \epsilon_1 + \epsilon_2 \biggr) \biggr\}\nonumber\\
        &= (1-\epsilon_1 - \epsilon_2) \biggl( \frac{1}{2} + 2\sqrt{\frac{\log(4/\alpha)}{2n}} + \epsilon_1 + \epsilon_2 \biggr)\nonumber\\
        &= \frac{1}{2} + 2\sqrt{\frac{\log(4/\alpha)}{2n}} + (\epsilon_1+\epsilon_2)\biggl(\frac{1}{2} - 2\sqrt{\frac{\log(4/\alpha)}{2n}} - \epsilon_1 - \epsilon_2 \biggr) \geq \frac{1}{2} + 2\sqrt{\frac{\log(4/\alpha)}{2n}}. \label{eq:Z1-theta'}
    \end{align}
    Thus, by \eqref{eq:theta-notin-CI-sym} and ~\eqref{eq:Z1-theta'},
    \begin{align*}
        \mathbb{P}\Bigl((\theta',\infty) \cap \hat{\mathrm{CI}}^{(\mathrm{Sym})}_{n,\alpha} = \emptyset\Bigr) \geq \mathbb{P}(\mathcal{E}) \geq 1-\frac{\alpha}{4}.
    \end{align*}
    Similarly, by symmetry, letting $\theta''\coloneqq \theta -C_1 \bigl(2\sqrt{\frac{\log(4/\alpha)}{2n}} + \epsilon_1 + \epsilon_2\bigr)$, we have $\mathbb{P}\bigl((-\infty,\theta'') \cap \hat{\mathrm{CI}}^{(\mathrm{Sym})}_{n,\alpha} = \emptyset\bigr) \geq 1-\alpha/4$.  We deduce that with probability at least $1-\alpha/2$, we have $|\hat{\mathrm{CI}}^{(\mathrm{Sym})}_{n,\alpha}| \leq 2C_1 \bigl(2\sqrt{\frac{\log(4/\alpha)}{2n}} + \epsilon_1 + \epsilon_2\bigr)$.
\end{proof}

\section{Results from Section~\ref{sec:causal}}\label{sec:proofs-causal}

\begin{Proposition}
    \label{prop:ATE-upper-bound}
     Let $\mathcal{P} \subseteq \mathcal{P}(\mathbb{R}^2)$ and $\mathcal{I} \subseteq [0, 1]^2$, and consider the collection of distributions
    \[
    \mathcal{M} \coloneqq \bigcup_{P \in \mathcal{P},\,  (q,\eta) \in \mathcal{I}}\; \mathcal{M}_2(P,q, \eta).
    \]  
    Let $Z_1,\ldots,Z_n \overset{\mathrm{iid}}{\sim} M \in \mathcal{M}$ and let $Z_i(0)$ and $Z_i(1)$ denote the coordinates of $Z_i$ for $i\in[n]$.
    If $\widehat{\mathrm{CI}}_0\bigl(Z_1(0), \ldots, Z_n(0)\bigr)$ and $\widehat{\mathrm{CI}}_1\bigl(Z_1(1), \ldots, Z_n(1)\bigr)$ satisfy the coverage guarantees
        \[
        \inf_{M \in \mathcal{M}}\; \mathbb{P}\Bigl\{\theta_0 \in \widehat{\mathrm{CI}}_0\bigl(Z_1(0), \ldots, Z_n(0)\bigr)\Bigr\} \geq 1 - \frac{\alpha}{2} \; \text{ and } \;\inf_{M \in \mathcal{M}}\; \mathbb{P}\Bigl\{\theta_1 \in \widehat{\mathrm{CI}}_1\bigl(Z_1(1), \ldots, Z_n(1)\bigr)\Bigr\}
        \geq 1 - \frac{\alpha}{2},
        \]
        then the interval
        \begin{align*}
            \widehat{\mathrm{CI}}'(Z_1,\ldots,Z_n) \coloneqq \Bigl\{b-a : a\in\widehat{\mathrm{CI}}_0\bigl(Z_1(0),\ldots,Z_n(0)\bigr),\, b\in \widehat{\mathrm{CI}}_1\bigl(Z_1(1),\ldots,Z_n(1)\bigr)\Bigr\},
        \end{align*}
        satisfies the coverage guarantee
        \[
        \inf_{M \in \mathcal{M}}\mathbb{P}\Bigl\{\mathrm{ATE}_P \in \widehat{\mathrm{CI}}'(Z_1, \ldots, Z_n)\Bigr\} \geq 1 - \alpha,
        \]
        and the length guarantee $\lvert \widehat{\mathrm{CI}}' \rvert \leq \lvert \widehat{\mathrm{CI}}_0 \rvert + \lvert \widehat{\mathrm{CI}}_1 \rvert$.
\end{Proposition}

\begin{proof}[Proof of Lemma~\ref{lem:bivariate-univariate-reduction}]
    Let $Z(0) \sim M_0'$.  Then, with $\bigl(X(0),X(1)\bigr) \sim P$, we may assume that $Z(0)$ is generated as
    \begin{align*}
    Z(0) = X(0) \cdot \bigl(1 - L^{(1)}\bigr)\bigl(1-L^{(2)}\bigr) + L^{(1)}\bigl(1 - L^{(2)}\bigr)\cdot \star + \bigl\{X(0) \ostar \Omega^{\mathrm{MNAR}}(0)\bigr\}\cdot L^{(2)},
    \end{align*}
    where $L^{(1)} \sim \mathrm{Bern}(q)$ and $L^{(2)} \sim \mathrm{Bern}(\eta)$ are independent and independent of $\bigl(X(0),\Omega^{\mathrm{MNAR}}(0)\bigr)$ and $\Omega^{\mathrm{MNAR}}(0)$ is binary.  Now define 
    \[
    Z(1) \coloneqq X(1) \cdot L^{(1)}\bigl(1-L^{(2)}\bigr) + \bigl(1 - L^{(1)}\bigr)\bigl(1 - L^{(2)}\bigr)\cdot \star + \bigl\{X(1) \ostar \Omega^{\mathrm{MNAR}}(1)\bigr\}\cdot L^{(2)},
    \]
    where $\Omega^{\mathrm{MNAR}}(1) \coloneqq 1-\Omega^{\mathrm{MNAR}}(0)$.
    Then by construction the joint distribution of $\bigl(Z(0), Z(1)\bigr)$ belongs to $\mathcal{M}_2(P,\eta,q)$.  The argument when $Z(1) \sim M_1'$ is almost identical.
\end{proof}

\begin{proof}[Proof of Lemma~\ref{Lemma:MaringalSensitivity}]
    Suppose that $M\in\mathcal{M}_2(P,q,\eta)$. Then
    \begin{align*}
        M = (1-\eta) \mathrm{Law}\biggl(\begin{pmatrix}
            Y(0) \\ Y(1)
        \end{pmatrix} \ostar 
        \begin{pmatrix}
            \Omega_{\mathrm{MCAR}}(0) \\ \Omega_{\mathrm{MCAR}}(1)
        \end{pmatrix}\biggr) +
        \eta \mathrm{Law}\biggl(\begin{pmatrix}
            Y(0) \\ Y(1)
        \end{pmatrix} \ostar 
        \begin{pmatrix}
            \Omega_{\mathrm{MNAR}}(0) \\ \Omega_{\mathrm{MNAR}}(1)
        \end{pmatrix}\biggr),
    \end{align*}
    where $\Omega_{\mathrm{MCAR}}(0) \sim \mathrm{Bern}(q)$ and $\Omega_{\mathrm{MCAR}}(1)\coloneqq 1-\Omega_{\mathrm{MCAR}}(0)$ are independent of $\bigl(Y(0),Y(1)\bigr)$, $\Omega_{\mathrm{MNAR}}(0)$ is binary and $\Omega_{\mathrm{MNAR}}(1)\coloneqq 1-\Omega_{\mathrm{MNAR}}(0)$. Let $Z=\bigl(Z(0),Z(1)\bigr)\sim M$ and $A\coloneqq \indi_{\{Z(1)\neq\star\}}$. Then 
    for $j\in\{0,1\}$, we have
    \begin{align}
        (1-\eta)(1-q) \leq \mathbb{P}\bigl(A=1 \,|\, Y(j)\bigr) \leq (1-\eta)(1-q) + \eta  \label{eq:prob-bound-1}
    \end{align}
    and
    \begin{align}
        (1-\eta)q \leq \mathbb{P}\bigl(A=0 \,|\, Y(j)\bigr) \leq (1-\eta)q + \eta  \label{eq:prob-bound-2}
    \end{align}
    almost surely, so these upper and lower bounds also hold unconditionally.  Thus, by Bayes rule, for $j\in\{0,1\}$, 
    \begin{align*}
        \frac{\mathrm{d}P\bigl(Y(j) \mid A = 1\bigr)}{\mathrm{d}P\bigl(Y(j) \mid A = 0\bigr)} = \frac{\mathbb{P}\bigl(A=0\bigr)}{\mathbb{P}\bigl(A=1\bigr)} \cdot \frac{P\bigl(A = 1 \,|\, Y(j)\bigr)}{P\bigl(A = 0 \,|\, Y(j)\bigr)}.
    \end{align*}
    The claim follows by applying~\eqref{eq:prob-bound-1},~\eqref{eq:prob-bound-2} and their unconditional versions.
\end{proof}

\section{Auxiliary lemmas}

\begin{Lemma} \label{lemma:empirical-quantile-property}
    Let $a\in\mathbb{R}$ and $p\in(0,1)$. Then 
\begin{align*}
    \mathbb{P}\biggl(\frac{1}{n}\sum_{i=1}^n \indi_{\{Z_i\leq a\}} \geq p\biggr) \leq \mathbb{P}(Q_n^-(p)\leq a) \leq \mathbb{P}\biggl(\frac{1}{n}\sum_{i=1}^n \indi_{\{Z_i\leq a\}} \geq p\biggr) + \mathbb{P}(Z_{\max} \leq a), 
\end{align*}
and
\begin{align*}
    \mathbb{P}\biggl(\frac{1}{n} \sum_{i=1}^n \indi_{\{Z_i > a\}} > p\biggr) \leq \mathbb{P}(Q_n^+(p)>a) \leq \mathbb{P}\biggl(\frac{1}{n} \sum_{i=1}^n \indi_{\{Z_i > a\}} > p\biggr) + \mathbb{P}(Z_{\min}>a). 
\end{align*}
\end{Lemma}
\begin{proof}
The function $x \mapsto n^{-1}\sum_{i=1}^n \mathbbm{1}_{\{Z_i \leq x\}}$ is both right-continuous and increasing, so $\bigl\{x \in \mathbb{R}:n^{-1}\sum_{i=1}^n \mathbbm{1}_{\{Z_i \leq x\}} \geq p\} = \bigl[Q_n^-(p),\infty)$ when $p \leq \narrowhat{p}_n$.  Thus
\begin{align*}
    \mathbb{P}(Q_n^-(p)\leq a) &= \mathbb{P}\bigl(Q_n^-(p)\leq a,\, p \leq \narrowhat{p}_n \bigr) + \mathbb{P}\bigl(Z_{\max} \leq a,\, p > \narrowhat{p}_n\bigr)\\
    &\leq \mathbb{P}\biggl(\frac{1}{n}\sum_{i=1}^n \indi_{\{Z_i\leq a\}} \geq p\biggr) + \mathbb{P}(Z_{\max} \leq a).
\end{align*}
Moreover,
\begin{align*}
    \mathbb{P}(Q_n^-(p)\leq a) \geq \mathbb{P}\bigl(Q_n^-(p)\leq a,\, p \leq \narrowhat{p}_n \bigr) = \mathbb{P}\biggl(\frac{1}{n}\sum_{i=1}^n \indi_{\{Z_i\leq a\}} \geq p\biggr),
\end{align*}
where the equality follows since $\frac{1}{n}\sum_{i=1}^n \indi_{\{Z_i\leq a\}} \geq p$ implies $\narrowhat{p}_n \geq p$.

Similarly, $\bigl\{x \in \mathbb{R} : \frac{1}{n} \sum_{i=1}^n \indi_{\{Z_i > x\}} \leq p\bigr\} = \bigl[Q_n^+(p), \infty\bigr)$ when $p<\narrowhat{p}_n$. Thus
\begin{align*}
    \mathbb{P}(Q_n^+(p)>a) &= \mathbb{P}\bigl(Q_n^+(p)>a ,\, p < \narrowhat{p}_n \bigr) + \mathbb{P}\bigl(Z_{\min}>a ,\, p \geq \narrowhat{p}_n \bigr)\\
    &\leq \mathbb{P}\biggl(\frac{1}{n} \sum_{i=1}^n \indi_{\{Z_i > a\}} > p\biggr) + \mathbb{P}(Z_{\min}>a),
\end{align*}
and
\begin{align*}
    \mathbb{P}(Q_n^+(p)>a) \geq \mathbb{P}\bigl(Q_n^+(p)>a ,\, p < \narrowhat{p}_n \bigr) = \mathbb{P}\biggl(\frac{1}{n} \sum_{i=1}^n \indi_{\{Z_i > a\}} > p\biggr),
\end{align*}
where the final equality follows since $\frac{1}{n} \sum_{i=1}^n \indi_{\{Z_i > a\}} > p$ implies $\narrowhat{p}_n > p$.
\end{proof}

\begin{Lemma} \label{lemma:mills-ratio}
    For all $x\geq 0$, we have
    \begin{align*}
        \frac{\phi(x)}{x+1/x} \leq \frac{2\phi(x)}{x+\sqrt{4+x^2}} \leq 1-\Phi(x) \leq \frac{2\phi(x)}{x+\sqrt{2+x^2}} \leq \frac{\phi(x)}{x}.
    \end{align*}
\end{Lemma}
\begin{proof}
    See, e.g., \citet[Eq.~(3)]{duembgen2010bounding}.
\end{proof}

\begin{Lemma} \label{lemma:gaussian-cdf-ratio}
    (a) If $t\geq 0$ and $b\in[0,1]$, then
    \begin{align*}
        \frac{1-\Phi(t-b)}{1-\Phi(t)} \geq 1 + \frac{b(t+1)}{2\sqrt{e}}.
    \end{align*}
    (b) If $b \in [0,t]$, then
    \begin{align*}
        \frac{1-\Phi(t-b)}{1-\Phi(t)} \geq \frac{1}{\sqrt{2}} e^{tb/2}.
    \end{align*}
\end{Lemma}

\begin{proof}
    \emph{(a)} For $x \geq 0$, define $f(x) \coloneqq \frac{1-\Phi(t-x)}{1-\Phi(t)}$. By Lemma~\ref{lemma:mills-ratio}, we have that for all $x\geq 0$ and $t\geq 0$,
    \begin{align*}
        f'(x) = \frac{\phi(t-x)}{1-\Phi(t)} = \frac{\phi(t)}{1-\Phi(t)} e^{tx-x^2/2} \geq \frac{t+\sqrt{2+t^2}}{2}e^{tx-x^2/2}.
    \end{align*}
    Thus, for $b \in [0,1]$, 
    \begin{align*}
        \frac{1-\Phi(t-b)}{1-\Phi(t)} = f(b) \geq f(0) + b\cdot \inf_{x\in[0,b]}f'(x) \geq 1 + \frac{b(t+\sqrt{2+t^2})}{2\sqrt{e}} \geq 1 + \frac{b(t+1)}{2\sqrt{e}},
    \end{align*}
    as required.

    \medskip
    \emph{(b)} By Lemma~\ref{lemma:mills-ratio}, for $b \in [0,t]$,
    \begin{align*}
        \frac{1-\Phi(t-b)}{1-\Phi(t)} &\geq \frac{2\phi(t-b)}{t-b+\sqrt{4+(t-b)^2}} \cdot \frac{t+\sqrt{2+t^2}}{2\phi(t)} \geq \frac{t+\sqrt{2+t^2}}{t+\sqrt{4+t^2}} \cdot e^{tb-b^2/2} \geq \frac{1}{\sqrt{2}} e^{tb/2},
    \end{align*}
    as required.
\end{proof}

\begin{Lemma} \label{lemma:max-min}
    Let $n\in\mathbb{N}$, $\alpha\in(0,1]$, $\epsilon \in \bigl[0,1-\frac{\log(8/\alpha)}{n}\bigr]$, and $Z_1,\ldots,Z_n \overset{\mathrm{iid}}{\sim} (1-\epsilon)N(\theta,\sigma^2)+\epsilon Q$ where $Q\in\mathsf{MNAR}_{N(\theta,\sigma^2)}$. Then
    \begin{align*}
        \mathbb{P}\Bigl(Z_{\max} > \theta + \sigma\sqrt{2\log(4n/\alpha)} \Bigr) \leq \frac{\alpha}{4} \quad\text{and}\quad \mathbb{P}\Bigl(Z_{\min} < \theta - \sigma\sqrt{2\log(4n/\alpha)} \Bigr) \leq \frac{\alpha}{4}.
    \end{align*}
\end{Lemma}
\begin{proof}
    We have
    \begin{align*}
        \mathbb{P}\Bigl(Z_{\max} > \theta &+ \sigma\sqrt{2\log(4n/\alpha)} \Bigr) \leq \mathbb{P}\Bigl(Z_i = \star \;\forall\; i\in[n]\Bigr) + \mathbb{P}\Bigl(\max_{i\in[n]} X_i > \theta + \sigma\sqrt{2\log(4n/\alpha)} \Bigr)\\
        &\leq \biggl(1-\frac{\log(8/\alpha)}{n}\biggr)^n + 1 - \Phi\bigl(\sqrt{2\log(4n/\alpha)}\bigr)^n \leq \frac{\alpha}{8} + 1 - \biggl(1-\frac{\alpha}{8n}\biggr)^n \leq \frac{\alpha}{4},
    \end{align*}
    where the second inequality follows since $\mathbb{P}(Z_1=\star) \leq \epsilon\leq 1-\frac{\log(8/\alpha)}{n}$; the third inequality follows since $\bigl(1-\frac{a}{n}\bigr)^n \leq e^{-a}$ for $a\in\mathbb{R}$, $n\in\mathbb{N}$ and since $1-\Phi(x) \leq \frac{1}{2}e^{-x^2/2}$ for $x \geq 0$; and the final inequality follows from the fact that $1-a \leq (1-a/n)^n$ for $a \in [0,1]$ and $n \geq 1$. The result for~$Z_{\min}$ follows by replacing $Z_i$ with $-Z_i$ for $i\in[n]$.
\end{proof}
The next lemma presents a slightly weaker conclusion in a more general setting.

\begin{Lemma} \label{lemma:max-min-2}
    Let $n\in\mathbb{N}$, $\alpha\in(0,1]$ and $\epsilon_1,\epsilon_2\in[0,1]$ be such that $\epsilon_1+\epsilon_2 \leq 1-\frac{\log(16/\alpha)}{n}$, and $Z_1,\ldots,Z_n \overset{\mathrm{iid}}{\sim} M \in \mathcal{M}\bigl(N(\theta,\sigma^2),\epsilon_1,\epsilon_2\bigr)$. Then
    \begin{align*}
        \mathbb{P}\Bigl(Z_{\max} > \theta + \sigma\sqrt{2\log\{12n(1-\epsilon_1)/\alpha\}} \Bigr) \leq \frac{\alpha}{4} \ \ \text{and} \ \ \mathbb{P}\Bigl(Z_{\min} < \theta - \sigma\sqrt{2\log\{12n(1-\epsilon_1)/\alpha\}} \Bigr) \leq \frac{\alpha}{4}.
    \end{align*}
\end{Lemma}
\begin{proof}
    We may write $M = (1-\epsilon_1-\epsilon_2)N(\theta,\sigma^2) + \epsilon_1\delta_{\{\star\}} + \epsilon_2 Q$, where $Q\in\mathsf{MNAR}_{N(\theta,\sigma^2)}$. Let $\mathcal{I}$ be the indices of $Z_i$ such that $Z_i$ is not drawn from $\delta_{\{\star\}}$. Then
    \begin{align*}
        \mathbb{P}\Bigl(&\Bigl\{Z_{\max} > \theta + \sigma\sqrt{2\log\{12n(1-\epsilon_1)/\alpha\}}\Bigr\} \bigcap \Bigl\{|\mathcal{I}| \leq 3n(1-\epsilon_1)\Bigr\} \Bigr)\\
        & \leq \mathbb{P}\bigl(Z_i = \star \;\forall\; i\in[n]\bigr) + \mathbb{P}\Bigl(\Bigl\{ \max_{i\in\mathcal{I}} X_i > \theta + \sigma\sqrt{2\log\{12n(1-\epsilon_1)/\alpha\}}\Bigr\} \bigcap \Bigl\{ |\mathcal{I}| \leq 3n(1-\epsilon_1) \Bigr\}\Bigr)\\
        &\leq \biggl(1-\frac{\log(16/\alpha)}{n}\biggr)^n + 1 - \Phi\bigl(\sqrt{2\log\{12n(1-\epsilon_1)/\alpha\}}\bigr)^{3n(1-\epsilon_1)}\\
        &\leq \frac{\alpha}{16} + 1 - \biggl(1-\frac{\alpha}{24n(1-\epsilon_1)}\biggr)^{3n(1-\epsilon_1)} \leq \frac{3\alpha}{16},
    \end{align*}
    where the second inequality follows since $\mathbb{P}(Z_1=\star) \leq 1 - \epsilon_1 - \epsilon_2 \leq 1-\frac{\log(16/\alpha)}{n}$.   
    Moreover, by a multiplicative Chernoff bound \citep[Theorem~2.3(c)]{McDiarmid1998}, we have
    \begin{align*}
        \mathbb{P} \Bigl(|\mathcal{I}| > 3n(1-\epsilon_1) \Bigr) \leq \frac{\alpha}{16}.
    \end{align*}
    Thus, 
    \begin{align*}
        &\mathbb{P}\Bigl(Z_{\max} > \theta + \sigma\sqrt{2\log\{12n(1-\epsilon_1)/\alpha\}} \Bigr)\\
        &\leq \mathbb{P}\Bigl(\Bigl\{Z_{\max} > \theta + \sigma\sqrt{2\log\{12n(1-\epsilon_1)/\alpha\}}\Bigr\} \bigcap \Bigl\{|\mathcal{I}| \leq 3n(1-\epsilon_1)\Bigr\} \Bigr) + \mathbb{P} \Bigl(|\mathcal{I}| > 3n(1-\epsilon_1) \Bigr) \leq \frac{\alpha}{4}.
    \end{align*}
    The result for $Z_{\min}$ follows by replacing $Z_i$ with $-Z_i$ for $i\in[n]$.
\end{proof}

\begin{Lemma} \label{lemma:dkw-missing}
    Let $n\in\mathbb{N}$, $\alpha\in(0,1]$, $\epsilon_1,\epsilon_2\in[0,1]$ be such that $\epsilon_1+\epsilon_2 \leq 1-\frac{\log(4/\alpha)}{n}$, and $Z_1,\ldots,Z_n \overset{\mathrm{iid}}{\sim} M \in \mathcal{M}\bigl(P,\epsilon_1,\epsilon_2\bigr)$. Let $\hat{M}_n$ denote the empirical distribution of $Z_1,\ldots,Z_n$. Then 
    \begin{align*}
        \mathbb{P}\biggl(\sup_{t\in\mathbb{R}} \Bigl|\hat{M}_n\bigl((-\infty,t]\bigr) - M\bigl((-\infty,t]\bigr)\Bigr| < 3\sqrt{\frac{(1-\epsilon_1)\log(4/\alpha)}{n}}\biggr) \geq 1-\alpha
    \end{align*}
    and
    \begin{align*}
        \mathbb{P}\biggl(\sup_{t\in\mathbb{R}} \Bigl|\hat{M}_n\bigl((t,\infty)\bigr) - M\bigl((t,\infty)\bigr)\Bigr| < 3\sqrt{\frac{(1-\epsilon_1)\log(4/\alpha)}{n}}\biggr) \geq 1-\alpha.
    \end{align*}
\end{Lemma}
\begin{proof}
    The proof follows from the proof of~\citet[Eq.~(S31)]{ma2024estimation}, with $\epsilon=\epsilon_2$ and $q=\frac{1-\epsilon_1-\epsilon_2}{1-\epsilon_2}$ therein.
\end{proof}

\end{document}